\documentclass[10pt,a4paper]{amsart}
\usepackage[foot]{amsaddr}
\usepackage[margin=0.9in]{geometry}
\usepackage{graphicx}
\usepackage[T1]{fontenc}
\usepackage[utf8]{inputenc}
\usepackage{mathtools}
\usepackage{amsthm}
\usepackage{lmodern}
\usepackage[english]{babel}
\usepackage{amsmath,amssymb,amsfonts}
\usepackage{calrsfs,calligra,mathrsfs}
\usepackage{enumitem}
\usepackage{tikz-cd}
\usepackage[mathscr]{euscript}
\usepackage{hyperref}
\usepackage[noabbrev,nameinlink]{cleveref}

\theoremstyle{plain}
\newtheorem{thm}{Theorem}[section]
\newtheorem*{thm*}{Theorem}
\newtheorem{lem}[thm]{Lemma}

\newtheorem{prop}[thm]{Proposition}

\theoremstyle{definition}
\newtheorem{defn}[thm]{Definition}
\newtheorem{nota}[thm]{Notation}
\newtheorem{ex}[thm]{Example}

\theoremstyle{remark}
\newtheorem{rem}[thm]{Remark}
\Crefname{thm}{Theorem}{Theorems}
\Crefname{lem}{Lemma}{Lemmata}
\Crefname{cor}{Corollary}{Corollaries}
\Crefname{prop}{Proposition}{Propositions}

\newcommand*{\isoarrow}[1]{\arrow[#1,"\rotatebox{90}{\(\sim\)}"]}
\newcommand{\Acal}{\mathcal{A}}
\newcommand{\Ccal}{\mathcal{C}}
\newcommand{\Hcal}{\mathcal{H}}

\newcommand{\Scal}{\mathcal{S}}
\newcommand{\Hbb}{\mathbb{H}}
\newcommand{\conn}{\mathsf{conn}}
\newcommand{\id}{\mathsf{id}}
\newcommand{\unit}{\mathbf{I}}

\newcommand{\Perv}{\mathsf{Perv}}

\newcommand\blfootnote[1]{%
	\begingroup
	\renewcommand\thefootnote{}\footnote{#1}%
	\addtocounter{footnote}{-1}%
	\endgroup
}

\title{Monoidal structures on triangulated fibered categories}
\author{Luca Terenzi}
\address{Fakult{\"a}t f{\"u}r Mathematik, Universit{\"a}t Regensburg}
\email{\normalfont{Luca.Terenzi@mathematik.uni-regensburg.de}}

\makeatletter
\hypersetup{
	pdfauthor={\author},
	pdftitle={\@title},
	colorlinks,
	linkcolor=[rgb]{0.2,0.2,0.6},
	citecolor=[rgb]{0.2,0.6,0.2},
	urlcolor=[rgb]{0.6,0.2,0.2}}
\makeatother

\calclayout
\begin{document}

\maketitle

\begin{abstract}
	Let $\Scal$ be a small category admitting binary products.
	There exists a canonical dictionary expressing the structure and properties of monoidal $\Scal$-fibered categories in terms of the external tensor product in place of the usual internal tensor product.
	In this paper, we specialize this dictionary to the setting of triangulated $\Scal$-fibered categories, by studying how the general translation process interacts with shift functors.
	Our results provide a satisfactory abstract framework to study monoidal structures from the perspective of perverse sheaves. 
\end{abstract}

\tableofcontents

\blfootnote{\textit{\subjclassname}. 18D30, 18M05.}
\blfootnote{\textit{\keywordsname}. Monoidal fibered categories, external tensor product, perverse sheaves.}

\section*{Introduction}

\subsection*{Motivation and goal of the paper}

Given a small category $\Scal$, there is a well-established theory of \textit{monoidal $\Scal$-fibered categories}.
If one regards $\Scal$-fibered categories as $\Scal$-indexed categories (that is, contravariant pseudo-functors on $\Scal$ with values in the $2$-category of categories), the basic definitions of this theory are listed, for example, in \cite[\S~2.3.1]{Ayo07a}.
In a nutshell, giving a monoidal structure on an $\Scal$-fibered category $\Hbb$ amounts to giving, for each $S \in \Scal$, a tensor product functor
\begin{equation}\label{eq:intro-otimes}
	- \otimes_S - \colon \Hbb(S) \times \Hbb(S) \to \Hbb(S),
\end{equation}
together with, for every morphism $f \colon T \to S$ in $\Scal$, a monoidality isomorphism
\begin{equation}\label{eq:intro-mono}
	m_f \colon f^* A \otimes_T f^* B \xrightarrow{\sim} f^* (A \otimes_S B),
\end{equation}
subject to natural coherence conditions with respect to identities and composition in $\Scal$.
The single tensor product functors are typically endowed with associativity constraints (and possibly also with commutativity and unit constraints), subject to further coherence conditions with respect to inverse image functors.
When the base category $\Scal$ admits binary products, one can also define the \textit{external tensor product} functors
\begin{equation*}
	- \boxtimes_{S_1,S_2} - \colon \Hbb(S_1) \times \Hbb(S_2) \to \Hbb(S_1 \times S_2)
\end{equation*}
by the formula
\begin{equation}\label{eq:intro-int_to_ext}
	A_1 \boxtimes_{S_1,S_2} A_2 := \pi^{12,*}_1 A_1 \otimes_{S_1 \times S_2} \pi^{12,*}_2 A_2,
\end{equation}
where $\pi^{12}_i \colon S_1 \times S_2 \to S_i$, $i = 1,2$, denote the canonical projections.
Informally, the usual tensor product -- also called the \textit{internal tensor product} -- can be recovered from the external one by the formula
\begin{equation}\label{eq:intro-ext_to_int}
	A \otimes_S B = \Delta_S^* (A \boxtimes_{S,S} B),
\end{equation}
where $\Delta_S \colon S \to S \times S$ denotes the diagonal embedding.
More precisely, the external tensor product functors assemble into a monoidal structure on the Grothendieck construction of $\Hbb$, compatibly with the Cartesian monoidal structure on $\Scal$.
By Shulman's result \cite[Thm.~12.8]{Shu08}, expanded by Moeller--Vasilakopoulou in \cite[Thm.~3.10]{MV20}, the formulae \eqref{eq:intro-int_to_ext} and \eqref{eq:intro-ext_to_int} induce an equivalence between the $2$-categories of (strong) monoidal $\Scal$-indexed categories and of (strong) monoidal fibrations over $\Scal$.
Concretely, this means that there is a canonical dictionary expressing the usual structures and properties of the internal tensor product as structures and properties of the external tensor product:
this applies to associativity, commutativity, and unit constraints, to projection formulae, as well as to monoidal structures on morphisms of $\Scal$-fibered categories.
As an example, Mac Lane's classical pentagon and hexagon axioms can be stated purely in terms of the external tensor product.
The main goal of the present paper is to specialize this general dictionary to the setting of triangulated $\Scal$-fibered categories.

Our main motivation comes from the theory of six-functor formalisms in Algebraic or Analytic Geometry, as developed in Ayoub's thesis \cite{Ayo07a,Ayo07b} and in Cisinski--Déglise's book \cite{CisDeg}.
In this setting, one is typically interested in the functoriality of triangulated coefficient systems over a base category of geometric objects.
Our specific goal is to understand monoidal structures from the perspective of perverse sheaves.
To fix notation, suppose that $\Scal$ is a suitable category of complex algebraic varieties and $\Hbb(S) = D^b_c(S)$ is the $\mathbb{Q}$-linear constructible derived category;
these naturally assemble into a monoidal $\Scal$-fibered category.
The abelian category $\Perv(S)$ of \textit{perverse sheaves} on $S$ is defined in \cite[\S~2]{BBD82} as the heart of the so-called perverse $t$-structure on $D^b_c(S)$.
By Beilinson's celebrated result \cite[Thm.~1.3]{Bei87}, the triangulated category $D^b_c(S)$ is canonically equivalent to the bounded derived category $D^b(\Perv(S))$.
Therefore, it is natural to wonder how much of the functoriality of the monoidal fibered category $S \mapsto D^b_c(S)$ can be recovered on the level of perverse sheaves.
One major technical difficulty is that the internal tensor product on $D^b_c(S)$ is not $t$-exact as soon as $\dim(S) \geq 1$.
However, by \cite[Prop.~4.2.8]{BBD82}, the external tensor product 
\begin{equation*}
	- \boxtimes_{S_1,S_2} - \colon D^b_c(S_1) \times D^b_c(S_2) \to D^b_c(S_1 \times S_2)
\end{equation*}
is always $t$-exact in each variable.
Even better, one can show that it agrees with the derived functor of its restriction
\begin{equation*}
	- \boxtimes_{S_1,S_2} - \colon \Perv(S_1) \times \Perv(S_2) \to \Perv(S_1 \times S_2)
\end{equation*}
modulo Beilinson's equivalences.
For this reason, it is natural to use the external tensor product when studying the monoidal structure of the constructible derived categories from the perspective of perverse sheaves.
To this end, however, one needs to know in advance that the external tensor product encodes the structure of $S \mapsto D^b_c(S)$ as a triangulated monoidal (as opposed to just monoidal) fibered category.
In the present paper, we show that this is indeed the case.

Since our results hold for arbitrary triangulated monoidal fibered categories, they can be applied to the study of refined coefficient systems modelled on perverse sheaves, such as mixed Hodge modules (see \cite{Sai90}) and perverse Nori motives (see \cite{IM24}).
In fact, this is the first of a series of three papers devoted to understanding the monoidal aspects of the constructible derived categories from the perspective of perverse sheaves;
the two companion papers \cite{Ter24Fact} and \cite{Ter24Emb} address more specific issues concerning monoidality isomorphisms and unit constraints. 
This foundational work is used heavily in the construction of the monoidal structure for perverse Nori motives, obtained by the author in \cite{Ter24Nori}.

\subsection*{Main results}

Even if we are primarily interested in applications to six-functor formalisms, we formulate our main results in the general setting of monoidal fibered categories over a base category $\Scal$ admitting binary products.
However, contrarily to \cite{Shu08} and \cite{MV20}, we always regard $\Scal$-fibered categories as $\Scal$-indexed categories rather than as Grothendieck fibrations over $\Scal$, and we always treat both internal and external monoidal structures fiber-wise:
this makes it simpler to apply our results to the companion papers \cite{Ter24Fact} and \cite{Ter24Emb}, for which the language of fibrations does not seem well-suited.

In the language of $\Scal$-indexed categories, one can study the internal and the external tensor product on a given $\Scal$-fibered category $\Hbb$ in parallel.
Let us generically call \textit{internal monoidal structure} the datum of internal tensor product functors \eqref{eq:intro-otimes}, monoidality isomorphisms \eqref{eq:intro-mono} and (depending on the case) associativity, commutativity, and unit constraints.
Each piece of the structure admits a natural counterpart in terms of the external tensor product, and they can be put together into the notion of \textit{external monoidal structure}.
The equivalence of $2$-categories constructed in \cite{Shu08} and \cite{MV20} can be interpreted concretely as a bijection between equivalence classes of internal and external monoidal structures on $\Hbb$, compatibly with respect to morphisms of $\Scal$-fibered categories.
 
We establish a version of this correspondence in the setting of triangulated fibered categories.
Our first main result can be stated as follows:
\begin{thm*}[\Cref{thm:tri-int_ext}]
	Let $\Hbb$ be a triangulated $\Scal$-fibered category.
	There exists a canonical bijection between equivalence classes of triangulated internal and external monoidal structures on $\Hbb$, respecting associativity, commutativity, and unit constraints, and compatible with projection formulae.
\end{thm*}
In a triangulated monoidal structure, the tensor product functors are triangulated in each variable:
this involves the datum of additional natural isomorphisms witnessing their compatibility with the shift functor in each variable, subject to a suitable coherence condition.
The monoidality isomorphisms are required to respect shift functors in a precise sense, and similarly for associativity, commutativity, and unit constraints.
In order to prove our theorem with the least amount of effort, we elaborate on the known correspondence between plain internal and external monoidal structures.
We observe that the correspondence is mediated by explicit constructions, which we describe in detail.
Then we show that all these constructions admit natural specializations to the triangulated setting:
this allows us to isolate the part of the monoidal structure involving shift functors, and to analyze it separately. 
In this way, the proof of the theorem reduces to checking the commutativity of certain natural diagrams involving shift functors, which is easy to achieve.

In order to complete the picture, it remains to show that the canonical bijection of our first theorem is functorial with respect to morphisms of triangulated fibered categories.
This is the content of our second main result, which can be stated as follows:
\begin{thm*}[\Cref{thm:tri-int_ext-mor}]
	Let $\Hbb_1$ and $\Hbb_2$ be triangulated monoidal $\Scal$-fibered categories, and let $R \colon \Hbb_1 \to \Hbb_2$ be a morphism of triangulated $\Scal$-fibered categories.
	There exists a canonical bijection between triangulated internal and external monoidal structures on $R$ (with respect to the monoidal structures on $\Hbb_1$ and $\Hbb_2$), respecting associativity, commutativity, and unit constraints.
\end{thm*}

Following the same strategy as before, the proof reduces to checking the commutativity of certain natural diagrams involving shift functors. 
We believe that one merit of our paper is to uncover the elementary computational nature of the proof method.
The interested reader can easily adapt it to obtain further variants of the correspondence between internal and external monoidal structures.

\subsection*{Related work}

The current literature counts various works devoted to the foundations of the theory of monoidal fibered categories, at various levels of abstraction.
The aforementioned papers \cite{Shu08} and \cite{MV20} investigate the monoidal aspects of the Grothendieck construction in the $2$-categorical setting, employing the language of Grothendieck fibrations.
This synthetic viewpoint makes it easy to understand why a given property of the internal tensor product automatically implies a similar property of the external tensor product and conversely:
since all such properties amount to the commutativity of suitable natural diagrams, everything comes down to the uniqueness of factorization through Cartesian arrows (see the proof of \cite[Thm.~12.8]{Shu08} for the case of Mac Lane's pentagon axiom).

In fact, it is possible to make sense of associativity, symmetry, and unitarity of monoidal fibered categories more conceptually, by interpreting monoidal fibrations over $\Scal$ as plain fibrations over a suitable category of correspondences canonically built out of $\Scal$.
This idea is developed rigorously in H{\"o}rmann's work \cite{Hoer17}, designed for applications to six-functor formalisms.
In a similar vein, monoidal structures are formalized in terms of operads in Ramzi's recent work \cite{Ram22}, which describes an $\infty$-categorical version of the monoidal Grothendieck construction.

The present paper differs in spirit from all these works:
being interested in the concrete construction of six-functor formalisms endowed with perverse $t$-structures, we need to study monoidal fibered categories as analytically as possible.
This is why we prefer working with explicit families of axioms for the various constraints involved instead of encoding everything into a synthetic conceptual structure.
On a related note, we find it useful to treat internal and external monoidal structures on an equal footing, as parameterized families of bi-functors relating the fibres of a fibered category in a suitable way.
For the reader's convenience, in the main body of the paper we briefly comment on how our notion of external monoidal structure reads in the language of Grothendieck fibrations.  

\subsection*{Structure of the paper}

Our main results are proved in the last three sections of the paper, based on the preliminary results collected in the first three sections.
In \Cref{sect:int_ext} we recall how the internal and the external tensor product can be obtained from one another, compatibly with inverse image functors.
In \Cref{sect:acu} we review how associativity, commutativity, and unit constraints are translated between the internal and the external formulation.
In \Cref{sect:mor} we similarly review monoidal structures on morphisms of fibered categories.
In \Cref{sect:tri-int_ext} we specialize the correspondence between internal and external tensor product to the triangulated setting.
In \Cref{sect:tri-acu} we do the same with respect to associativity, commutativity, and unit constraints.
In \Cref{sect:tri-mor} we study monoidal structures on triangulated morphisms.

\subsection*{Acknowledgments}

The present paper originated from the first chapter of my Ph.D. thesis, written at the University of Freiburg under the supervision of Annette Huber-Klawitter.
It is a pleasure to thank her for many useful discussions around the subject of this article, as well as for her constant support and encouragement.
I would also like to thank Frédéric Déglise for pointing me to the works \cite{Shu08}, \cite{MV20}, and \cite{Ram22}, and Swann Tubach for helpful comments on an earlier version of this paper.

\section*{Notation and conventions}

\begin{itemize}
	\item Unless otherwise dictated, categories are assumed to be small with respect to some fixed universe.
	\item Given a category $\Ccal$, the notation $A \in \Ccal$ means that $A$ is an object of $\Ccal$.
	\item Given an adjoint pair of functors $F \colon \Ccal \rightleftarrows \Ccal' \colon G$, we let $\eta \colon \id_{\Ccal} \to G \circ F$ and $\epsilon \colon F \circ G \to \id_{\Ccal'}$ denote the associated unit and co-unit natural transformations, respectively.
	\item Given categories $\Ccal_1, \dots, \Ccal_n$, we let $\Ccal_1 \times \dots \times \Ccal_n$ denote the direct product category.
	\item Let $\Scal$ be a category admitting binary products.
	\begin{itemize}
		\item For every object $S \in \Ccal$, and every $n \geq 2$, we let $\Delta_S^{(n)} \colon S \to S \times \dots \times S$ denote the diagonal embedding of $S$ into the product of $n$ copies of itself.
		Example: $\Delta_S^{(3)} \colon S \to S \times S \times S$.
		\item When $n = 2$, we write $\Delta_S^{(2)}$ simply as $\Delta_S$.
		\item Fix objects $S_1, \dots, S_n \in \Scal$, and let $I = i_1 \dots i_r$ be a string in the alphabet $\left\{1, \dots, n\right\}$.
		\begin{itemize}
			\item We consider the corresponding product $S_{i_1} \times \dots \times S_{i_r} \in \Scal$, ordering the factors as prescribed by $I$.
			Example: $S_1 \times S_3 \times S_2$ for $I = 132$.
			\item Given a substring $J = j_1 \dots j_s$ of $I$, we let
			\begin{equation*}
				\pi^I_J \colon S_{i_1} \times \dots \times S_{i_r} \to S_{j_1} \times \dots \times S_{j_s}
			\end{equation*} 
			denote the corresponding projection.
			Example: $\pi^{132}_{12} \colon S_1 \times S_3 \times S_2 \to S_1 \times S_2$.
			\item Given a permutation $\sigma \in \mathsf{Sym}(i_1,\dots,i_r)$, written as a product of cyclic permutations in the usual way, we let
			\begin{equation*}
				\tau_{\sigma} \colon S_{i_1} \times \dots \times S_{i_r} \xrightarrow{\sim} S_{\sigma(i_1)} \times \dots \times S_{\sigma(i_r)}
			\end{equation*}
			denote the corresponding isomorphism.
			Example: $\tau_{(123)} \colon S_1 \times S_2 \times S_3 \xrightarrow{\sim} S_2 \times S_3 \times S_1$.
			\item When $r = 2$ and $\sigma = (12)$, we write $\tau_{\sigma}$ simply as $\tau$.
		\end{itemize} 
	\end{itemize}
\end{itemize}

\section{Translating between internal and external tensor product}\label{sect:int_ext}

Throughout the whole paper, we work over a fixed base category $\Scal$ admitting binary products;
it is not relevant for our purposes whether it admits more general finite limits (such as a terminal object).

In the first three sections, as a preparation for our main results, we elaborate on the known correspondence between internal and external monoidal structures on general $\Scal$-fibered categories.
In particular, we reformulate in terms of pseudo-functors the main results of \cite{Shu08} and \cite{MV20}, which are stated in terms of fibrations.
The first part of the correspondence can be conveniently summarized as follows:
\begin{thm}[{\cite[Thm.~12.8, Prop.~13.30]{Shu08}}; {\cite[Thm.~3.10]{MV20}}]\label[thm]{thm:int_ext}
	Let $\Hbb$ be an $\Scal$-fibered category.
	There exists a canonical bijection between equivalence classes of internal and external monoidal structures on $\Hbb$, respecting associativity, commutativity, and unit constraints, and compatible with projection formulae.
\end{thm}
The second part of the correspondence, concerning monoidal morphisms between monoidal $\Scal$-fibered categories, is addressed in \Cref{sect:mor}.
In the present section, after fixing our general conventions on fibered categories, we review the basic translation process between internal and external tensor product;
this involves introducing the appropriate notions of equivalence appearing in the statement of the theorem above.
In the final part, we explain how projection formulae translate between the internal and the external setting.
A detailed discussion of associativity, commutativity, and unit constraints is deferred to \Cref{sect:acu}.

\subsection{Recollections on fibered categories}\label{subsect:rec_fib}

Throughout the whole paper, when speaking of $\Scal$-fibered categories we always mean contravariant $2$-functors from $\Scal$ to the $2$-category of categories:
what is called "catégorie clivée normalisée" in the original reference \cite[Exp.~VI, Defn.~7.1]{SGA1}.
Our conventions on $2$-categories and $2$-functors follow those of \cite[\S~1.4.1]{Ayo07a} and \cite[\S~1.1]{CisDeg}:
we consider $2$-categories in the strict sense and $2$-functors in the lax sense.
In detail:

\begin{defn}\label[defn]{defn:S-fib}
	An \textit{$\Scal$-fibered category} $\Hbb$ is the datum of
	\begin{itemize}
		\item for every $S \in \Scal$, a category $\Hbb(S)$,
		\item for every morphism $f \colon T \rightarrow S$ in $\Scal$, a functor
		\begin{equation*}
			f^* \colon \Hbb(S) \rightarrow \Hbb(T),
		\end{equation*}  
		called the \textit{inverse image functor} along $f$,
		\item for every pair of composable morphisms $f \colon T \rightarrow S$, $g \colon S \rightarrow V$ in $\Scal$, a natural isomorphism of functors $\Hbb(V) \rightarrow \Hbb(T)$
		\begin{equation*}
			\conn = \conn_{f,g} \colon (gf)^* A \xrightarrow{\sim} f^* g^* A
		\end{equation*}
		called the \textit{connection isomorphism} at $(f,g)$
	\end{itemize}
	satisfying the following conditions:
	\begin{enumerate}
		\item[(fib-0)] For every $S \in \Scal$, we have $\id_S^* = \id_{\Hbb(S)}$.
		\item[(fib-1)] For every triple of composable morphisms $f \colon T \rightarrow S$, $g \colon S \rightarrow V$, $h \colon V \rightarrow W$ in $\Scal$, the diagram of functors $\Hbb(W) \rightarrow \Hbb(T)$
		\begin{equation*}
			\begin{tikzcd}
				(hgf)^* A \arrow{rr}{\conn_{f,hg}} \arrow{d}{\conn_{gf,h}} &&
				f^* (hg)^* \arrow{d}{\conn_{g,h}} A \\
				(gf)^* h^* A \arrow{rr}{\conn_{f,g}} &&
				f^* g^* h^* A
			\end{tikzcd}
		\end{equation*}
		is commutative.
	\end{enumerate}
\end{defn}

\begin{rem}\label[rem]{rem:fib}
	\begin{enumerate}
		\item The strict unitality expressed by axiom (fib-0) is not a severe requirement (for example, see \cite[Lem.~2.5]{DelVoe}).
		In any case, this axiom does not play an important role in the following.
		\item The significance of axiom (fib-1) is to ensure that compositions of inverse image functors cannot get identified in more than one way:
		given a morphism $f \colon T \to S$ in $\Scal$ and two factorizations
		\begin{equation*}
			f = g_r \circ \dots \circ g_1, \quad f = h_s \circ \dots \circ h_1,
		\end{equation*}
		all possible natural isomorphisms of functors $\Hbb(S) \to \Hbb(T)$
		\begin{equation*}
			g_1^* \cdots g_r^* A \xrightarrow{\sim} h_1^* \cdots h_s^* A
		\end{equation*}
		obtained by concatenating connection isomorphisms (and inverses thereof) coincide.
		In practice, this allows us to treat all connection isomorphisms as equalities.
	\end{enumerate}
\end{rem}

\begin{defn}\label[defn]{defn:sect}
	Let $\Hbb$ be an $\Scal$-fibered category.
	\begin{enumerate}
		\item A \textit{section} $A$ of $\Hbb$ over $\Scal$ is the datum of
		\begin{itemize}
			\item for every $S \in \Scal$, an object $A_S \in \Hbb(S)$,
			\item for every morphism $f \colon T \rightarrow S$ in $\Scal$, an isomorphism in $\Hbb(T)$
			\begin{equation*}
				A^* = A_f^* \colon f^* A_S \xrightarrow{\sim} A_T
			\end{equation*}
		\end{itemize}
		satisfying the following condition:
		\begin{enumerate}
			\item[(sect)] For every pair of composable morphisms $f \colon T \rightarrow S$, $g \colon S \rightarrow V$ in $\Scal$, the diagram in $\Hbb(T)$
			\begin{equation*}
				\begin{tikzcd}
					f^* g^* A_V \arrow{r}{A_g^*} &
					f^* A_S \arrow{d}{A_f^*} \\
					(gf)^* A_V \arrow{r}{A_{gf}^*} \arrow{u}{\conn_{f,g}} &
					A_T \\
				\end{tikzcd}
			\end{equation*}
			is commutative.
		\end{enumerate}
		\item Let $A$ and $B$ be two sections of $\Hbb$.
		A \textit{morphism of sections} $w \colon A \rightarrow B$ is the datum of
		\begin{itemize}
			\item for every $S \in \Scal$, a morphism
			\begin{equation*}
				w_S \colon A_S \rightarrow B_S
			\end{equation*}
		\end{itemize}
		satisfying the following condition:
		\begin{enumerate}
			\item[(mor-sect)] For every morphism $f \colon T \rightarrow S$ in $\Scal$, the diagram in $\Hbb(T)$
			\begin{equation*}
				\begin{tikzcd}
					f^* A_S \arrow{r}{w_S} \arrow{d}{A^*_f} &
					f^* B_S \arrow{d}{B^*_f} \\
					A_T \arrow{r}{w_T} &
					B_T
				\end{tikzcd}
			\end{equation*}
			is commutative.
		\end{enumerate}
	\end{enumerate}
\end{defn}

\begin{rem}\label[rem]{rem:iso-sect}
	There are natural notions of identity and of composition for morphisms between sections of an $\Scal$-fibered category $\Hbb$.
	Clearly, a morphism of sections $w \colon A \rightarrow B$ is invertible with respect to this composition if and only if the morphism $w_S \colon A_S \rightarrow B_S$ is an isomorphism in $\Hbb(S)$ for every $S \in \Scal$:
	this yields the notion of \textit{isomorphism of sections}.
\end{rem}

\subsection{Internal and external tensor structures}\label{subsect:tens}

Until the end of \Cref{sect:acu}, we work over a fixed $\Scal$-fibered category $\Hbb$ in the sense of \Cref{defn:S-fib}.

\begin{nota}\label[nota]{nota:fib-when-tens}
	For sake of simplicity, from now on we systematically write connection isomorphisms for inverse image functors of $\Hbb$ as unmarked equalities:
	this is harmless, as explained in \Cref{rem:fib}(2).
\end{nota}

Turning $\Hbb$ into a monoidal $\Scal$-fibered category amounts to defining a monoidal structure on $\Hbb(S)$ for each $S \in \Scal$, compatibly with inverse image functors in a precise way.
Typically, the tensor product functors
\begin{equation*}
	- \otimes_S - \colon \Hbb(S) \times \Hbb(S) \to \Hbb(S)
\end{equation*}
are required to carry at least an associativity constraint (and possibly also commutativity and unit constraints), and these have to be compatible with inverse image functors as well.
For the moment, we prefer to ignore these additional constraints, and we consider the following minimal notion:
\begin{defn}\label[defn]{defn:ITS}
	\begin{enumerate}
		\item An \textit{internal tensor structure} $(\otimes,m)$ on $\Hbb$ is the datum of
		\begin{itemize}
			\item for every $S \in \Scal$, a functor
			\begin{equation*}
				-\otimes- = -\otimes_S- \colon \Hbb(S) \times \Hbb(S) \to \Hbb(S),
			\end{equation*}
			called the \textit{internal tensor product functor} over $S$,
			\item for every morphism $f \colon T \to S$ in $\Scal$, a natural isomorphism of functors $\Hbb(S) \times \Hbb(S) \to \Hbb(T)$
			\begin{equation*}
				m = m_f \colon f^* A \otimes f^* B \xrightarrow{\sim} f^* (A \otimes B),
			\end{equation*}
			called the \textit{internal monoidality isomorphism} along $f$
		\end{itemize}
		satisfying the following condition:
		\begin{enumerate}
			\item[\hypertarget{mITS}{($m$ITS)}] For every pair of composable morphisms $f \colon T \to S$, $g \colon S \to V$ in $\Scal$, the diagram of functors $\Hbb(V) \times \Hbb(V) \to \Hbb(T)$
			\begin{equation*}
				\begin{tikzcd}
					(gf)^* A \otimes (gf)^* B \arrow{rr}{m} \arrow[equal]{d} && 
					(gf)^*(A \otimes B)  \arrow[equal]{d} \\
					f^* g^* A \otimes f^* g^* B \arrow{r}{m} &
					f^* (g^* A \otimes g^* B) \arrow{r}{m} &
					f^* g^* (A \otimes B)
				\end{tikzcd}
			\end{equation*}
			is commutative.
		\end{enumerate}
		\item Let $(\otimes,m)$ and $(\otimes',m')$ be two internal tensor structures on $\Hbb$. 
		An \textit{equivalence of internal tensor structures} $e \colon (\otimes,m) \xrightarrow{\sim} (\otimes',m')$ is the datum of
		\begin{itemize}
			\item for every $S \in \Scal$, a natural isomorphism of functors $\Hbb(S) \times \Hbb(S) \rightarrow \Hbb(S)$
			\begin{equation*}
				e = e_S \colon A \otimes B \xrightarrow{\sim} A \otimes' B
			\end{equation*}
		\end{itemize}
		satisfying the following condition:
		\begin{enumerate}
			\item[(eq-ITS)] For every morphism $f \colon T \to S$ in $\Scal$, the diagram of functors $\Hbb(S) \times \Hbb(S) \to \Hbb(T)$
			\begin{equation*}
				\begin{tikzcd}
					f^* A \otimes f^* B \arrow{r}{m} \arrow{d}{e} &
					f^* (A \otimes B)  \arrow{d}{e} \\
					f^* A \otimes' f^* B \arrow{r}{m'} &
					f^* (A \otimes' B)  
				\end{tikzcd}
			\end{equation*}
			is commutative.
		\end{enumerate}
	\end{enumerate}
\end{defn}

\begin{rem}\label[rem]{rem:ITS-id}
	Fix $S \in \Scal$. 
	In view of axiom (fib-0) from \Cref{defn:S-fib}, axiom ($m$ITS) in the case where $f = g = \id_S$ asserts that the natural isomorphism of functors $\Hbb(S) \times \Hbb(S) \to \Hbb(S)$
	\begin{equation*}
		m_{\id_S} \colon A \otimes B \xrightarrow{\sim} A \otimes B
	\end{equation*}
	is idempotent.
	Hence it must be the identity.
\end{rem}

Recall that we assume the base category $\Scal$ to admit binary products.
This allows one to define the external tensor product starting from the internal tensor product and conversely, via the classical formulae \eqref{eq:intro-int_to_ext} and \eqref{eq:intro-ext_to_int}.
In order to formalize this process, it is natural to consider the following notion:
\begin{defn}\label[defn]{defn:ETS}
	\begin{enumerate}
		\item An \textit{external tensor structure} $(\boxtimes,m)$ on $\Hbb$ is the datum of
		\begin{itemize}
			\item for every $S_1, S_2 \in \Scal$, a functor
			\begin{equation*}
				-\boxtimes- = - \boxtimes_{S_1,S_2} - \colon \Hbb(S_1) \times \Hbb(S_2) \to \Hbb(S_1 \times S_2),
			\end{equation*}
			called the \textit{external tensor product functor} over $(S_1,S_2)$,
			\item for every choice of morphisms $f_i \colon T_i \rightarrow S_i$ in $\Scal$, $i = 1,2$, a natural isomorphism of functors $\Hbb(S_1) \times \Hbb(S_2) \to \Hbb(T_1 \times T_2)$
			\begin{equation*}
				m = m_{f_1,f_2} \colon f_1^* A_1 \boxtimes f_2^* A_2 \xrightarrow{\sim} (f_1 \times f_2)^*(A_1 \boxtimes A_2),
			\end{equation*}
			called the \textit{external monoidality isomorphism} along $(f_1,f_2)$
		\end{itemize}
		satisfying the following condition:
		\begin{enumerate}
			\item[($m$ETS)] For every choice of composable morphisms $f_i \colon T_i \to S_i$, $g_i \colon S_i \to V_i$ in $\Scal$, $i = 1,2$, the diagram of functors $\Hbb(V_1) \times \Hbb(V_2) \to \Hbb(T_1 \times T_2)$
			\begin{equation*}
				\begin{tikzcd}[font=\small]
					(g_1 f_1)^* A_1 \boxtimes (g_2 f_2)^* A_2  \arrow{rr}{m} \arrow[equal]{d} &&
					(g_1 f_1 \times g_2 f_2)^* (A_1 \boxtimes A_2)  \arrow[equal]{d} \\
					f_1^* g_1^* A_1 \boxtimes f_2^* g_2^* A_2  \arrow{r}{m} &
					(f_1 \times f_2)^* (g_1^* A_1 \boxtimes g_2^* A_2) \arrow{r}{m} &
					(f_1 \times f_2)^* (g_1 \times g_2)^* (A_1 \boxtimes A_2)
				\end{tikzcd}
			\end{equation*}	
			is commutative.
		\end{enumerate}
		\item Let $(\boxtimes',m')$ and $(\boxtimes,m)$ be two external tensor structures on $\Hbb$.
		An \textit{equivalence of external tensor structures} $e \colon (\boxtimes',m') \xrightarrow{\sim} (\boxtimes,m)$ is the datum of
		\begin{itemize}
			\item for every $S_1, S_2 \in \Scal$, a natural isomorphism of functors $\Hbb(S_1) \times \Hbb(S_2) \to \Hbb(S_1 \times S_2)$
			\begin{equation*}
				e = e_{S_1,S_2} \colon A_1 \boxtimes' A_2 \xrightarrow{\sim} A_1 \boxtimes A_2
			\end{equation*}
		\end{itemize}
		satisfying the following condition:
		\begin{enumerate}
			\item[(eq-ETS)] For every choice of morphisms $f_i \colon T_i \to S_i$ in $\Scal$, $i = 1,2$, the diagram of functors $\Hbb(S_1) \times \Hbb(S_2) \to \Hbb(T_1 \times T_2)$
			\begin{equation*}
				\begin{tikzcd}
					f_1^* A_1 \boxtimes' f_2^* A_2 \arrow{rr}{m'} \arrow{d}{e} &&
					(f_1 \times f_2)^*(A_1 \boxtimes' A_2)  \arrow{d}{e} \\
					f_1^* A_1 \boxtimes f_2^* A_2 \arrow{rr}{m} &&
					(f_1 \times f_2)^*(A_1 \boxtimes A_2)
				\end{tikzcd}
			\end{equation*}
			is commutative.
		\end{enumerate}
	\end{enumerate}
\end{defn}

\begin{rem}
	The analogue of \Cref{rem:ITS-id} holds:
	as an instance of axiom ($m$ETS), the natural isomorphism of functors $\Hbb(S_1) \times \Hbb(S_2) \to \Hbb(S_1 \times S_2)$
	\begin{equation*}
		m_{\id_{S_1},\id_{S_2}} \colon A_1 \boxtimes A_2 \xrightarrow{\sim} A_1 \boxtimes A_2
	\end{equation*}
	is the identity for every $S_1, S_2 \in \Scal$.
\end{rem}

As we now explain, the parallel notions of internal and external tensor structures are different incarnations of the same structure. 
To begin with, the classical formula \eqref{eq:intro-int_to_ext} converts an internal tensor structure into an external tensor structure, as follows:
\begin{lem}\label[lem]{lem:int_to_ext}
	Let $(\otimes,m)$ be an internal tensor structure on $\Hbb$. Then, assigning
	\begin{itemize}
		\item to every $S_1, S_2 \in \Scal$, the functor
		\begin{equation*}
			-\boxtimes- = -\boxtimes_{S_1,S_2}- \colon \Hbb(S_1) \times \Hbb(S_2) \rightarrow \Hbb(S_1 \times S_2)
		\end{equation*}  
		defined by the formula
		\begin{equation*}
			A_1 \boxtimes A_1 := \pi^{12*}_1 A_1 \otimes \pi^{12*}_2 A_2,
		\end{equation*}
		\item to every choice of morphisms $f_i \colon T_i \rightarrow S_i$ in $\Scal$, $i = 1,2$, the natural isomorphism of functors $\Hbb(S_1) \times \Hbb(S_2) \to \Hbb(T_1 \times T_2)$
		\begin{equation*}
			m^e = m^e_{f_1,f_2} \colon f_1^* A_1 \boxtimes f_2^* A_2 \xrightarrow{\sim} (f_1 \times f_2)^*(A_1 \boxtimes A_2)
		\end{equation*}
		defined as the composite
		\begin{equation*}
			\begin{tikzcd}[font=\small]
				f_1^* A_1 \boxtimes f_2^* A_2 \arrow[equal]{d} &&
				(f_1 \times f_2)^* (A_1 \boxtimes A_2) \arrow[equal]{d} \\
				\pi^{12,*}_1 f_1^* A_1 \otimes \pi^{12,*}_2 f_2^* A_2 \arrow[equal]{r} &
				(f_1 \times f_2)^* \pi^{12,*}_1 A_1 \otimes (f_1 \times f_2)^* \pi^{12,*}_2 A_2  \arrow{r}{m} &
				(f_1 \times f_2)^* (\pi^{12,*}_1 A_1 \otimes \pi^{12,*}_2 A_2)  
			\end{tikzcd}
		\end{equation*}
	\end{itemize}
	defines an external tensor structure $(\boxtimes,m^e)$ on $\Hbb$. 
\end{lem}
\begin{proof}
	Let us check that the definition satisfies condition ($m$ETS) from \Cref{defn:ETS}(1):
	given composable morphisms $f_i \colon T_i \to S_i$, $g_i \colon S_i \to V_i$, $i = 1,2$, we have to show that the diagram of functors $\Hbb(V_1) \times \Hbb(V_2) \to \Hbb(T_1 \times T_2)$
	\begin{equation*}
		\begin{tikzcd}
			(g_1 f_1)^* A_1 \boxtimes (g_2 f_2)^* A_2 \arrow{rr}{m^e} \arrow[equal]{d} &&
			(g_1 f_1 \times g_2 f_2)^* (A_1 \boxtimes A_2)   \arrow[equal]{d} \\
			f_1^* g_1^* A_1 \boxtimes f_1^* g_1^* A_2 \arrow{r}{m^e} &
			(f_1 \times f_2)^* (g_1^* A_1 \boxtimes g_2^* A_2) \arrow{r}{m^e} &
			(f_1 \times f_2)^* (g_1 \times g_2)^* (A_1 \boxtimes A_2)
		\end{tikzcd}
	\end{equation*}
	commutes. 
	Unwinding the definitions in terms of the internal tensor product, we obtain the outer rectangle of the diagram
	\begin{equation*}
		\begin{tikzcd}[font=\small]
			(g_{12} f_{12})^* \pi^{12,*}_{1} A_1 \otimes (g_{12} f_{12})^* \pi^{12,*}_{2} A_2 \arrow{rr}{m} \arrow[equal]{d} &&
			(g_1 f_1 \times g_2 f_2)^* (\pi^{12,*}_{1} A_1 \otimes \pi^{12,*}_{2} A_2) \arrow[equal]{d} \\
			\pi^{12,*}_{1} (g_1 f_1)^* A_1 \otimes \pi^{12,*}_{2} (g_2 f_2)^* A_2 \arrow[equal]{u} \arrow[equal]{d} &&
			f_{12}^* g_{12}^* (\pi^{12,*}_{1} A_1 \otimes \pi^{12,*}_{2} A_2) \\
			\pi^{12,*}_{1} f_1^* g_1^* A_1 \otimes \pi^{12,*}_{2} f_2^* g_2^* A_2 \arrow[equal]{d} &
			f_{12}^* g_{12}^* \pi^{12*}_{1} A_1 \otimes f_{12}^* g_{12}^* \pi^{12*}_{2} A_2 \arrow{r}{m} \arrow[equal]{uul} \arrow[equal]{dl} &
			f_{12}^*(g_{12}^* \pi^{12,*}_{1} A_1 \otimes g_{12}^* \pi^{12,*}_{2} A_2) \arrow{u}{m} \\
			f_{12}^* \pi^{12,*}_{1} g_1^* A_1 \otimes f_{12}^* \pi^{12,*}_{2} g_2^* A_2 \arrow{rr}{m} &&
			f_{12}^* (\pi^{12,*}_{1} g_1^* A_1 \otimes \pi^{12,*}_{2} g_2^* A_2), \arrow[equal]{u}
		\end{tikzcd}
	\end{equation*}
	where $f_1 \times f_2$ and $g_1 \times g_2$ are abbreviated as $f_{12}$ and $g_{12}$, respectively.
	Here, the upper-right piece commutes by axiom ($m$ITS),
	the left-most triangle commutes by axiom (fib-1), while the lower-right piece commutes by naturality. 
\end{proof}

Conversely, the classical formula \eqref{eq:intro-ext_to_int} converts an external tensor structure into an internal tensor structure, as follows:
\begin{lem}\label[lem]{lem:ext_to_int}
	Let $(\boxtimes,m)$ be an external tensor structure on $\Hbb$. Then, assigning
	\begin{itemize}
		\item to every $S \in \Scal$, the functor
		\begin{equation*}
			-\otimes- = -\otimes_S - \colon \Hbb(S) \times \Hbb(S) \rightarrow \Hbb(S)
		\end{equation*}  
		defined by the formula
		\begin{equation*}
			A \otimes B := \Delta_S^* (A \boxtimes B),
		\end{equation*}
		\item to every morphism $f \colon T \to S$ in $\Scal$, the natural isomorphism of functors $\Hbb(S) \times \Hbb(S) \to \Hbb(T)$
		\begin{equation*}
			m^i = m^i_f \colon f^* A \otimes f^* B \xrightarrow{\sim} f^*(A \otimes B)
		\end{equation*}
		defined as the composite
		\begin{equation*}
			\begin{tikzcd}[font=\small]
				f^*A \otimes f^*B \arrow[equal]{d} &&
				f^* (A \otimes B) \arrow[equal]{d} \\
				\Delta_T^* (f^*A \boxtimes f^*B) \arrow{r}{m} &
				\Delta_T^* (f \times f)^*(A \boxtimes B) \arrow[equal]{r} &
				f^* \Delta_S^* (A \boxtimes B)
			\end{tikzcd}
		\end{equation*}
	\end{itemize} 
	defines an internal tensor structure $(\otimes,m^i)$ on $\Hbb$.
\end{lem}
\begin{proof}
	Let us check that the definition satisfies condition ($m$ITS) from \Cref{defn:ITS}(1):
	given composable morphisms $f \colon T \to S$, $g \colon S \to V$, we have to show that the diagram of functors $\Hbb(V) \times \Hbb(V) \to \Hbb(T)$
	\begin{equation*}
		\begin{tikzcd}
			(gf)^* A \otimes (gf)^* B  \arrow{rr}{m^i} \arrow[equal]{d} &&
			(gf)^*(A \otimes B) \arrow[equal]{d} \\
			f^* g^* A \otimes f^* g^* B \arrow{r}{m^i} &
			f^* (g^* A \otimes g^* B) \arrow{r}{m^i} &
			f^* g^* (A \otimes B)
		\end{tikzcd}
	\end{equation*}
	commutes.
	Unwinding the definitions in terms of the external tensor product, we obtain the outer rectangle of the diagram
	\begin{equation*}
		\begin{tikzcd}[font=\small]
			\Delta_T^* ((gf)^* A \otimes (gf)^* B) \arrow{r}{m} \arrow[equal]{d} &
			\Delta_T^* (gf \times gf)^* (A \boxtimes B) \arrow[equal]{r} &
			(gf)^* \Delta_V^*(A \boxtimes B) \arrow[equal]{d} \\
			\Delta_T^*(f^* g^* A \boxtimes f^* g^* B) \arrow{d}{m} &
			\Delta_T^* (f \times f)^* (g \times g)^* (A \boxtimes B)  \arrow[equal]{u} \arrow[equal]{dr} &
			f^* g^* \Delta_V^*(A \boxtimes B) \arrow[equal]{d} \\
			\Delta_T^* (f \times f)^* (g^* A \boxtimes g^* B)	\arrow[equal]{r} \arrow{ur}{m} &
			f^* \Delta_S^* (g^* A \boxtimes g^* B) \arrow{r}{m} &
			f^* \Delta_S^* (g \times g)^*(A \boxtimes B).
		\end{tikzcd}
	\end{equation*}
	Here, the upper-left piece commutes by axiom ($m$ETS), the upper-right piece commutes by axiom (fib-1), while the lower piece commutes by naturality.
\end{proof}

The notions of equivalence introduced in \Cref{defn:ITS}(2) and \Cref{defn:ETS}(2) allow one to make it precise in which sense the two constructions just described are mutually inverse:
\begin{prop}\label[prop]{prop_bij_in_ext}
	The constructions of \Cref{lem:int_to_ext} and \Cref{lem:ext_to_int} canonically define mutually inverse bijections between equivalence classes of internal and external tensor structures on $\Hbb$.
\end{prop}
\begin{proof}
	Start from an internal tensor structure $(\otimes,m)$ on $\Hbb$, and let $(\otimes',m')$ denote the internal tensor structure obtained from it by applying first \Cref{lem:int_to_ext} and then \Cref{lem:ext_to_int}: concretely, for every $S \in \Scal$, it is defined by the formula
	\begin{equation*}
		A \otimes' B := \Delta_S^*(\pi^{12,*}_1 A \otimes \pi^{12,*}_2 B).
	\end{equation*}
	We want to construct a canonical equivalence $e \colon (\otimes,m) \xrightarrow{\sim} (\otimes',m')$.
	To this end, for every $S \in \Scal$ we define an isomorphism of functors $\Hbb(S) \times \Hbb(S) \rightarrow \Hbb(S)$
	\begin{equation*}
		e = e_S \colon A \otimes B \xrightarrow{\sim} A \otimes' B
	\end{equation*}
	by taking the composite
	\begin{equation*}
		A \otimes B = \Delta_S^* \pi^{12,*}_1 A \otimes \Delta_S^* \pi^{12,*}_2 B \xrightarrow{m} \Delta_S^* (\pi^{12,*}_1 A \otimes \pi^{12,*}_2 B) =: A \otimes' B.
	\end{equation*}
	We claim that, as $S$ varies, these isomorphisms satisfy condition (eq-ITS):
	given a morphism $f \colon T \rightarrow S$, we have to show that the diagram of functors $\Hbb(S) \times \Hbb(S) \rightarrow \Hbb(T)$
	\begin{equation*}
		\begin{tikzcd}
			f^* A \otimes f^* B \arrow{r}{m} \arrow{d}{e} &
			f^* (A \otimes B)  \arrow{d}{e} \\
			f^* A \otimes' f^* B \arrow{r}{m'} &
			f^* (A \otimes' B)  
		\end{tikzcd}
	\end{equation*}
	commutes.
	Unwinding the definitions, we obtain the outer part of the diagram
	\begin{equation*}
		\begin{tikzcd}[font=\small]
			&
			\Delta_T^* \pi^{12,*}_1 f^* A \otimes \Delta_T^* \pi^{12,*}_2 f^* B \arrow{r}{m} \arrow[equal]{d} &
			\Delta_T^* (\pi^{12,*}_1 f^* A \otimes \pi^{12,*}_2 f^* B) \arrow[equal]{d} \\
			f^* A \otimes f^* B \arrow[equal]{ur} \arrow[equal]{dr} \arrow{dd}{m} &
			\Delta_T^* (f \times f)^* \pi^{12,*}_1 A \otimes \Delta_T^* (f \times f)^* \pi^{12,*}_2 B \arrow[equal]{d} \arrow{r}{m} &
			\Delta_T^* ((f \times f)^* \pi^{12,*}_1 A \otimes (f \times f)^* \pi^{12,*}_2 B) \arrow{d}{m} \\
			&
			f^* \Delta_S^* \pi^{12,*}_1 A \otimes f^* \Delta_S^* \pi^{12,*}_2 B \arrow{d}{m} & 
			\Delta_T^* (f \times f)^* (\pi^{12,*}_1 A \otimes \pi^{12,*}_2 B) \arrow[equal]{d} \\
			f^* (A \otimes B) \arrow[equal]{r} &
			f^*(\Delta_S^* \pi^{12,*}_1 A \otimes \Delta_S^* \pi^{12,*}_2 B) \arrow{r}{m} &
			f^* (\Delta_S^* (\pi^{12,*}_1 A \otimes \pi^{12,*}_2 B)).
		\end{tikzcd}
	\end{equation*}
	Here, the lower-right rectangle commutes by axiom ($m$ITS), the upper-left triangle commutes by axiom (fib-1), while the other two pieces commute by naturality.
	This proves the claim.
	
	Conversely, start from an external tensor structure $(\boxtimes,m)$ on $\Hbb$, and let $(\boxtimes',m')$ denote the external tensor structure obtained from it by applying first \Cref{lem:ext_to_int} and then \Cref{lem:int_to_ext}: concretely, for every $S_1,S_2 \in \Scal$, it is defined by the formula
	\begin{equation*}
		A_1 \boxtimes' A_2 := \Delta_{S_1 \times S_2}^*(\pi^{12,*}_1 A_1 \boxtimes \pi^{12,*}_2 A_2).
	\end{equation*}
	We want to construct a canonical equivalence $e \colon (\boxtimes',m') \xrightarrow{\sim} (\boxtimes,m)$.
	To this end, for every $S_1, S_2 \in \Scal$ we define an isomorphism of functors $\Hbb(S_1) \times \Hbb(S_2) \rightarrow \Hbb(S_1 \times S_2)$
	\begin{equation*}
		e = e_{S_1,S_2} \colon A_1 \boxtimes' A_2 \xrightarrow{\sim} A_1 \boxtimes A_2
	\end{equation*}
	by taking the composite
	\begin{equation*}
		A_1 \boxtimes' A_2 := \Delta_{S_1 \times S_2}^* (\pi^{12,*}_1 A_1 \boxtimes \pi^{12,*}_2 A_2) \xrightarrow{m} \Delta_{S_1 \times S_2}^* (\pi^{12}_1 \times \pi^{12}_2)^* (A_1 \boxtimes A_2) = A_1 \boxtimes A_2.
	\end{equation*}
	We claim that, as $S_1,S_2$ vary, these isomorphisms satisfy condition (eq-ETS): 
	given two morphisms $f_i \colon T_i \rightarrow S_i$, $i = 1,2$, we have to show that the diagram of functors $\Hbb(S_1) \times \Hbb(S_2) \rightarrow \Hbb(T_1 \times T_2)$
	\begin{equation*}
		\begin{tikzcd}
			f_1^* A_1 \boxtimes' f_2^* A_2 \arrow{r}{m'} \arrow{d}{e} &
			(f_1 \times f_2)^* (A_1 \boxtimes' A_2)  \arrow{d}{e} \\
			f_1^* A_1 \boxtimes f_2^* A_2 \arrow{r}{m} &
			(f_1 \times f_2)^* (A_1 \boxtimes A_2)
		\end{tikzcd}
	\end{equation*}
	commutes.
	Unwinding the definitions, we obtain the outer part of the diagram
	\begin{equation*}
		\begin{tikzcd}[font=\small]
			\Delta_{T_1 \times T_2}^* (\pi^{12,*}_1 f_1^* A_1 \boxtimes \pi^{12}_2 f_2^* A_2) \arrow{r}{m} \arrow[equal]{d} &
			\Delta_{T_1 \times T_2}^* (\pi^{12}_1 \times \pi^{12}_2)^* (f_1^* A_1 \boxtimes f_2^* A_2) \arrow[equal]{r} \arrow{d}{m} &
			f_1^* A_1 \boxtimes f_2^* A_2 \arrow{dd}{m} \\
			\Delta_{T_1 \times T_2}^* (f_{12}^* \pi^{12,*}_1 A_1 \boxtimes f_{12}^* \pi^{12,*}_2 A_2) \arrow{d}{m} &
			\Delta_{T_1 \times T_2}^* (\pi^{12}_1 \times \pi^{12}_2)^* f_{12}^* (A_1 \boxtimes A_2) \arrow[equal]{d} \\
			\Delta_{T_1 \times T_2}^* (f_{12} \times f_{12})^* (\pi^{12,*}_1 A_1 \boxtimes \pi^{12,*}_2 A_2) \arrow[equal]{d} \arrow{r}{m} &
			\Delta_{T_1 \times T_2}^* (f_{12} \times f_{12})^* (\pi^{12}_1 \times \pi^{12}_2)^* (A_1 \boxtimes A_2) \arrow[equal]{d} &
			f_{12}^* (A_1 \boxtimes A_2) \arrow[equal]{ul} \arrow[equal]{dl} \\
			f_{12}^* \Delta_{S_1 \times S_2}^* (\pi^{12,*}_1 A_1 \boxtimes \pi^{12,*}_2 A_2) \arrow{r}{m} &
			f_{12}^* \Delta_{S_1 \times S_2}^* (\pi^{12}_1 \times \pi^{12}_2)^* (A_1 \boxtimes A_2),
		\end{tikzcd}
	\end{equation*}	
	where $f_1 \times f_2$ is abbreviated as $f_{12}$.
	Here, the upper-left rectangle commutes by axiom ($m$ETS), the lower-right triangle commutes by axiom (fib-1), while the other two pieces commute by naturality.
	This proves the claim, thereby concluding the proof.
\end{proof}

\begin{rem}\label[rem]{rem:total-otimes}
	Let us rephrase the above discussion in the language of Grothendieck fibrations. 
	We let $\Hcal$ denote the Grothendieck construction of $\Hbb$:
	it is the category where
	objects are pairs of the form $(S,A)$ with $S \in \Scal$ and $A \in \Hbb(S)$, and where a morphism $(T,B) \to (S,A)$ is a pair $(f,\phi)$ consisting of a morphism $f \colon T \to S$ in $\Scal$ and a morphism $\phi \colon B \to f^* A$ in $\Hbb(T)$.
	Let $\beta \colon \Hcal \to \Scal$ denote the canonical functor.
	If $(\boxtimes,m)$ is an external tensor structure on $\Hbb$, one can attach to it a bi-functor
	\begin{equation*}
		- \circledast - \colon \Hcal \times \Hcal \to \Hcal
	\end{equation*}  
	preserving Cartesian arrows and rendering the diagram
	\begin{equation*}
		\begin{tikzcd}
			\Hcal \times \Hcal \arrow{rr}{- \circledast -} \arrow{d}{\beta \times \beta} &&
			\Hcal \arrow{d}{\beta} \\
			\Scal \times \Scal \arrow{rr}{- \times -} &&
			\Scal
		\end{tikzcd}
	\end{equation*}
	strictly commutative:
	it is defined on objects by the formula $(S_1,A_1) \circledast (S_2,A_2) := (S_1 \times S_2,A_1 \boxtimes A_2)$, and on morphisms by the formula $(f_1,\phi_1) \circledast (f_2,\phi_2) := (f_1 \times f_2, m_{f_1,f_2}(\phi_1 \boxtimes \phi_2))$;
	compatibility with composition follows from axiom ($m$ETS).
	The external tensor structure $(\boxtimes,m)$ can be canonically recovered from this bi-functor:
	the external tensor product functor $- \boxtimes_{S_1,S_2} -$ is obtained by restriction to the fibre $(\Hcal \times \Hcal)_{(S_1,S_2)} = \Hbb(S_1) \times \Hbb(S_2)$, while the external monoidality isomorphism $m_{f_1,f_2}$ is obtained by restriction to the Cartesian arrows $(f_i,\id_{f_i^* A_i}) \colon (T_i,f_i^* A_i) \to (S_i,A_i)$, $i = 1,2$.
	See the proof of \cite[Thm.~12.8]{Shu08} for a more thorough discussion. 
\end{rem}

\subsection{Projection formulae}

To conclude the first section, we record a complementary result which is important in applications to six-functor formalisms but plays no role in the rest of the present paper.

\begin{nota}
	For the duration of the present subsection, we assume to be given a (not necessarily full) subcategory $\Scal^{sm}$ of $\Scal$, stable under pull-backs in the following sense: 
	for every Cartesian square in $\Scal$
	\begin{equation*}
		\begin{tikzcd}
			P' \arrow{r}{g'} \arrow{d}{p'} &
			P \arrow{d}{p} \\
			S' \arrow{r}{g} &
			S
		\end{tikzcd}
	\end{equation*} 
	with $p \colon P \to S$ in $\Scal^{sm}$, the morphism $p' \colon P' \to S'$ belongs to $\Scal^{sm}$ as well. 
	We call \textit{smooth morphisms} the morphisms in $\Scal^{sm}$.
\end{nota}

\begin{ex}
	The motivating example (which also justifies the terminology) is when $\Scal$ is some category of algebraic varieties or schemes and $\Scal^{sm}$ denotes its subcategory of smooth morphisms. 
\end{ex}

\begin{defn}\label[defn]{defn:geometric-fibcat}
	We say that a given $\Scal$-fibered category $\Hbb$ is \textit{geometric} if it satisfies the following two conditions:
	\begin{enumerate}
		\item[(sm-0)] For every smooth morphism $p \colon P \rightarrow S$ in $\Scal$, the functor $p^* \colon \Hbb(S) \rightarrow \Hbb(P)$ admits a left-adjoint $p_{\#} \colon \Hbb(P) \rightarrow \Hbb(S)$.
		\item[(sm-1)] For every Cartesian square in $\Scal$
		\begin{equation*}
			\begin{tikzcd}
				P' \arrow{r}{g'} \arrow{d}{p'} &
				P \arrow{d}{p} \\
				S' \arrow{r}{g} &
				S
			\end{tikzcd}
		\end{equation*}
		where $p$ (hence also $p'$) is a smooth morphism, the natural transformation of functors $\Hbb(P) \rightarrow \Hbb(S')$
		\begin{equation*}
			{p'}_{\#} {g'}^* A \xrightarrow{\eta} {p'}_{\#} {g'}^* p^* p_{\#} A = {p'}_{\#} {p'}^* g^* p_{\#} A \xrightarrow{\epsilon} g^* p_{\#} A
		\end{equation*}
		is invertible -- this is known as the \textit{Beck--Chevalley condition}.
	\end{enumerate}
\end{defn}

Until the end of the present subsection, we assume that the $\Scal$-fibered category $\Hbb$ is geometric in the sense of \Cref{defn:geometric-fibcat}.
It is natural to ask whether a given monoidal structure on $\Hbb$ satisfies the projection formulae.
This makes sense both in the internal and in the external setting:

\begin{defn}
	We say that an internal tensor structure $(\otimes,m)$ on $\Hbb$ is \textit{geometric} if the following condition holds:
	\begin{enumerate}
		\item[(ITS-pf)] For every smooth morphism $p \colon P \rightarrow S$, the natural transformation of functors $\Hbb(P) \times \Hbb(S) \rightarrow \Hbb(S)$
		\begin{equation*}
			p_{\#}(A \otimes p^* B) \xrightarrow{\eta} p_{\#}(p^* p_{\#} A \otimes p^* B) \xrightarrow{m} p_{\#} p^*(p_{\#} A \otimes B) \xrightarrow{\epsilon} p_{\#} A \otimes B
		\end{equation*}
		and the natural transformation of functors $\Hbb(S) \times \Hbb(P) \rightarrow \Hbb(S)$
		\begin{equation*}
			p_{\#}(p^* A \otimes B) \xrightarrow{\eta} p_{\#}(p^* A \otimes p^* p_{\#} B) \xrightarrow{m} p_{\#} p^*(A \otimes p_{\#} B) \xrightarrow{\epsilon} A \otimes p_{\#} B
		\end{equation*}
		are invertible.
	\end{enumerate}
\end{defn}

\begin{defn}
	We say that an external tensor structure $(\boxtimes,m)$ on $\Hbb$ is \textit{geometric} if the the following condition holds:
	\begin{enumerate}
		\item[(ETS-pf)] For every choice of smooth morphisms $p_i \colon P_i \rightarrow S_i$, $i = 1,2$, the natural transformation of functors $\Hbb(P_1) \times \Hbb(S_2) \rightarrow \Hbb(S_1 \times S_2)$
		\begin{equation*}
			\begin{tikzcd}[font=\small]
				(p_1 \times \id_{S_2})_{\#}(A_1 \boxtimes A_2) \arrow{d}{\eta} &
				p_{1,\#} A_1 \boxtimes A_2 \\
				(p_1 \times \id_{S_2})_{\#}(p_1^* p_{1,\#} A_1 \boxtimes A_2) \arrow{r}{m} &
				(p_1 \times \id_{S_2})_{\#} (p_1 \times \id_{S_2})^*(p_{1,\#} A_1 \boxtimes A_2) \arrow{u}{\epsilon}
			\end{tikzcd}
		\end{equation*}
		and the natural transformation of functors $\Hbb(S_1) \times \Hbb(P_2) \rightarrow \Hbb(S_1 \times S_2)$
		\begin{equation*}
			\begin{tikzcd}[font=\small]
				(\id_{S_1} \times p_2)_{\#}(A_1 \boxtimes A_2) \arrow{d}{\eta} &
				A_1 \boxtimes p_{2,\#} A_2 \\ 
				(\id_{S_1} \times p_2)_{\#}(A_1 \boxtimes p_2^* p_{2,\#} A_2) \arrow{r}{m} &
				(\id_{S_1} \times p_2)_{\#} (\id_{S_1} \times p_2)^*(A_1 \boxtimes p_{2,\#} A_2) \arrow{u}{\epsilon} 
			\end{tikzcd}
		\end{equation*}
		are invertible.
	\end{enumerate}
\end{defn}

Clearly, any internal (resp. external) tensor structure which is equivalent to a geometric internal (resp. external) tensor structure is itself geometric.
Hence, it makes sense to ask whether the bijection of \Cref{prop_bij_in_ext} respects the property of being geometric.
This is indeed the case, as the following results show:

\begin{lem}\label[lem]{lem:pf-int_to_ext}
	Let $(\otimes,m)$ be a geometric internal tensor structure on $\Hbb$.
	Then, the external tensor structure $(\boxtimes,m^e)$ obtained from it via \Cref{lem:int_to_ext} is geometric as well.
\end{lem}
\begin{proof}
	Let us show that, for every smooth morphism $p_1 \colon P_1 \to S_1$ and every $S_2 \in \Scal$, the natural transformation of functors $\Hbb(P_1) \times \Hbb(S_2) \to \Hbb(S_1 \times S_2)$
	\begin{equation*}
		\begin{tikzcd}
			(p_1 \times \id_{S_2})_{\#} (A_1 \boxtimes A_2) \arrow{d}{\eta} &
			p_{1,\#} A_1 \boxtimes A_2 \\
			(p_1 \times \id_{S_2})_{\#} (p_1^* p_{1,\#} A_1 \boxtimes A_2) \arrow{r}{m^e} &
			(p_1 \times \id_{S_2})_{\#} (p_1 \times \id_{S_2})^* (p_{1,\#} A_1 \boxtimes A_2) \arrow{u}{\epsilon}
		\end{tikzcd}
	\end{equation*}
	is invertible;
	the same argument works for the specular natural transformation.
	Unwinding the definitions, the natural transformation in question is the left-most vertical composite in the commutative diagram
	\begin{equation*}
		\begin{tikzcd}[font=\tiny]
			(p_1 \times \id_{S_2})_{\#} (\pi^{12,*}_1 A_1 \otimes \pi^{12,*}_2 A_2) \arrow[equal]{r} \arrow{d}{\eta} &
			(p_1 \times \id_{S_2})_{\#} (\pi^{12,*}_1 A_1 \otimes (p_1 \times \id_{S_2})^* \pi^{12,*}_2 A_2) \arrow{d}{\eta} \\
			(p_1 \times \id_{S_2})_{\#} (\pi^{12,*}_1 p_1^* p_{1,\#} A_1 \otimes \pi^{12*}_2 A_2) \arrow[equal]{d} &
			(p_1 \times \id_{S_2})_{\#} ((p_1 \times \id_{S_2})^* (p_1 \times \id_{S_2})_{\#} \pi^{12,*}_1 A_1 \otimes (p_1 \times \id_{S_2})^* \pi^{12,*}_2 A_2) \arrow{d}{m} \arrow{dl}{\sim} \\
			(p_1 \times \id_{S_2})_{\#} ((p_1 \times \id_{S_2})^* \pi^{12,*}_1 p_{1,\#} A_1 \otimes (p_1 \times \id_{S_2})^* \pi^{12,*}_2 A_2) \arrow{d}{m} &
			(p_1 \times \id_{S_2})_{\#} (p_1 \times \id_{S_2})^* ((p_1 \times \id_{S_2})_{\#} \pi^{12,*}_1 A_1 \otimes \pi^{12,*}_2 A_2) \arrow{d}{\epsilon} \arrow{dl}{\sim} \\
			(p_1 \times \id_{S_2})_{\#} (p_1 \times \id_{S_2})^* (\pi^{12,*}_1 p_{1,\#} A_1 \otimes \pi^{12,*}_2 A_2) \arrow{d}{\epsilon} &
			((p_1 \times \id_{S_2})_{\#} \pi^{12,*}_1 A_1 \otimes \pi^{12,*}_2 A_2) \arrow{dl}{\sim} \\
			(\pi^{12,*}_1 p_{1,\#} A_1 \otimes \pi^{12,*}_2 A_2).
		\end{tikzcd}
	\end{equation*}
	Here, the three diagonal natural isomorphisms are induced by the natural isomorphism between functors $\Hbb(P_1) \to \Hbb(S_1 \times S_2)$
	\begin{equation*}
		(p_1 \times \id_{S_2})_{\#} \pi^{12*}_1 A_1 \xrightarrow{\sim} \pi^{12*}_1 p_{1,\#} A_1
	\end{equation*}  
	obtained by applying axiom (sm-1) to the Cartesian square
	\begin{equation*}
		\begin{tikzcd}
			P_1 \times S_2 \arrow{rr}{\pi^{12}_1} \arrow{d}{p_1 \times \id_{S_2}} &&
			P_1 \arrow{d}{p_1} \\
			S_1 \times S_2 \arrow{rr}{\pi^{12}_1} &&
			S_1.
		\end{tikzcd}
	\end{equation*}
	Since the right-most vertical composite in the above diagram is invertible (because $(\otimes,m)$ is geometric, by hypothesis), the left-most vertical composite must be invertible as well.
\end{proof}

\begin{lem}\label[lem]{lem:pf-ext_to_int}
	Let $(\boxtimes,m)$ be a geometric external tensor structure on $\Hbb$.
	Then, the internal tensor structure $(\otimes,m^i)$ obtained from it via \Cref{lem:ext_to_int} is geometric as well.
\end{lem}
\begin{proof}
	Let us show that, for every smooth morphism $p \colon P \to S$, the natural transformation of functors $\Hbb(P) \times \Hbb(S) \to \Hbb(S)$
	\begin{equation*}
		p_{\#}(A \otimes p^* B) \xrightarrow{\eta} p_{\#}(p^* p_{\#} A \otimes p^* B) \xrightarrow{m} p_{\#} p^*(p_{\#} A \otimes B) \xrightarrow{\epsilon} p_{\#} A \otimes B
	\end{equation*}
	is invertible;
	the same argument works for the specular natural transformation.
	Unwinding the definitions, the natural transformation in question is the left-most diagonal composite in the commutative diagram
	\begin{equation*}
		\begin{tikzcd}[font=\tiny]
			p_{\#} \Delta_P^* (A \boxtimes p^* B) \arrow{r}{m} \arrow{d}{\eta} &
			p_{\#} \Delta_P^* (\id_P \times p)^*(A \boxtimes B) \arrow{d}{\eta} \arrow[equal]{r} &
			p_{\#}(\id_P,p)^* (A \boxtimes B) \arrow{dr}{\sim} \arrow{d}{\eta}  \\
			p_{\#} \Delta_P^* (p^* p_{\#} A \boxtimes p^* B) \arrow{r}{m} \arrow{dr}{m} &
			p_{\#} \Delta_P^* (\id_P \times p)^* (p^*p_{\#} A \boxtimes B) \arrow{d}{m} \arrow[equal]{r} &
			p_{\#} (\id_P,p)^* (p^* p_{\#} A \boxtimes B) \arrow{dr}{\sim} \arrow{d}{m} &
			\Delta_S^* (p \times \id_S)_{\#} (A \boxtimes B) \arrow{d}{\eta} \\
			&
			p_{\#} \Delta_P^* (p \times p)^* (p_{\#} A \boxtimes B) \arrow[equal]{r} \arrow[equal]{dr} &
			p_{\#} (\id_P,p)^* (p \times \id_S)^* (p_{\#} A \boxtimes B) \arrow[equal]{d} \arrow{dr}{\sim} &
			\Delta_S^* (p \times \id_S)_{\#}(p^* p_{\#} A \boxtimes B) \arrow{d}{m} \\
			&&
			p_{\#} p^* \Delta_S^* (p_{\#} A \boxtimes B) \arrow{dr}{\epsilon} &
			\Delta_S^* (p \times \id_S)_{\#} (p \times \id_S)^* (p_{\#} A \boxtimes B) \arrow{d}{\epsilon} \\
			&&&
			\Delta_S^* (p_{\#} A \boxtimes B).
		\end{tikzcd}
	\end{equation*}
	Here, the three diagonal natural isomorphisms on the right side are induced by the natural isomorphism between functors $\Hbb(P \times S) \to \Hbb(S)$
	\begin{equation*}
		p_{\#} (\id_P,p)^* A \xrightarrow{\sim} \Delta_S^* (p \times \id_S)_{\#} A
	\end{equation*} 
	obtained by applying axiom (sm-1) to the Cartesian square
	\begin{equation*}
		\begin{tikzcd}
			P \arrow{rr}{(\id_P,p)} \arrow{d}{p} &&
			P \times S \arrow{d}{p \times \id_S} \\
			S \arrow{rr}{\Delta_S} &&
			S \times S.
		\end{tikzcd}
	\end{equation*}
	Since the right-most vertical composite in the above diagram is invertible (because $(\boxtimes,m)$ is geometric, by hypothesis), the left-most composite must be invertible as well.
\end{proof}

\section{Translating associativity, commutativity, and unit constraints}\label{sect:acu}

As in the previous section, we work over a fixed $\Scal$-fibered category $\Hbb$.
In this section, we recall how the correspondence between internal and external tensor structures given by \Cref{prop_bij_in_ext} can be refined so as to include associativity, commutativity, and unit constraints.
After treating the three constraints separately, we discuss their mutual compatibility conditions.
Summing up all these results, the proof of \Cref{thm:int_ext} will be complete.

For sake of brevity, we omit the proofs of all the results stated below:
they consist in lengthy but straightforward computations with commutative diagrams, along the lines of \Cref{sect:int_ext}.

\subsection{Associativity constraints}

\begin{defn}\label[defn]{defn:aITS}
	Let $(\otimes,m)$ be an internal tensor structure on $\Hbb$, in the sense of \Cref{defn:ITS}.
	\begin{enumerate}
		\item An \textit{internal associativity constraint} $a$ for $(\otimes,m)$ is the datum of
		\begin{itemize}
			\item for every $S \in \Scal$, a natural isomorphism between functors $\Hbb(S) \times \Hbb(S) \times \Hbb(S) \to \Hbb(S)$
			\begin{equation*}
				a = a_S \colon (A \otimes B) \otimes C \xrightarrow{\sim} A \otimes (B \otimes C)
			\end{equation*}
		\end{itemize}
		satisfying the following conditions:
		\begin{enumerate}
			\item[($a$ITS-1)] For every $S \in \Scal$, the diagram of functors $\Hbb(S) \times \Hbb(S) \times \Hbb(S) \times \Hbb(S) \to \Hbb(S)$
			\begin{equation*}
				\begin{tikzcd}
					((A \otimes B) \otimes C) \otimes D \arrow{d}{a} \arrow{rr}{a} && 
					(A \otimes B) \otimes (C \otimes D) \arrow{d}{a} \\
					(A \otimes (B \otimes C)) \otimes D \arrow{r}{a} & 
					A \otimes ((B \otimes C) \otimes D) \arrow{r}{a} &
					A \otimes (B \otimes (C \otimes D))
				\end{tikzcd}
			\end{equation*}
			is commutative -- this is Mac Lane's \textit{pentagon axiom}.
			\item[($a$ITS-2)] For every morphism $f \colon T \to S$ in $\Scal$, the diagram of functors $\Hbb(S) \times \Hbb(S) \times \Hbb(S) \to \Hbb(T)$
			\begin{equation*}
				\begin{tikzcd}
					(f^* A \otimes f^* B) \otimes f^* C \arrow{r}{m} \arrow{d}{a} &
					f^* (A \otimes B) \otimes f^*C \arrow{r}{m} &
					f^* ((A \otimes B) \otimes C) \arrow{d}{a} \\
					f^* A \otimes (f^* B \otimes f^* C) \arrow{r}{m} &
					f^* A \otimes f^* (B \otimes C) \arrow{r}{m} &
					f^* (A \otimes (B \otimes C))
				\end{tikzcd}
			\end{equation*}
			is commutative.
		\end{enumerate}
		We say that $(\otimes,m;a)$ is an \textit{associative internal tensor structure} on $\Hbb$.
		\item Let $(\otimes,m;a)$ and $(\otimes',m';a')$ be two associative internal tensor structures on $\Hbb$. 
		An \textit{equivalence of associative internal tensor structures} $e \colon (\otimes,m;a) \xrightarrow{\sim} (\otimes',m;a')$ is an equivalence of internal tensor structures $e \colon (\otimes,m) \xrightarrow{\sim} (\otimes',m')$ satisfying the following additional condition:
		\begin{enumerate}
			\item[(eq-$a$ITS)] For every $S \in \Scal$, the diagram of functors $\Hbb(S) \times \Hbb(S) \times \Hbb(S) \rightarrow \Hbb(S)$
			\begin{equation*}
				\begin{tikzcd}
					(A \otimes B) \otimes C \arrow{r}{e} \arrow{d}{a} &
					(A \otimes' B) \otimes C \arrow{r}{e} &
					(A \otimes' B) \otimes' C \arrow{d}{a'} \\
					A \otimes (B \otimes C) \arrow{r}{e} &
					A \otimes (B \otimes' C) \arrow{r}{e} &
					A \otimes' (B \otimes' C)
				\end{tikzcd}
			\end{equation*}
			is commutative.
		\end{enumerate}
	\end{enumerate}
\end{defn}

\begin{defn}\label[defn]{defn:aETS}
	Let $(\boxtimes,m)$ be an external tensor structure on $\Hbb$, in the sense of \Cref{defn:ETS}.
	\begin{enumerate}
		\item An \textit{external associativity constraint} $a$ for $(\boxtimes,m)$ is the datum of
		\begin{itemize}
			\item for every $S_1, S_2, S_3 \in \Scal$, a natural isomorphism between functors $\Hbb(S_1) \times \Hbb(S_2) \times \Hbb(S_3) \to \Hbb(S_1 \times S_2 \times S_3)$
			\begin{equation*}
				a = a_{S_1,S_2,S_3} \colon (A_1 \boxtimes A_2) \boxtimes A_3 \xrightarrow{\sim} A_1 \boxtimes (A_2 \boxtimes A_3)
			\end{equation*}
		\end{itemize}
		satisfying the following conditions:
		\begin{enumerate}
			\item[($a$ETS-1)] For every $S_1, S_2, S_3, S_4 \in \Scal$, the diagram of functors $\Hbb(S_1) \times \Hbb(S_2) \times \Hbb(S_3) \times \Hbb(S_4) \to \Hbb(S_1 \times S_2 \times S_3 \times S_4)$
			\begin{equation*}
				\begin{tikzcd}
					((A_1 \boxtimes A_2) \boxtimes A_3) \boxtimes A_4 \arrow{d}{a} \arrow{rr}{a} &&
					(A_1 \boxtimes A_2) \boxtimes (A_3 \boxtimes A_4) \arrow{d}{a} \\
					(A_1 \boxtimes (A_2 \boxtimes A_3)) \boxtimes A_4 \arrow{r}{a} &
					A_1 \boxtimes ((A_2 \boxtimes A_3) \boxtimes A_4) \arrow{r}{a} &
					A_1 \boxtimes (A_2 \boxtimes (A_3 \boxtimes A_4))
				\end{tikzcd}
			\end{equation*}
			is commutative.
			\item[($a$ETS-2)] For every three morphisms $f_i \colon T_i \to S_i$ in $\Scal$, $i = 1,2,3$, the diagram of functors $\Hbb(S_1) \times \Hbb(S_2) \times \Hbb(S_3) \to \Hbb(T_1 \times T_2 \times T_3)$
			\begin{equation*}
				\begin{tikzcd}
					(f_1^* A_1 \boxtimes f_2^* A_2) \boxtimes f_3^* A_3 \arrow{r}{m} \arrow{d}{a} &
					(f_1 \times f_2)^* (A_1 \boxtimes A_2) \boxtimes f_3^* A_3 \arrow{r}{m} &
					(f_1 \times f_2 \times f_3)^* ((A_1 \boxtimes A_2) \boxtimes A_3) \arrow{d}{a} \\
					f_1^* A_1 \boxtimes (f_2^* A_2 \boxtimes f_3^* A_3) \arrow{r}{m} &
					f_1^* A_1 \boxtimes (f_2 \times f_3)^* (A_2 \boxtimes A_3) \arrow{r}{m} &
					(f_1 \times f_2 \times f_3)^* (A_1 \boxtimes (A_2 \boxtimes A_3))  
				\end{tikzcd}
			\end{equation*}
			is commutative.
		\end{enumerate}
		We say that $(\boxtimes,m;a)$ is an \textit{associative external tensor structure} on $\Hbb$.
		\item Let $(\boxtimes',m';a')$ and $(\boxtimes,m;a)$ be two associative external tensor structures on $\Hbb$. 
		An \textit{equivalence of associative external tensor structures} $e \colon (\boxtimes',m';a') \xrightarrow{\sim} (\boxtimes,m;a)$ is an equivalence of external tensor structures $e \colon (\boxtimes',m') \xrightarrow{\sim} (\boxtimes,m)$ satisfying the following additional condition:
		\begin{enumerate}
			\item[(eq-$a$ETS)] For every $S_1, S_2, S_3 \in \Scal$, the diagram of functors $\Hbb(S_1) \times \Hbb(S_2) \times \Hbb(S_3) \to \Hbb(S_1 \times S_2 \times S_3)$
			\begin{equation*}
				\begin{tikzcd}
					(A_1 \boxtimes' A_2) \boxtimes' A_3 \arrow{r}{e} \arrow{d}{a'} &
					(A_1 \boxtimes A_2) \boxtimes' A_3 \arrow{r}{e} &
					(A_1 \boxtimes A_2) \boxtimes A_3 \arrow{d}{a} \\
					A_1 \boxtimes' (A_2 \boxtimes' A_3) \arrow{r}{e} &
					A_1 \boxtimes' (A_2 \boxtimes A_3) \arrow{r}{e} &
					A_1 \boxtimes (A_2 \boxtimes A_3)
				\end{tikzcd}
			\end{equation*}
			is commutative.
		\end{enumerate}
	\end{enumerate}
\end{defn}

Associativity constraints can be naturally translated between the internal and the external setting, as follows:

\begin{lem}\label[lem]{lem:ass-int_to_ext}
	Let $(\otimes,m)$ be an internal tensor structure on $\Hbb$, and let $(\boxtimes,m^e)$ denote the external tensor structure obtained from it via \Cref{lem:int_to_ext}. 
	Suppose that we are given an internal associativity constraint $a$ for $(\otimes,m)$. 
	Then, associating
	\begin{itemize}
		\item to every $S_1, S_2, S_3 \in \Scal$, the natural isomorphism between functors $\Hbb(S_1) \times \Hbb(S_2) \times \Hbb(S_3) \to \Hbb(S_1 \times S_2 \times S_3)$
		\begin{equation*}
			a^e = a^e_{S_1,S_2,S_3} \colon (A_1 \boxtimes A_2) \boxtimes A_3 \xrightarrow{\sim} A_1 \boxtimes (A_2 \boxtimes A_3)
		\end{equation*}
		defined by taking the composite
		\begin{equation*}
			\begin{tikzcd}[font=\small]
				(A_1 \boxtimes A_2) \boxtimes A_3 \arrow[equal]{d} &&
				A_1 \boxtimes (A_2 \boxtimes A_3) \arrow[equal]{d} \\
				\pi^{123,*}_{12} (\pi^{12,*}_{1} A_1 \otimes \pi^{12,*}_{2} A_2) \otimes \pi^{123,*}_{3} A_3 &&
				\pi^{123,*}_{1} A_1 \otimes \pi^{123,*}_{23} (\pi^{23,*}_{2} A_2 \otimes \pi^{23,*}_{3} A_3) \\
				(\pi^{123,*}_{12} \pi^{12,*}_{1} A_1 \otimes \pi^{123,*}_{12} \pi^{12,*}_{2} A_2) \otimes \pi^{123,*}_{3} A_3 \arrow{u}{m} \arrow[equal]{d} && 
				\pi^{123,*}_{1} A_1 \otimes (\pi^{123,*}_{23} \pi^{23,*}_{2} A_2 \otimes \pi^{123,*}_{23} \pi^{23,*}_{3} A_3) \arrow{u}{m} \arrow[equal]{d} \\
				(\pi^{123,*}_{1} A_1 \otimes \pi^{123,*}_{2} A_2) \otimes \pi^{123,*}_{3} A_3 \arrow{rr}{a} &&
				\pi^{123,*}_{1} A_1 \otimes (\pi^{123,*}_{2} A_2 \otimes \pi^{123,*}_{3} A_3)
			\end{tikzcd}
		\end{equation*}
	\end{itemize}
	defines an external associativity constraint $a^e$ for $(\boxtimes,m^e)$.
\end{lem}

\begin{lem}\label[lem]{lem:ass-ext_to_int}
	Let $(\boxtimes,m)$ be an external tensor structure on $\Hbb$, and let $(\otimes,m^i)$ denote the internal tensor structure obtained from it via \Cref{lem:ext_to_int}. 
	Suppose that we are given an external associativity constraint $a$ for $(\boxtimes,m)$. 
	Then, associating
	\begin{itemize}
		\item to every $S \in \Scal$, the natural isomorphism between functors $\Hbb(S) \times \Hbb(S) \times \Hbb(S) \to \Hbb(S)$
		\begin{equation*}
			a^i = a^i_S \colon (A \otimes B) \otimes C \xrightarrow{\sim} A \otimes (B \otimes C)
		\end{equation*}
		defined by taking the composite
		\begin{equation*}
			\begin{tikzcd}[font=\small]
				(A \otimes B) \otimes C \arrow[equal]{d} &&
				A \otimes (B \otimes C) \arrow[equal]{d} \\
				\Delta_S^* (\Delta_S^* (A \boxtimes B) \boxtimes C) \arrow{d}{m} &&
				\Delta_S^* (A \boxtimes \Delta_S^* (B \boxtimes C)) \arrow{d}{m} \\
				\Delta_S^* (\Delta_S \times \id_S)^* ((A \boxtimes B) \boxtimes C) \arrow[equal]{d} &&
				\Delta_S^* (\id_S \times \Delta_S)^* (A \boxtimes (B \boxtimes C)) \arrow[equal]{d} \\
				\Delta_S^{(3),*} ((A \boxtimes B) \boxtimes C) \arrow{rr}{a} &&
				\Delta_S^{(3),*} (A \boxtimes (B \boxtimes C))
			\end{tikzcd}
		\end{equation*}
	\end{itemize}
	defines an internal associativity constraint $a^i$ for $(\otimes,m^i)$.
\end{lem}

The proofs amount to checking the commutativity of the natural diagrams expressing the axioms of associativity constraints:
axiom ($a$ITS-1) in the internal setting corresponds to axiom ($a$ETS-1) in the external setting, while axiom ($a$ITS-2) corresponds to axiom ($a$ETS-2).

\begin{prop}
	The constructions of \Cref{lem:ass-int_to_ext} and \Cref{lem:ass-ext_to_int} refine the correspondence of \Cref{prop_bij_in_ext} to equivalence classes of associative internal and external tensor structures on $\Hbb$.
\end{prop}

This means that the explicit equivalences constructed in the proof of \Cref{prop_bij_in_ext} are associative, in the sense of \Cref{defn:aITS}(2) and \Cref{defn:aETS}(2).
In both cases, only the compatibility of associativity constraints with inverse images plays a role in the proof.

\subsection{Commutativity constraints}

\begin{defn}\label[defn]{defn:cITS}
	Let $(\otimes,m)$ be an internal tensor structure on $\Hbb$, in the sense of \Cref{defn:ITS}.
	\begin{enumerate}
		\item An \textit{internal commutativity constraint} $c$ for $(\otimes,m)$ is the datum of
		\begin{itemize}
			\item for every $S \in \Scal$, a natural isomorphism between functors $\Hbb(S) \times \Hbb(S) \rightarrow \Hbb(S)$
			\begin{equation*}
				c = c_S \colon A \otimes B \xrightarrow{\sim} B \otimes A
			\end{equation*}
		\end{itemize}
		satisfying the following condition
		\begin{enumerate}
			\item[($c$ITS-1)] For every $S \in \Scal$, the diagram of functors $\Hbb(S) \times \Hbb(S) \rightarrow \Hbb(S)$
			\begin{equation*}
				\begin{tikzcd}
					A \otimes B \arrow{r}{c} \arrow{dr}{\id} &
					B \otimes A \arrow{d}{c} \\
					&
					A \otimes B
				\end{tikzcd}
			\end{equation*}
			is commutative.
			\item[($c$ITS-2)] For every morphism $f \colon T \rightarrow S$ in $\Scal$, the diagram of functors $\Hbb(S) \times \Hbb(S) \rightarrow \Hbb(T)$
			\begin{equation*}
				\begin{tikzcd}
					f^* A \otimes f^* B \arrow{r}{c} \arrow{d}{m} &
					f^* B \otimes f^* A \arrow{d}{m} \\
					f^* (A \otimes B)  \arrow{r}{c} &
					f^* (B \otimes A)
				\end{tikzcd}
			\end{equation*}
			is commutative.
		\end{enumerate}
		We say that $(\otimes,m;c)$ is a \textit{symmetric internal tensor structure} on $\Hbb$.
		\item Let $(\otimes,m;c)$ and $(\otimes',m';c')$ be two symmetric internal tensor structures on $\Hbb$.
		An \textit{equivalence of symmetric internal tensor structures} $e \colon (\otimes,m;c) \xrightarrow{\sim} (\otimes',m';c')$ is an equivalence of internal tensor structures $e \colon (\otimes,m) \xrightarrow{\sim} (\otimes',m')$ satisfying the following additional condition:
		\begin{enumerate}
			\item[(eq-$c$ITS)] For every $S \in \Scal$, the diagram of functors $\Hbb(S) \times \Hbb(S) \rightarrow \Hbb(S)$
			\begin{equation*}
				\begin{tikzcd}
					A \otimes B \arrow{r}{e} \arrow{d}{c} &
					A \otimes' B \arrow{d}{c} \\
					B \otimes A \arrow{r}{e} &
					B \otimes' A
				\end{tikzcd}
			\end{equation*}
			is commutative.
		\end{enumerate}
	\end{enumerate}
\end{defn}

\begin{defn}\label[defn]{defn:cETS}
	Let $(\boxtimes,m)$ be an external tensor structure on $\Hbb$, in the sense of \Cref{defn:ETS}.  
	\begin{enumerate}
		\item An \textit{external commutativity constraint} $c$ for $(\boxtimes,m)$ is the datum of
		\begin{itemize}
			\item for every $S_1, S_2 \in \Scal$, a natural isomorphism between functors $\Hbb(S_1) \times \Hbb(S_2) \rightarrow \Hbb(S_1 \times S_2)$
			\begin{equation*}
				c = c_{S_1,S_2} \colon A_1 \boxtimes A_2 \xrightarrow{\sim} \tau^*(A_2 \boxtimes A_1),
			\end{equation*}
			where $\tau = \tau_{(12)}: S_1 \times S_2 \rightarrow S_2 \times S_1$ denotes the permutation isomorphism
		\end{itemize}
		satisfying the following conditions:
		\begin{enumerate}
			\item[($c$ETS-1)]  For every $S_1, S_2 \in \Scal$, the diagram of functors $\Hbb(S_1) \times \Hbb(S_2) \rightarrow \Hbb(S_1 \times S_2)$
			\begin{equation*}
				\begin{tikzcd}
					A_1 \boxtimes A_2 \arrow{r}{c} \arrow[equal]{dr} &
					\tau^*(A_2 \boxtimes A_1) \arrow{d}{c} \\
					&
					\tau^* \tau^*(A_1 \boxtimes A_2)
				\end{tikzcd}
			\end{equation*}
			is commutative.
			\item[($c$ETS-2)] For every choice of of morphisms $f_i \colon T_i \rightarrow S_i$ in $\Scal$, $i = 1,2$, the diagram of functors $\Hbb(S_1) \times \Hbb(S_2) \rightarrow \Hbb(T_2 \times T_1)$
			\begin{equation*}
				\begin{tikzcd}
					f_1^* A_1 \boxtimes f_2^* A_2 \arrow{rr}{c} \arrow{d}{m} &&
					\tau^* (f_2^* A_2 \boxtimes f_1^* A_1) \arrow{d}{m} \\
					(f_1 \times f_2)^* (A_1 \boxtimes A_2)\arrow{r}{c} &
					(f_1 \times f_2)^* \tau^* (A_2 \boxtimes A_1) \arrow[equal]{r} &
					\tau^* (f_2 \times f_1)^* (A_2 \boxtimes A_1)
				\end{tikzcd}
			\end{equation*}
			is commutative.
		\end{enumerate}
		We say that $(\boxtimes,m;c)$ is a \textit{symmetric external tensor structure} on $\Hbb$.
		\item Let $(\boxtimes',m';c')$ and $(\boxtimes,m;c)$ be two symmetric external tensor structures on $\Hbb$.
		An \textit{equivalence of symmetric external tensor structures} $e \colon (\boxtimes',m';c') \xrightarrow{\sim} (\boxtimes,m;c)$ is an equivalence of external tensor structures $e \colon (\boxtimes',m') \xrightarrow{\sim} (\boxtimes,m)$ satisfying the following additional condition:
		\begin{enumerate}
			\item[(eq-$c$ETS)] For every $S_1,S_2 \in \Scal$, the diagram of functors $\Hbb(S_1) \times \Hbb(S_2) \rightarrow \Hbb(S_1 \times S_2)$
			\begin{equation*}
				\begin{tikzcd}
					A_1 \boxtimes' A_2 \arrow{r}{e} \arrow{d}{c'} &
					A_1 \boxtimes A_2 \arrow{d}{c} \\
					\tau^*(A_2 \boxtimes' A_1) \arrow{r}{e}  &
					\tau^*(A_2 \boxtimes A_1)
				\end{tikzcd}
			\end{equation*}
			is commutative.
		\end{enumerate}
	\end{enumerate}
\end{defn}

Just as in the case of associativity constraints, commutativity constraints can be naturally translated between the internal and the external setting:

\begin{lem}\label[lem]{lem:symm_int_to_ext}
	Let $(\otimes,m)$ be an internal tensor structure on $\Hbb$, and let $(\boxtimes,m^e)$ denote the external tensor structure obtained from it via \Cref{lem:int_to_ext}. 
	Suppose that we are given an internal commutativity constraint $c$ for $(\otimes,m)$. 
	Then, associating
	\begin{itemize}
		\item to every $S_1, S_2 \in \Scal$, the natural isomorphism between functors $\Hbb(S_1) \times \Hbb(S_2) \to \Hbb(S_1 \times S_2)$
		\begin{equation*}
			c^e = c^e_{S_1,S_2} \colon A_1 \boxtimes A_2 \xrightarrow{\sim} \tau^*(A_2 \boxtimes A_1)
		\end{equation*}
		defined by taking the composite
		\begin{equation*}
			\begin{tikzcd}[font=\small]
				A_1 \boxtimes A_2 \arrow[equal]{d} &
				\tau^* (A_2 \boxtimes A_1)  \\
			    \pi^{12,*}_{1} A_1 \otimes \pi^{12,*}_{2} A_2 \arrow[equal]{d} &
				\tau^* (\pi^{21,*}_{2} A_2 \otimes \pi^{21,*}_{1} A_1) \arrow[equal]{u}  \\
				\tau^* \pi^{21,*}_{1} A_1 \otimes \tau^* \pi^{21,*}_{2} A_2 \arrow{r}{c} &
				\tau^* \pi^{21,*}_{2} A_2 \otimes \tau^* \pi^{21,*}_{1} A_1 \arrow{u}{m}
			\end{tikzcd}
		\end{equation*}
	\end{itemize}
	defines an external commutativity constraint $c^e$ for $(\boxtimes,m^e)$.
\end{lem}

\begin{lem}\label[lem]{lem:symm_ext_to_int}
	Let $(\boxtimes,m)$ be an external tensor structure on $\Hbb$, and let $(\otimes,m^i)$ denote the internal tensor structure obtained from it via \Cref{lem:ext_to_int}.
	Suppose that we are given an internal commutativity constraint $c$ for $(\boxtimes,m)$.
	Then, associating
	\begin{itemize}
		\item to every $S \in \Scal$, the natural isomorphism of functors $\Hbb(S) \times \Hbb(S) \rightarrow \Hbb(S)$
		\begin{equation*}
			c^i = c^i_S \colon A \otimes B \xrightarrow{\sim} B \otimes A
		\end{equation*}
		defined by taking the composite
		\begin{equation*}
			\begin{tikzcd}[font=\small]
				A \otimes B \arrow[equal]{d} &&
				B \otimes A \arrow[equal]{d} \\
				\Delta_S^*(A \boxtimes B) \arrow{r}{c} &
				\Delta_S^* \tau^*(B \boxtimes A) \arrow[equal]{r} & 
				\Delta_S^*(B \boxtimes A)
			\end{tikzcd}
		\end{equation*}
	\end{itemize}
	defines an internal commutativity constraint $c^i$ for $(\otimes,m^i)$.
\end{lem}

More precisely, axiom ($c$ITS-1) in the internal setting corresponds to axiom ($c$ETS-1) in the external setting, while axiom ($c$ITS-2) corresponds to axiom ($c$ETS-2).

\begin{prop}
	The constructions of \Cref{lem:symm_int_to_ext} and \Cref{lem:symm_ext_to_int} refine the correspondence of \Cref{prop_bij_in_ext} to equivalence classes of symmetric internal and external tensor structures on $\Hbb$.
\end{prop}

This means that the equivalences constructed in the proof of \Cref{prop_bij_in_ext} are symmetric, in the sense of \Cref{defn:cITS}(2) and \Cref{defn:cETS}(2).
As in the case of associativity constraints, only the compatibility with inverse images plays a role in the proof.

\subsection{Unit constraints}

\begin{defn}\label[defn]{defn:uITS}
	Let $(\otimes,m)$ be an internal tensor structure on $\Hbb$, in the sense of \Cref{defn:ITS}.
	\begin{enumerate}
		\item An \textit{internal unit constraint} $u$ for $(\otimes,m)$ is the datum of
		\begin{itemize}
			\item a section $\unit$ of $\Hbb$ over $\Scal$, called the \textit{unit section},
			\item for every $S \in \Scal$, two natural isomorphisms between functors $\Hbb(S) \rightarrow \Hbb(S)$
			\begin{equation*}
				u_r = u_{r,S} \colon A \otimes \unit_S \xrightarrow{\sim} A, \qquad u_l = u_{l,S}: \unit_S \otimes B \xrightarrow{\sim} B
			\end{equation*}
		\end{itemize}
		satisfying the following conditions:
		\begin{enumerate}
			\item[($u$ITS-1)] For every $S \in \Scal$, the diagram in $\Hbb(S)$
			\begin{equation*}
				\begin{tikzcd}
					&
					\unit_S \otimes \unit_S \arrow{dr}{u_r} \arrow{dl}{u_l} \\
					\unit_S \arrow{rr}{\id} &&
					\unit_S
				\end{tikzcd}
			\end{equation*}
			is commutative.
			\item[($u$ITS-2)] For every morphism $f \colon T \rightarrow S$ in $\Scal$, the two diagrams of functors $\Hbb(S) \rightarrow \Hbb(T)$
			\begin{equation*}
				\begin{tikzcd}
					f^* A \otimes f^* \unit_S \arrow{r}{\unit^*} \arrow{d}{m} &
					f^* A \otimes \unit_T \arrow{d}{u_r} \\
					f^*(A \otimes \unit_S) \arrow{r}{u_r}  &
					f^* A
				\end{tikzcd}
				\qquad
				\begin{tikzcd}
					f^* \unit_S \otimes f^* B  \arrow{d}{m} \arrow{r}{\unit^*} &
					\unit_T \otimes f^* B \arrow{d}{u_l} \\
					f^*(\unit_S \otimes B)  \arrow{r}{u_l} &
					f^* B
				\end{tikzcd}
			\end{equation*}
			are commutative.
		\end{enumerate}
		We say that $(\otimes,m;u)$ is a \textit{unitary internal tensor structure} on $\Hbb$.
		\item Let $(\otimes,m;u)$ and $(\otimes',m';u')$ be unitary internal tensor structures on $\Hbb$.
		An \textit{equivalence of unitary internal tensor structures} $e \colon (\otimes,m;u) \xrightarrow{\sim} (\otimes',m';u')$ is the datum of
		\begin{itemize}
			\item an equivalence of internal tensor structures $e \colon (\otimes,m) \xrightarrow{\sim} (\otimes',m')$,
			\item an isomorphism of unit sections $w \colon \unit \xrightarrow{\sim} \unit'$ (in the sense of \Cref{rem:iso-sect})
		\end{itemize}
		satisfying the following additional condition:
		\begin{enumerate}
			\item[(eq-$u$ITS)] For every $S \in \Scal$, the two diagrams of functors $\Hbb(S) \rightarrow \Hbb(S)$
			\begin{equation*}
				\begin{tikzcd}
					A \otimes \unit_S \arrow{r}{u_r} \arrow{d}{w} &
					A \\
					A \otimes \unit'_S \arrow{r}{e} &
					A \otimes' \unit'_S \arrow{u}{u'_r}
				\end{tikzcd}
				\qquad
				\begin{tikzcd}
					\unit_S \otimes B \arrow{r}{u_l} \arrow{d}{w} &
					B \\
					\unit'_S \otimes B \arrow{r}{e} &
					\unit'_S \otimes' B \arrow{u}{u'_l}
				\end{tikzcd}
			\end{equation*}
			are commutative.
		\end{enumerate}
	\end{enumerate}
\end{defn}

\begin{defn}\label[defn]{defn:uETS}
	Let $(\boxtimes,m)$ be an external tensor structure on $\Hbb$, in the sense of \Cref{defn:ETS}.
	\begin{enumerate}
		\item An \textit{external unit constraint} $u$ for $(\boxtimes,m)$ is the datum of
		\begin{itemize}
			\item a section $\unit$ of $\Hbb$ over $\Scal$, called the \textit{unit section},
			\item for every $S_1, S_2 \in \Scal$, a natural isomorphism of functors $\Hbb(S_1) \rightarrow \Hbb(S_1 \times S_2)$
			\begin{equation*}
				u_r = u_{r,S_1,S_2} \colon A_1 \boxtimes \unit_{S_2} \xrightarrow{\sim} \pi^{12,*}_1 A_1
			\end{equation*}
			as well as a natural isomorphism of functors $\Hbb(S_2) \rightarrow \Hbb(S_1 \times S_2)$
			\begin{equation*}
				u_l = u_{l,S_1,S_2} \colon \unit_{S_1} \boxtimes A_2 \xrightarrow{\sim} \pi^{12,*}_2 A_2
			\end{equation*}
		\end{itemize}
		satisfying the following conditions:
		\begin{enumerate}
			\item[($u$ETS-1)] For every $S_1, S_2 \in \Scal$, the diagram in $\Hbb(S_1 \times S_2)$
			\begin{equation*}
				\begin{tikzcd}
					\pi^{12,*}_1 \unit_{S_1} \arrow[bend right]{dr}{\unit^*} &
					\unit_{S_1} \boxtimes \unit_{S_2} \arrow{l}{u_l} \arrow{r}{u_r} &
					\pi^{12,*}_2 \unit_{S_2} \arrow[bend left]{dl}{\unit^*} \\
					&
					\unit_{S_1 \times S_2}
				\end{tikzcd}
			\end{equation*}
			is commutative.
			\item[($u$ETS-2)] For every choice of morphisms $f_i \colon T_i \to S_i$ in $\Scal$, $i = 1,2$, the diagram of functors $\Hbb(S_1) \to \Hbb(T_1 \times T_2)$
			\begin{equation*}
				\begin{tikzcd}
					f_1^* A_1 \boxtimes f_2^* \unit_{S_2} \arrow{r}{\unit^*} \arrow{d}{m} &
					f_1^* A_1 \boxtimes \unit_{T_2} \arrow{r}{u_r} &
					\pi^{12,*}_1 f_1^* A_1 \arrow[equal]{d} \\
					(f_1 \times f_2)^* (A_1 \boxtimes \unit_{S_2}) \arrow{rr}{u_r} &&
					(f_1 \times f_2)^* \pi^{12,*}_1 A_1
				\end{tikzcd}
			\end{equation*}
			and the diagram of functors $\Hbb(S_2) \times \Hbb(T_1 \times T_2)$
			\begin{equation*}
				\begin{tikzcd}
					f_1^* \unit_{S_1} \boxtimes f_2^* A_2 \arrow{r}{\unit^*} \arrow{d}{m} &
					\unit_{T_1} \boxtimes f_2^* A_2 \arrow{r}{u_l} &
					\pi^{12,*}_2 f_2^* A_2 \arrow[equal]{d} \\
					(f_1 \times f_2)^*(\unit_{T_1} \boxtimes A_2) \arrow{rr}{u_l} &&
					(f_1 \times f_2)^* \pi^{12,*}_2 A_2
				\end{tikzcd}
			\end{equation*}
			are commutative.
		\end{enumerate}
		We say that $(\boxtimes,m;u)$ is a \textit{unitary external tensor structure} on $\Hbb$.
		\item Let $(\boxtimes',m';u')$ and $(\boxtimes,m;u)$ be two unitary external tensor structures on $\Hbb$. 
		An \textit{equivalence of unitary external tensor structures} $e \colon (\boxtimes',m';u') \xrightarrow{\sim} (\boxtimes,m;u)$ is the datum of
		\begin{itemize}
			\item an equivalence of external tensor structures $e \colon (\boxtimes',m') \xrightarrow{\sim} (\boxtimes,m)$,
			\item an isomorphism of unit sections $w \colon \unit \xrightarrow{\sim} \unit'$ (in the sense of \Cref{rem:iso-sect})
		\end{itemize}
		satisfying the following additional condition:
		\begin{enumerate}
			\item[(eq-$u$ETS)] For every $S_1, S_2 \in \Scal$, the diagram of functors $\Hbb(S_1) \to \Hbb(S_1 \times S_2)$
			\begin{equation*}
				\begin{tikzcd}
					A_1 \boxtimes' \unit'_{S_2} \arrow{r}{u'_r} &
					\pi^{12,*}_1 A_1 \\
					A_1 \boxtimes' \unit_{S_2} \arrow{u}{w} \arrow{r}{e} &
					A_1 \boxtimes \unit_{S_2} \arrow{u}{u_r}
				\end{tikzcd}
			\end{equation*}
			and the diagram of functors $\Hbb(S_2) \to \Hbb(S_1 \times S_2)$ 
			\begin{equation*}
				\begin{tikzcd}
					\unit'_{S_1} \boxtimes' A_2 \arrow{r}{u'_l} & 
					\pi^{12,*}_2 A_2 \\
					\unit_{S_1} \boxtimes' A_2 \arrow{r}{e} \arrow{u}{w} & \unit_{S_1} \boxtimes A_2 \arrow{u}{u_l}
				\end{tikzcd}
			\end{equation*}
			are commutative.
		\end{enumerate}
	\end{enumerate}
\end{defn}

Again, there is a natural was to translate unit constraints between the internal and the external setting:

\begin{lem}\label[lem]{lem:unit-int_to_ext}
	Let $(\otimes,m)$ be an internal tensor structure on $\Hbb$, and let $(\boxtimes,m^e)$ denote the external tensor structure obtained from it via \Cref{lem:int_to_ext}. 
	Suppose that we are given an internal unit constraint $u$ (with unit section $\unit$) for $(\otimes,m)$. 
	Then, associating
	\begin{itemize}
		\item to every $S_1, S_2 \in \Scal$, the natural isomorphism of functors $\Hbb(S_1) \to \Hbb(S_1 \times S_2)$
		\begin{equation*}
			u_r^e \colon A_1 \boxtimes \unit_{S_2} := \pi^{12,*}_1 A_1 \otimes \pi^{12,*}_2 \unit_{S_2} \xrightarrow{\unit^*} \pi^{12,*}_1 A_1 \otimes \unit_{S_1 \times S_2} \xrightarrow{u_r} \pi^{12,*}_1 A_1 
		\end{equation*}
		and the natural isomorphism of functors $\Hbb(S_2) \rightarrow \Hbb(S_1 \times S_2)$
		\begin{equation*}
			u_l^e \colon \unit_{S_2} \boxtimes A_2 := \pi^{12,*}_1 \unit_{S_1} \otimes \pi^{12,*}_2 A_2 \xrightarrow{\unit^*} \unit_{S_1 \times S_2} \otimes \pi^{12,*}_2 A_2 \xrightarrow{u_l} \pi^{12,*}_2 A_2 
		\end{equation*}
	\end{itemize}
	defines an external unit constraint $u^e$ (with unit section $\unit$) for $(\boxtimes,m^e)$.
\end{lem}

\begin{lem}\label[lem]{lem:unit-ext_to_int}
	Let $(\boxtimes,m)$ be an external tensor structure on $\Hbb$, and let $(\otimes,m^i)$ denote the internal tensor structure obtained from it via \Cref{lem:ext_to_int}. 
	Suppose that we are given an external unit constraint $u$ (with unit section $\unit$) for $(\boxtimes,m)$. 
	Then, associating
	\begin{itemize}
		\item to every $S \in \Scal$, the two natural isomorphisms of functors $\Hbb(S) \rightarrow \Hbb(S)$
		\begin{equation*}
			u_r^i \colon A \otimes \unit_S := \Delta_S^* (A \boxtimes \unit_S) \xrightarrow{u_r} \Delta_S^* \pi^{12,*}_1 A = A, \quad u_l^i \colon \unit_S \otimes A := \Delta_S^* (\unit_S \boxtimes A) \xrightarrow{u_l} \Delta_S^* \pi^{12,*}_2 A = A
		\end{equation*}
	\end{itemize}
	defines an internal unit constraint $u^i$ (with unit section $\unit$) for $(\otimes,m^i)$.
\end{lem}

More precisely, axiom ($u$ITS-1) in the internal setting corresponds to axiom ($u$ETS-1) in the external setting, while axiom ($u$ITS-2) corresponds to axiom ($u$ETS-2).

\begin{prop}
	The constructions of \Cref{lem:unit-int_to_ext} and \Cref{lem:unit-ext_to_int} refine the correspondence of \Cref{prop_bij_in_ext} to equivalence classes of unitary internal and external tensor structures (with unit section $\unit$) on $\Hbb$.
\end{prop}

Again, this means that the equivalences in the proof of \Cref{prop_bij_in_ext} are unitary, in the sense of \Cref{defn:uITS}(2) and \Cref{defn:uETS}(2).
As in the previous cases, only the compatibility with inverse images plays a role in the proof.

\subsection{Compatibility conditions}\label{subsect:comp}

Whenever two among an associativity, a commutativity, and a unit constraint for a monoidal category are given, one should impose a suitable compatibility condition on them (see Mac Lane's discussion in~\cite[\S~VII.2]{Mac71} or~\cite{Mac63}).
Nothing changes in the case of monoidal fibered categories:
the compatibility conditions among the various constraints are imposed fiber-wise, and inverse image functors are not involved.
It turns out that each compatibility condition for the internal tensor product admits a counterpart for the external tensor product.

We start by considering associativity versus commutativity constraints.
\begin{defn}
	Let $(\otimes,m)$ be an internal tensor structure on $\Hbb$, endowed with an internal associativity constraint $a$ and an internal commutativity constraint $c$.
	We say that the constraints $a$ and $c$ are \textit{compatible} if they satisfy the following condition:
	\begin{enumerate}
		\item[($ac$ITS)] For every $S \in \Scal$, the diagram of functors $\Hbb(S) \times \Hbb(S) \times \Hbb(S) \rightarrow \Hbb(S)$
		\begin{equation*}
			\begin{tikzcd}
				(A \otimes B) \otimes C \arrow{r}{c} \arrow{d}{a} &
				(B \otimes A) \otimes C \arrow{r}{a} &
				B \otimes (A \otimes C) \arrow{d}{c} \\
				A \otimes (B \otimes C) \arrow{r}{c} &
				(B \otimes C) \otimes A \arrow{r}{a} &
				B \otimes (C \otimes A)
			\end{tikzcd}
		\end{equation*}
		is commutative -- this is Mac Lane's \textit{hexagon axiom}.
	\end{enumerate}
	In this case, we say that $(\otimes,m;a,c)$ is a \textit{symmetric associative internal tensor structure} on $\Hbb$. 
\end{defn}

\begin{defn}
	Let $(\boxtimes,m)$ be an external tensor structure on $\Hbb$, endowed with an external associativity constraint $a$ and an external commutativity constraint $c$.
	We say that the constraints $a$ and $c$ are \textit{compatible} if they satisfy the following condition:
	\begin{enumerate}
		\item[($ac$ETS)] For every $S_1, S_2, S_3 \in \Scal$, the diagram of functors $\Hbb(S_1) \times \Hbb(S_2) \times \Hbb(S_3) \rightarrow \Hbb(S_1 \times S_2 \times S_3)$
		\begin{equation*}
			\begin{tikzcd}[font=\small]
				(A_1 \boxtimes A_2) \boxtimes A_3 \arrow{r}{c} \arrow{dd}{a} &
				\tau_{(12)}^*(A_2 \boxtimes A_1) \boxtimes A_3 \arrow{r}{m} &
				\tau_{(12)}^*((A_2 \boxtimes A_1) \boxtimes A_3) \arrow{r}{a} &
				\tau_{(12)}^* (A_2 \boxtimes (A_1 \boxtimes A_3)) \arrow{d}{c} \\
				&&&
				\tau_{(12)}^* (A_2 \boxtimes \tau_{(13)}^*(A_3 \boxtimes A_1)) \arrow{d}{m} \\
				A_1 \boxtimes (A_2 \boxtimes A_3) \arrow{r}{c} &
				\tau_{(123)}^*((A_2 \boxtimes A_3) \boxtimes A_1) \arrow{r}{a} & \tau_{(123)}^*(A_2 \boxtimes (A_3 \boxtimes A_1)) \arrow[equal]{r} & \tau_{(12)}^* \tau_{(13)}^*(A_2 \boxtimes (A_3 \boxtimes A_1))
			\end{tikzcd}
		\end{equation*}
		is commutative.
	\end{enumerate}
	In this case, we say that $(\boxtimes,m;a,c)$ is a \textit{symmetric associative external tensor structure} on $\Hbb$.
\end{defn}

\begin{lem}
	Let $(\otimes,m)$ be an internal tensor structure on $\Hbb$, endowed with compatible internal associativity and commutativity constraints $a$ and $c$.
	Then, the external constraints $a^e$ and $c^e$ obtained from them via \Cref{lem:ass-int_to_ext} and \Cref{lem:symm_int_to_ext}, respectively, are compatible as well.
\end{lem}

\begin{lem}
	Let $(\boxtimes,m)$ be an external tensor structure on $\Hbb$, endowed with compatible external associativity and commutativity constraints $a$ and $c$.
	Then, the internal constraints $a^i$ and $c^i$ obtained from them via \Cref{lem:ass-ext_to_int} and \Cref{lem:symm_ext_to_int}, respectively, are compatible as well.
\end{lem}

Next, we turn to associativity versus unit constraints.

\begin{defn}
	Let $(\otimes,m)$ be an internal tensor structure on $\Hbb$, endowed with an internal associativity constraint $a$ and an internal unit constraint $u$ (with unit section $\unit$). 
	We say that the constraints $a$ and $u$ are \textit{compatible} if they satisfy the following condition:
	\begin{enumerate}
		\item[($au$ITS)] For every $S \in \Scal$, the three diagrams of functors $\Hbb(S) \times \Hbb(S) \rightarrow \Hbb(S)$
		\begin{equation*}
			\begin{tikzcd}
				(\unit_S \otimes B) \otimes C \arrow{rr}{a} \arrow{dr}{u_l} && \unit_S \otimes (B \otimes C) \arrow{dl}{u_l} \\
				&
				B \otimes C
			\end{tikzcd}
			\qquad
			\begin{tikzcd}
				(A \otimes \unit_S) \otimes C \arrow{rr}{a} \arrow{dr}{u_r} &&
				A \otimes (\unit_S \otimes C) \arrow{dl}{u_l} \\
				&
				A \otimes C
			\end{tikzcd}
		\end{equation*}
		\begin{equation*}
			\begin{tikzcd}
				(A \otimes B) \otimes \unit_S \arrow{rr}{a} \arrow{dr}{u_r} &&
				A \otimes (B \otimes \unit_S) \arrow{dl}{u_r} \\
				&
				A \otimes B
			\end{tikzcd}
		\end{equation*}
		are commutative.
	\end{enumerate}
	In this case, we say that $(\otimes,m;a,u)$ is a \textit{unitary associative internal tensor structure} on $\Hbb$.
\end{defn}

\begin{defn}
	Let $(\boxtimes,m)$ be an external tensor structure on $\Hbb$, endowed with an external associativity constraint $a$ and an external unit constraint $u$ (with unit section $\unit$). 
	We say that the constraints $a$ and $u$ are \textit{compatible} if they satisfy the following condition:
	\begin{enumerate}
		\item[($au$ETS)] For every $S_1, S_2, S_3 \in \Scal$, the diagram of functors $\Hbb(S_2) \times \Hbb(S_3) \rightarrow \Hbb(S_2 \times S_3)$
		\begin{equation*}
			\begin{tikzcd}
				(\unit_{S_1} \boxtimes A_2) \boxtimes A_3 \arrow{r}{a} \arrow{d}{u_l} &
				\unit_{S_1} \boxtimes (A_2 \boxtimes A_3) \arrow{d}{u_l} \\
				\pi^{12,*}_2 A_2 \boxtimes A_3 \arrow{r}{m} &
				\pi^{123,*}_{23} (A_2 \boxtimes A_3),
			\end{tikzcd}
		\end{equation*}
		the diagram of functors $\Hbb(S_1) \times \Hbb(S_3) \rightarrow \Hbb(S_1 \times S_3)$
		\begin{equation*}
			\begin{tikzcd}
				(A_1 \boxtimes \unit_{S_2}) \boxtimes A_3 \arrow{rr}{a} \arrow{d}{u_r} &&
				A_1 \boxtimes (\unit_{S_2} \boxtimes A_3) \arrow{d}{u_l} \\
				\pi^{12,*}_1 A_1 \boxtimes A_3 \arrow{r}{m} &
				\pi^{123,*}_{13} (A_1 \boxtimes A_3) &
				A_1 \boxtimes \pi^{23,*}_3 A_3 \arrow{l}{m}
			\end{tikzcd}
		\end{equation*}
		and the diagram of functors $\Hbb(S_1) \times \Hbb(S_2) \rightarrow \Hbb(S_1 \times S_2)$
		\begin{equation*}
			\begin{tikzcd}
				(A_1 \boxtimes A_2) \boxtimes \unit_{S_3} \arrow{r}{a} \arrow{d}{u_r} &
				A_1 \boxtimes (A_2 \boxtimes \unit_{S_3}) \arrow{d}{u_r} \\
				\pi^{123,*}_{23} (A_1 \boxtimes A_2) &
				A_1 \boxtimes \pi^{12,*}_{2} A_2 \arrow{l}{m}
			\end{tikzcd}
		\end{equation*}
		are commutative.
	\end{enumerate}
	In this case, we say that $(\boxtimes,m;a,u)$ is a \textit{unitary associative external tensor structure} on $\Hbb$.
\end{defn}

\begin{lem}
	Let $(\otimes,m)$ be an internal tensor structure on $\Hbb$, endowed with compatible internal associativity and unit constraints $a$ and $u$.
	Then, the external constraints $a^e$ and $u^e$ obtained from them via \Cref{lem:ass-int_to_ext} and \Cref{lem:unit-int_to_ext}, respectively, are compatible as well.
\end{lem}

\begin{lem}
	Let $(\boxtimes,m)$ be an external tensor structure on $\Hbb$, endowed with compatible external associativity and unit constraints $a$ and $u$.
	Then, the internal constraints $a^i$ and $u^i$ obtained from them via \Cref{lem:ass-ext_to_int} and \Cref{lem:unit-ext_to_int}, respectively, are compatible as well.
\end{lem}

Lastly, we consider commutativity versus unit constraints.

\begin{defn}
	Let $(\otimes,m)$ be an internal tensor structure on $\Hbb$, endowed with an internal commutativity constraint $c$ and an internal unit constraint $u$ (with unit section $\unit$).
	We say that the constraints $c$ and $u$ are \textit{compatible} if they satisfy the following condition:
	\begin{enumerate}
		\item[($cu$ITS)] For every $S \in \Scal$, the diagram of functors $\Hbb(S) \rightarrow \Hbb(S)$
		\begin{equation*}
			\begin{tikzcd}
				\unit_S \otimes B \arrow{rr}{c} \arrow{dr}{u_l} &&
				B \otimes \unit_S \arrow{dl}{u_r} \\
				&
				B
			\end{tikzcd}
		\end{equation*}
		is commutative.
	\end{enumerate}
	In this case, we say that $(\otimes,m;c,u)$ is a \textit{unitary symmetric internal tensor structure} on $\Hbb$.
\end{defn}

\begin{defn}
	Let $(\boxtimes,m)$ be an external tensor structure on $\Hbb$, endowed with an external commutativity constraint $c$ and an external unit constraint $u$ (with unit section $\unit$).
	We say that the constraints $c$ and $u$ are \textit{compatible} if they satisfy the following condition:
	\begin{enumerate}
		\item[($cu$ETS)] For every $S_1, S_2 \in \Scal$, the diagram of functors $\Hbb(S_2) \rightarrow \Hbb(S_1 \times S_2)$
		\begin{equation*}
			\begin{tikzcd}
				\unit_{S_1} \boxtimes B_2 \arrow{r}{c} \arrow{d}{u_l} &
				\tau^* (B_2 \boxtimes \unit_{S_1}) \arrow{d}{u_r} \\
				\pi^{12,*}_2 B_2 \arrow[equal]{r} &
				\tau^* \pi^{21,*}_2 B_2
			\end{tikzcd}
		\end{equation*}
		is commutative.
	\end{enumerate}
	In this case, we say that $(\boxtimes,m;c,u)$ is a \textit{unitary symmetric external tensor structure} on $\Hbb$.
\end{defn}

\begin{lem}
	Let $(\otimes,m)$ be an internal tensor structure on $\Hbb$, endowed with compatible internal commutativity and unit constraints $c$ and $u$.
	Then, the external constraints $c^e$ and $u^e$ obtained from them via \Cref{lem:symm_int_to_ext} and \Cref{lem:unit-int_to_ext}, respectively, are compatible as well.
\end{lem}

\begin{lem}
	Let $(\boxtimes,m)$ be an external tensor structure on $\Hbb$, endowed with compatible external commutativity and unit constraints $c$ and $u$.
	Then, the internal constraints $c^i$ and $u^i$ obtained from them via \Cref{lem:symm_ext_to_int} and \Cref{lem:unit-ext_to_int}, respectively, are compatible as well.
\end{lem}

\begin{proof}[Proof of \Cref{thm:int_ext}]
	Combine \Cref{prop_bij_in_ext} and \Cref{lem:pf-int_to_ext,lem:pf-ext_to_int} with the results of the present section, according to the additional constraints involved.
\end{proof}

\section{Translating monoidal structures on morphisms}\label{sect:mor}

In this section, we discuss the second part of the correspondence between internal and external monoidal structures, about monoidality of morphisms between monoidal $\Scal$-fibered categories.
It can be conveniently summarized as follows:
\begin{thm}[{\cite[Thm.~12.8]{Shu08}}; {\cite[Thm.~3.10]{MV20}}]\label[thm]{thm:int_ext-mor}
	Let $R \colon \Hbb_1 \to \Hbb_2$ be a morphism of $\Scal$-fibered categories.
	Under the correspondence of \Cref{thm:int_ext}, there exists a canonical bijection between internal and external monoidal structures on $R$, respecting associativity, commutativity, and unit constraints.
\end{thm}

After fixing our conventions on morphisms of fibered categories, we introduce the appropriate notions of internal and external tensor structure on morphisms, and we study their relation.
In the final part, we discuss the associativity, symmetry, and unitarity properties of morphisms in the internal and in the external setting.
The proof of \Cref{thm:int_ext-mor} will consist of the combinations of all these results.

As done in the previous section, we omit all intermediate proofs for sake of brevity.

\subsection{Recollections on morphisms of fibered categories}

Following our conventions on fibered categories, the natural notion of morphism to use is the following:
\begin{defn}\label[defn]{defn:mor-Scal-fib}
	Let $\Hbb_1$ and $\Hbb_2$ be two $\Scal$-fibered categories, in the sense of \Cref{defn:S-fib}.
	A \textit{morphism of $\Scal$-fibered categories} $R \colon \Hbb_1 \rightarrow \Hbb_2$ is the datum of
	\begin{itemize}
		\item for every $S \in \Scal$, a functor 
		\begin{equation*}
			R_S \colon \Hbb_1(S) \rightarrow \Hbb_2(S),
		\end{equation*}
		called the \textit{$R$-realization functor} over $S$,
		\item for every morphism $f \colon T \rightarrow S$ in $\Scal$, a natural isomorphism of functors $\Hbb_1(S) \rightarrow \Hbb_2(T)$
		\begin{equation*}
			\theta = \theta_f \colon f^* R_S(A) \xrightarrow{\sim} R_T (f^* A),
		\end{equation*}
		called the \textit{$R$-transition isomorphism} along $f$
	\end{itemize}
	satisfying the following condition:
	\begin{enumerate}
		\item[(mor-fib)] For every pair of composable morphisms $f \colon T \rightarrow S$, $g \colon S \rightarrow V$ in $\Scal$, the diagram of functors $\Hbb_1(V) \rightarrow \Hbb_2(T)$
		\begin{equation*}
			\begin{tikzcd}
				(gf)^* R_V(A) \arrow{rr}{\theta_{gf}} \arrow{d}{\conn^{(2)}_{f,g}} &&
				R_T ((gf)^* A) \arrow{d}{\conn^{(1)}_{f,g}} \\
				f^* g^* R_V(A) \arrow{r}{\theta_g} &
				f^* R_S (g^* A) \arrow{r}{\theta_f} &
				R_T (f^* g^* A)
			\end{tikzcd}
		\end{equation*}
		is commutative.
	\end{enumerate}
\end{defn}

\begin{rem}\label[rem]{rem:mor}
	\begin{enumerate}
		\item In view of axiom (fib-0) from \Cref{defn:S-fib}, axiom (mor-fib) implies that the natural isomorphism of functors $\Hbb_1(S) \rightarrow \Hbb_2(S)$
		\begin{equation*}
			\theta_{\id_S} \colon R_S(A) \xrightarrow{\sim} R_S(A)
		\end{equation*}
		is the identity for every $S \in \Scal$.
		\item The significance of axiom (mor-fib) is to ensure that any two combinations of inverse image functors and realization functors cannot get identified in more than one way.
	\end{enumerate}
\end{rem}

\begin{defn}
	Let $R \colon \Hbb_1 \to \Hbb_2$ be a morphism of $\Scal$-fibered categories, and let $A$ be a section of $\Hbb_1$ over $\Scal$ in the sense of \Cref{defn:sect}.
	We call the \textit{image} of $A$ under $R$ the section $R(A)$ of $\Hbb_2$ over $\Scal$ defined as follows:
	\begin{itemize}
		\item for every $S \in \Scal$, the object $R(A)_S \in \Hbb_2(S)$ is just $R_S(A_S)$;
		\item for every morphism $f \colon T \rightarrow S$ in $\Scal$, the isomorphism in $\Hbb_2(T)$
		\begin{equation*}
			R(A)_f^* \colon f^* R(A)_S \xrightarrow{\sim} R(A)_T
		\end{equation*}
		is the composite
		\begin{equation*}
			f^* R_S(A_S) \xrightarrow{\theta_f} R_T(f^* A_S) \xrightarrow{A_f^*} R_T(A_T).
		\end{equation*}
	\end{itemize}
	It is easy to check that $R(A)$ satisfies condition (sect) from \Cref{defn:sect} as soon as $A$ does.
\end{defn}

\subsection{Internal and external tensor structures on morphisms}

From now until the end of this section, we fix two $\Scal$-fibered categories $\Hbb_1$ and $\Hbb_2$ as well as a morphism $R \colon \Hbb_1 \to \Hbb_2$.

\begin{nota}\label[nota]{nota:mor-when-tens}
	Following \Cref{nota:fib-when-tens}, from now on we systematically write $R$-transition isomorphisms as unmarked equalities:
	this is harmless, as explained in \Cref{rem:mor}(2).
\end{nota}

As done in \Cref{subsect:tens} for monoidal structures on fibered categories, it is convenient to start by considering the minimal monoidality conditions that one might require on a morphism.
In detail:

\begin{defn}
	Let $(\otimes_1,m_1)$ and $(\otimes_2,m_2)$ be internal tensor structures on $\Hbb_1$ and $\Hbb_2$, respectively.
	An \textit{internal tensor structure} $\rho$ on $R$ (with respect to $(\otimes_1,m_1)$ and $(\otimes_2,m_2)$) is the datum of
	\begin{itemize}
		\item for every $S \in \Scal$, a natural isomorphism of functors $\Hbb_1(S) \times \Hbb_1(S) \rightarrow \Hbb_2(S)$
		\begin{equation*}
			\rho = \rho_S \colon R_S(A) \otimes_2 R_S(B) \xrightarrow{\sim} R_S(A \otimes_1 B),
		\end{equation*}
		called the \textit{internal $R$-monoidality isomorphism} at $S$
	\end{itemize}
	satisfying the following condition:
	\begin{enumerate}
		\item[(mor-ITS)] For every morphism $f \colon T \rightarrow S$ in $\Scal$, the diagram of functors $\Hbb_1(S) \times \Hbb_1(S) \rightarrow \Hbb_2(T)$
		\begin{equation*}
			\begin{tikzcd}
				f^* R_S(A) \otimes_2 f^* R_S(B) \arrow{r}{m_2} \arrow{d}{\theta} &
				f^* (R_S(A) \otimes_2 R_S(B)) \arrow{r}{\rho} &
				f^* R_S(A \otimes_1 B) \arrow{d}{\theta} \\
				R_T(f^* A) \otimes_2 R_T(f^* B) \arrow{r}{\rho} &
				R_T(f^* A \otimes_1 f^* B) \arrow{r}{m_1} &
				R_T(f^*(A \otimes_1 B))
			\end{tikzcd}
		\end{equation*}
		is commutative.
	\end{enumerate}
\end{defn}

\begin{defn}
	Let $(\boxtimes_1,m_1)$ and $(\boxtimes_2,m_2)$ be external tensor structures on $\Hbb_1$ and $\Hbb_2$, respectively.
	An \textit{external tensor structure} $\rho$ on $R$ (with respect to $(\boxtimes_1,m_1)$ and $(\boxtimes_2,m_2)$) is the datum of
	\begin{itemize}
		\item for every $S_1, S_2 \in \Scal$, a natural isomorphism of functors $\Hbb_1(S_1) \times \Hbb_1(S_2) \rightarrow \Hbb_2(S_1 \times S_2)$
		\begin{equation*}
			\rho = \rho_{S_1,S_2} \colon R_{S_1}(A_1) \boxtimes_2 R_{S_2}(A_2) \xrightarrow{\sim} R_{S_1 \times S_2}(A_1 \boxtimes_1 A_2),
		\end{equation*}
		called the \textit{external $R$-monoidality isomorphism} at $(S_1,S_2)$,
	\end{itemize}
	satisfying the following condition:
	\begin{enumerate}
		\item[(mor-ETS)] For every two morphisms $f_i \colon T_i \rightarrow S_i$ in $\Scal$, $i = 1,2$, the diagram of functors $\Hbb_1(S_1) \times \Hbb_1(S_2) \rightarrow \Hbb_2(T_1 \times T_2)$
		\begin{equation*}
			\begin{tikzcd}
				f_1^* R_{S_1}(A_1) \boxtimes_2 f_2^* R_{S_2}(A_2) \arrow{r}{m_2} \arrow{d}{\theta} &
				(f_1 \times f_2)^* (R_{S_1}(A_1) \boxtimes_2 R_{S_2}(A_2)) \arrow{r}{\rho} &
				(f_1 \times f_2)^* R_{S_1 \times S_2}(A_1 \boxtimes_1 A_2) \arrow{d}{\theta} \\
				R_{T_1}(f_1^* A_1) \boxtimes_2 R_{T_2}(f_2^* A_2) \arrow{r}{\rho} &
				R_{T_1 \times T_2}(f_1^* A_1 \boxtimes_1 f_2^* A_2) \arrow{r}{m_1} &
				R_{T_1 \times T_2}((f_1 \times f_2)^*(A_1 \boxtimes_1 A_2))
			\end{tikzcd}
		\end{equation*}
		is commutative.
	\end{enumerate}
\end{defn}

The translation between internal and external tensor structures extends to morphisms, as follows:
\begin{lem}\label[lem]{lem:mor-int_to_ext}
	For $j = 1,2$, let $(\otimes_j,m_j)$ be an internal tensor structure on $\Hbb_j$, and let $(\boxtimes_j,m_j^e)$ denote the external tensor structure obtained from it via \Cref{lem:int_to_ext}.
	Suppose that we are given an internal tensor structure $\rho$ on $R$ (with respect to $(\otimes_1,m_1)$ and $(\otimes_2,m_2)$). 
	Then, associating
	\begin{itemize}
		\item to every $S_1, S_2 \in \Scal$, the natural isomorphism of functors $\Hbb_1(S_1) \times \Hbb_1(S_2) \rightarrow \Hbb_2(S_1 \times S_2)$
		\begin{equation*}
			\rho^e = \rho^e_{S_1,S_2} \colon R_{S_1}(A_1) \boxtimes_2 R_{S_2}(A_2) \xrightarrow{\sim} R_{S_1 \times S_2}(A_1 \boxtimes_1 A_2)
		\end{equation*}
		defined by taking the composite
		\begin{equation*}
			\begin{tikzcd}[font=\small]
				R_{S_1}(A_1) \boxtimes_2 R_{S_2}(A_2) \arrow[equal]{d} &&
				R_{S_1 \times S_2}(A_1 \boxtimes_1 A_2) \arrow[equal]{d} \\
				\pi^{12,*}_1 R_{S_1}(A_1) \otimes_2 \pi^{12,*}_2 R_{S_2}(A_2) \arrow{r}{\theta} & 
				R_{S_1 \times S_2}(\pi^{12,*}_1 A_1) \otimes_2 R_{S_1 \times S_2}(\pi^{12,*}_2 A_2) \arrow{r}{\rho} &
				R_{S_1 \times S_2}(\pi^{12,*}_1 A_1 \otimes_1 \pi^{12,*}_2 A_2)
			\end{tikzcd}
		\end{equation*}
	\end{itemize}
	defines an external tensor structure $\rho^e$ on $R$ (with respect to $(\boxtimes_1,m_1^e)$ and $(\boxtimes_2,m_2^e)$).
\end{lem}

\begin{lem}\label[lem]{lem:mor-ext_to_int}
	For $j = 1,2$, let $(\boxtimes_j,m_j)$ be an external tensor structure on $\Hbb_j$, and let $(\otimes_j,m_j^i)$ denote the internal tensor structure obtained from it via \Cref{lem:ext_to_int}.
	Suppose that we are given an external tensor structure $\rho$ on $R$ (with respect to $(\boxtimes_1,m_1)$ and $(\boxtimes_2,m_2)$). 
	Then, associating
	\begin{itemize}
		\item to every $S \in \Scal$, the natural isomorphism of functors $\Hbb_1(S) \times \Hbb_1(S) \rightarrow \Hbb_2(S)$
		\begin{equation*}
			\rho^i = \rho^i_S \colon R_S(A) \otimes_2 R_S(B) \xrightarrow{\sim} R_S(A \otimes_1 B)
		\end{equation*}
		defined by taking the composite
		\begin{equation*}
			\begin{tikzcd}[font=\small]
				R_S(A) \otimes_2 R_S(B) \arrow[equal]{d} && 
				R_S(A \otimes_1 B) \arrow[equal]{d} \\
				\Delta_S^* (R_S(A) \boxtimes_2 R_S(B)) \arrow{r}{\rho} & 
				\Delta_S^* R_S(A \boxtimes_1 B) \arrow{r}{\theta} & 
				R_S(\Delta_S^* (A \boxtimes_1 B))
			\end{tikzcd}
		\end{equation*}
	\end{itemize}
	defines an internal tensor structure $\rho^i$ on $R$ (with respect to $(\otimes_1,m_1^i)$ and $(\otimes_2,m_2^i)$).
\end{lem}

Similarly, the correspondence between internal and external tensor structures extends to morphisms.
Here is a suggestive way to state this result:
\begin{prop}\label[prop]{prop_int_ext_morph}
	For $j = 1,2$, let $(\otimes_j,m_j^i)$ and $(\boxtimes_j,m_j^e)$ be internal and external tensor structures on $\Hbb_j$ corresponding to one another under \Cref{prop_bij_in_ext}.
	Then, the constructions of \Cref{lem:mor-int_to_ext} and \Cref{lem:mor-ext_to_int} define mutually inverse bijections between internal tensor structures on $R$ (with respect to $(\otimes_1,m_1^i)$ and $(\otimes_2,m_2^i)$) and external tensor structures on $R$ (with respect to $(\boxtimes_1,m_1^e)$ and $(\boxtimes_2,m_2^e)$).
\end{prop}

More precisely, the constructions of \Cref{lem:mor-int_to_ext} and \Cref{lem:mor-ext_to_int} are compatible with the equivalences constructed in the proof of \Cref{prop_bij_in_ext} over $\Hbb_1$ and $\Hbb_2$:
with the obvious notation, the natural diagrams of the form
\begin{equation*}
	\begin{tikzcd}
		R_S(A) \otimes_2 R_S(B) \arrow{rr}{\rho} \arrow{d}{e_2} &&
		R_S(A \otimes_1 B) \arrow{d}{e_1} \\
		R_S(A) \otimes'_2 R_S(B) \arrow{rr}{\rho'} &&
		R_S(A \otimes'_1 B)
	\end{tikzcd}
\end{equation*} 
as well as the natural diagrams of the form
\begin{equation*}
	\begin{tikzcd}
		R_{S_1}(A_1) \boxtimes'_2 R_{S_2}(A_2) \arrow{rr}{\rho'} \arrow{d}{e_2} &&
		R_{S_1 \times S_2}(A_1 \boxtimes'_1 A_2) \arrow{d}{e_1} \\
		R_{S_1}(A_1) \boxtimes_2 R_{S_2}(A_2) \arrow{rr}{\rho} &&
		R_{S_1 \times S_2}(A_1 \boxtimes_1 A_2)
	\end{tikzcd}
\end{equation*}
are all commutative.

\begin{rem}
	Let us rephrase the above discussion in the language of Grothendieck fibrations, along the lines of \Cref{rem:total-otimes}.
	For $j = 1,2$, let $\Hcal_j$ denote the Grothendieck construction of $\Hbb_j$.
	The morphism $R \colon \Hbb_1 \to \Hbb_2$ corresponds to the functor
	\begin{equation*}
		\mathsf{R} \colon \Hcal_1 \to \Hcal_2, \quad (S,A) \mapsto (S,R_S(A)).
	\end{equation*}
	Given an external tensor structure $(\boxtimes_j,m_j)$ on $\Hbb_j$, let $- \circledast_j - \colon \Hcal_j \times \Hcal_j \to \Hcal_j$ denote the corresponding bi-functor.
	If $\rho$ is an external tensor structure on $R$ (with respect to $(\boxtimes_1,m_1)$ and $(\boxtimes_2,m_2)$), one can construct a natural isomorphism between bi-functors $\Hcal_1 \times \Hcal_1 \to \Hcal_2$
	\begin{equation*}
		\rho^{\circledast} \colon \mathsf{R}(S_1,A_1) \circledast_2 \mathsf{R}(S_2,A_2) \xrightarrow{\sim} \mathsf{R}((S_1,A_1) \circledast_1 (S_2,A_2))
	\end{equation*}
	lying over the identity natural transformation of $\id_{\Scal \times \Scal}$:
	it is defined by the formula
	\begin{equation*}
		\begin{tikzcd}
			\mathsf{R}(S_1,A_1) \circledast_2 \mathsf{R}(S_2,A_2) \arrow[equal]{d} &&
			\mathsf{R}((S_1,A_1) \circledast_1 (S_2,A_2)), \arrow[equal]{d} \\
			(S_1, R_{S_1}(A_1)) \circledast_2 (S_2, R_{S_2}(A_2)) \arrow[equal]{d} &&
			\mathsf{R}(S_1 \times S_2, A_1 \boxtimes_1 A_2) \arrow[equal]{d} \\
			(S_1 \times S_2, R_{S_1}(A_1) \boxtimes_2 R_{S_2}(A_2)) \arrow{rr}{\rho_{S_1,S_2}} &&
			(S_1 \times S_2, R_{S_1 \times S_2}(A_1 \boxtimes_1 A_2))
		\end{tikzcd}
	\end{equation*}
	and naturality follows from axiom (mor-ETS).
	The external tensor structure $\rho$ can be canonically recovered from this natural isomorphism:
	the $R$-monoidality isomorphism $\rho_{S_1,S_2}$ is obtained by restriction to the fibre $(\Hcal_1 \times \Hcal_1)_{(S_1,S_2)} = \Hbb_1(S_1) \times \Hbb_1(S_2)$.
	See the proof of \cite[Thm.~12.8]{Shu08} for a more thorough discussion.
\end{rem}

\subsection{Associative, symmetric, and unitary morphisms}\label{subsect:acu-mor}

To conclude this section, we discuss further properties of monoidal structures on morphisms in the presence of associativity, commutativity, or unit constraints (as formalized in \Cref{sect:acu}).

We start from the case of associativity constraints.

\begin{defn}
	Let $(\otimes_1,m_1)$ and $(\otimes_2,m_2)$ be internal tensor structures on $\Hbb_1$ and $\Hbb_2$, respectively, and let $\rho$ be an internal tensor structure on $R$ (with respect to $(\otimes_1,m_1)$ and $(\otimes_2,m_2)$).
	Suppose that we are given internal associativity constraints $a_1$ and $a_2$ for $(\otimes_1,m_1)$ and $(\otimes_2,m_2)$, respectively. 
	We say that $\rho$ is \textit{associative} (with respect to $a_1$ and $a_2$) if it satisfies the following additional condition:
	\begin{enumerate}
		\item[(mor-$a$ITS)] For every $S \in \Scal$, the diagram of functors $\Hbb_1(S) \times \Hbb_1(S) \times \Hbb_1(S) \rightarrow \Hbb_2(S)$
		\begin{equation*}
			\begin{tikzcd}
				(R_S(A) \otimes_2 R_S(B)) \otimes_2 R_S(C) \arrow{r}{\rho} \arrow{d}{a_2} &
				R_S(A \otimes_1 B) \otimes_2 R_S(C) \arrow{r}{\rho} &
				R_S((A \otimes_1 B) \otimes_1 C) \arrow{d}{a_1} \\
				R_S(A) \otimes_2 (R_S(B) \otimes_2 R_S(C)) \arrow{r}{\rho} &
				R_S(A) \otimes_2 R_S(B \otimes_1 C) \arrow{r}{\rho} &
				R_S(A \otimes_1 (B \otimes_1 C))
			\end{tikzcd}
		\end{equation*}
		is commutative.
	\end{enumerate}
\end{defn}

\begin{defn}
	Let $(\boxtimes_1,m_1)$ and $(\boxtimes_2,m_2)$ be external tensor structures on $\Hbb_1$ and $\Hbb_2$, respectively, and let $\rho$ be an external tensor structure on $R$ (with respect to $(\boxtimes_1,m_1)$ and $(\boxtimes_2,m_2)$).
	Suppose that we are given external associativity constraints $a_1$ and $a_2$ for $(\boxtimes_1,m_1)$ and $(\boxtimes_2,m_2)$, respectively. 
	We say that $\rho$ is \textit{associative} (with respect to $a_1$ and $a_2$) if it satisfies the following additional condition:
	\begin{enumerate}
		\item[(mor-$a$ETS)] For every $S_1, S_2, S_3 \in \Scal$, the diagram of functors $\Hbb_1(S_1) \times \Hbb_1(S_2) \times \Hbb_1(S_3) \rightarrow \Hbb_2(S_1 \times S_2 \times S_3)$
		\begin{equation*}
			\begin{tikzcd}
				(R_{S_1}(A_1) \boxtimes R_{S_2}(A_2)) \boxtimes R_{S_3}(A_3) \arrow{r}{\rho} \arrow{d}{a_2} &
				R_{S_1 \times S_2}(A_1 \boxtimes A_2) \boxtimes R_{S_3}(A_3) \arrow{r}{\rho} &
				R_{S_1 \times S_2 \times S_3}((A_1 \boxtimes A_2) \boxtimes A_3) \arrow{d}{a_1} \\
				R_{S_1}(A_1) \boxtimes (R_{S_2}(A_2) \boxtimes R_{S_3}(A_3)) \arrow{r}{\rho} &
				R_{S_1}(A_1) \boxtimes R_{S_2 \times S_2}(A_2 \boxtimes A_3) \arrow{r}{\rho} &
				R_{S_1 \times S_2 \times S_3}(A_1 \boxtimes (A_2 \boxtimes A_3))
			\end{tikzcd}
		\end{equation*}
		is commutative.
	\end{enumerate}
\end{defn}

\begin{lem}
	For $j = 1,2$ let $(\otimes_j,m_j;a_j)$ be an associative internal tensor structure on $\Hbb_j$, and let $(\boxtimes_j,m_j^e;a_j^e)$ denote the associative external tensor structure obtained from it via \Cref{lem:int_to_ext} and \Cref{lem:ass-int_to_ext}.
	Moreover, let $\rho$ be an associative internal tensor structure on $R$.
	Then, the external tensor structure $\rho^e$ obtained from it via \Cref{lem:mor-int_to_ext} is associative as well.
\end{lem}

\begin{lem}
	For $j = 1,2$ let $(\boxtimes_j,m_j;a_j)$ be an associative external tensor structure on $\Hbb_j$, and let $(\otimes_j,m_j^i;a_j^i)$ denote the associative internal tensor structure obtained from it via \Cref{lem:ext_to_int} and \Cref{lem:ass-ext_to_int}.
	Moreover, let $\rho$ be an associative external tensor structure on $R$.
	Then, the internal tensor structure $\rho^i$ obtained from it via \Cref{lem:mor-ext_to_int} is associative as well.
\end{lem}

Next, we consider the case of commutativity constraints.

\begin{defn}
	Let $(\otimes_1,m_1)$ and $(\otimes_2,m_2)$ be internal tensor structures on $\Hbb_1$ and $\Hbb_2$, respectively, and let $\rho$ be an internal tensor structure on $R$ (with respect to $(\otimes_1,m_1)$ and $(\otimes_2,m_2)$).
	Suppose that we are given internal commutativity constraints $c_1$ and $c_2$ for $(\otimes_1,m_1)$ and $(\otimes_2,m_2)$, respectively. 
	We say that $\rho$ is \textit{symmetric} (with respect to $c_1$ and $c_2$) if it satisfies the following additional condition:
	\begin{enumerate}
		\item[(mor-$c$ITS)] For every $S \in \Scal$, the diagram of functors $\Hbb_1(S) \times \Hbb_1(S) \rightarrow \Hbb_2(S)$
		\begin{equation*}
			\begin{tikzcd}
				R_S(A) \otimes_2 R_S(B) \arrow{r}{\rho} \arrow{d}{c_2} &
				R_S(A \otimes_1 B) \arrow{d}{c_1} \\
				R_S(B) \otimes_2 R_S(A) \arrow{r}{\rho} &
				R_S(B \otimes_1 A)
			\end{tikzcd}
		\end{equation*}
		is commutative.
	\end{enumerate}
\end{defn}

\begin{defn}
	Let $(\boxtimes_1,m_1)$ and $(\boxtimes_2,m_2)$ be external tensor structures on $\Hbb_1$ and $\Hbb_2$, respectively, and let $\rho$ be an external tensor structure on $R$ (with respect to $(\boxtimes_1,m_1)$ and $(\boxtimes_2,m_2)$).
	Suppose that we are given external commutativity constraints $c_1$ and $c_2$ for $(\boxtimes_1,m_1)$ and $(\boxtimes_2,m_2)$, respectively. 
	We say that $\rho$ is \textit{symmetric} (with respect to $c_1$ and $c_2$) if it satisfies the following additional condition:
	\begin{enumerate}
		\item[(mor-$c$ETS)] For every $S_1, S_2 \in \Scal$, the diagram of functors $\Hbb_1(S_1) \times \Hbb_1(S_2) \rightarrow \Hbb_2(S_1 \times S_2)$
		\begin{equation*}
			\begin{tikzcd}
				R_{S_1}(A_1) \boxtimes_2 R_{S_2}(A_2) \arrow{rr}{\rho} \arrow{d}{c_2} &&
				R_{S_1 \times S_2}(A_1 \boxtimes_1 A_2) \arrow{d}{c_1} \\
				\tau^*(R_{S_2}(A_2) \boxtimes_2 R_{S_1}(A_1)) \arrow{r}{\rho} &
				\tau^* R_{S_2 \times S_1}(A_2 \boxtimes_1 A_1) \arrow{r}{\theta} &
				R_{S_1 \times S_2} \tau^*(A_2 \boxtimes_1 A_1)
			\end{tikzcd}
		\end{equation*}
		is commutative.
	\end{enumerate}
\end{defn}

\begin{lem}
	For $j = 1,2$ let $(\otimes_j,m_j;c_j)$ be a symmetric internal tensor structure on $\Hbb_j$, and let $(\boxtimes_j,m_j^e;c_j^e)$ denote the symmetric external tensor structure obtained from it via \Cref{lem:int_to_ext} and \Cref{lem:symm_int_to_ext}.
	Moreover, let $\rho$ be a symmetric internal tensor structure on $R$.
	Then, the external tensor structure $\rho^e$ obtained from it via \Cref{lem:mor-int_to_ext} is symmetric as well.
\end{lem}

\begin{lem}
	For $j = 1,2$ let $(\boxtimes_j,m_j;c_j)$ be a symmetric external tensor structure on $\Hbb_j$, and let $(\otimes_j,m_j^i;c_j^i)$ denote the symmetric internal tensor structure obtained from it via \Cref{lem:ext_to_int} and \Cref{lem:symm_ext_to_int}.
	Moreover, let $\rho$ be a symmetric external tensor structure on $R$.
	Then, the internal tensor structure $\rho^i$ obtained from it via \Cref{lem:mor-ext_to_int} is symmetric as well.
\end{lem}

Lastly, we treat the case of unit constraints.

\begin{defn}
	Let $(\otimes_1,m_1)$ and $(\otimes_2,m_2)$ be internal tensor structures on $\Hbb_1$ and $\Hbb_2$, respectively, and let $\rho$ be an internal tensor structure on $R$ (with respect to $(\otimes_1,m_1)$ and $(\otimes_2,m_2)$).
	Suppose that we are given internal unit constraints $u_1$ and $u_2$ (with unit sections $\unit^{(1)}$ and $\unit^{(2)}$) for $(\otimes_1,m_1)$ and $(\otimes_2,m_2)$, respectively, as well as an isomorphism of sections $w \colon R(\unit^{(1)}) \xrightarrow{\sim} \unit^{(2)}$.
	We say that $\rho$ is \textit{unitary} (via $w$, and with respect to $u_1$ and $u_2$) if it satisfies the following additional condition:
	\begin{enumerate}
		\item[(mor-$u$ITS)] For every $S \in \Scal$, the two diagrams of functors $\Hbb_1(S) \rightarrow \Hbb_2(S)$
		\begin{equation*}
			\begin{tikzcd}
				R_S(A) \otimes_2 R_S(\unit^{(1)}_S) \arrow{r}{w_S} \arrow{d}{\rho} &
				R_S(A) \otimes_2 \unit^{(2)}_S \arrow{d}{u_{2,r}} \\
				R_S(A \otimes_1 \unit^{(1)}_S) \arrow{r}{u_{1,r}} &
				R_S(A)
			\end{tikzcd}
			\qquad
			\begin{tikzcd}
				R_S(\unit^{(1)}_S) \otimes_2 R_S(B) \arrow{r}{w_S} \arrow{d}{\rho} &
				\unit^{(2)}_S \otimes_2 R_S(B) \arrow{d}{u_{2,l}} \\
				R_S(\unit^{(1)}_S \otimes_1 B) \arrow{r}{u_{1,l}} &
				R_S(B)
			\end{tikzcd}
		\end{equation*}
		are commutative.
	\end{enumerate}
\end{defn}

\begin{defn}
	Let $(\boxtimes_1,m_1)$ and $(\boxtimes_2,m_2)$ be external tensor structures on $\Hbb_1$ and $\Hbb_2$, respectively, and let $\rho$ be an external tensor structure on $R$ (with respect to $(\boxtimes_1,m_1)$ and $(\boxtimes_2,m_2)$).
	Suppose that we are given external unit constraints $u_1$ and $u_2$ be external unit constraints (with unit sections $\unit^{(1)}$ and $\unit^{(2)}$) on $(\boxtimes_1,m_1)$ and $(\boxtimes_2,m_2)$, respectively, as well as an isomorphism of sections $w \colon R(\unit^{(1)}) \xrightarrow{\sim} \unit^{(2)}$. 
	We say that $\rho$ is \textit{unitary} (via $w$, and with respect to $u_1$ and $u_2$) if it satisfies the following additional condition:
	\begin{enumerate}
		\item[(mor-$u$ETS)] For every $S_1, S_2 \in \Scal$, the diagram of functors $\Hbb_1(S_1) \to \Hbb_2(S_1 \times S_2)$
		\begin{equation*}
			\begin{tikzcd}
				R_{S_1}(A_1) \boxtimes_2 R_{S_2}(\unit^{(1)}_{S_2}) \arrow{r}{w} \arrow{d}{\rho} & 
				R_{S_1}(A_1) \boxtimes_2 \unit^{(2)}_{S_2} \arrow{r}{u_{2,r}} & 
				\pi^{12,*}_1 R_{S_1}(A_1) \arrow{d}{\theta} \\
				R_{S_1 \times S_2}(A_1 \boxtimes_1 \unit^{(1)}_{S_2}) \arrow{rr}{u_{1,r}} && 
				R_{S_1 \times S_2}(\pi^{12,*}_1 A_1)
			\end{tikzcd}
		\end{equation*}
		and the diagram of functors $\Hbb_1(S_2) \to \Hbb_2(S_1 \times S_2)$
		\begin{equation*}
			\begin{tikzcd}
				R_{S_1}(\unit^{(1)}_{S_1}) \boxtimes_2 R_{S_2}(A_2) \arrow{r}{w} \arrow{d}{\rho} &
				\unit^{(2)}_{S_1} \boxtimes_2 R_{S_2}(A_2) \arrow{r}{u_{2,l}} & 
				\pi^{12,*}_2 R_{S_2}(A_2) \arrow{d}{\theta} \\
				R_{S_1 \times S_2}(\unit^{(1)}_{S_1} \boxtimes_1 A_2) \arrow{rr}{u_{1,l}} && 
				R_{S_1 \times S_2}(\pi^{12,*}_2 A_2)
			\end{tikzcd}
		\end{equation*}
		are commutative.
	\end{enumerate}
\end{defn}

\begin{lem}
	For $j = 1,2$ let $(\otimes_j,m_j;u_j)$ be a unitary internal tensor structure on $\Hbb_j$ with unit section $\unit^{(j)}$, and let $(\boxtimes_j,m_j^e;u_j^e)$ denote the unitary external tensor structure obtained from it via \Cref{lem:int_to_ext} and \Cref{lem:unit-int_to_ext}.
	Moreover, let $\rho$ be a unitary internal tensor structure on $R$. 
	Then, the external tensor structure $\rho^e$ obtained from it via \Cref{lem:mor-int_to_ext} is unitary as well.
\end{lem}

\begin{lem}
	For $j = 1,2$ let $(\boxtimes_j,m_j;u_j)$ be a unitary external tensor structure on $\Hbb_j$ with unit section $\unit^{(j)}$, and let $(\otimes_j,m_j^i;u_j^i)$ denote the unitary internal tensor structure obtained from it via \Cref{lem:ext_to_int} and \Cref{lem:unit-ext_to_int}.
	Moreover, let $\rho$ be a unitary external tensor structure on $R$.
	Then, the internal tensor structure $\rho^i$ obtained from it via \Cref{lem:mor-ext_to_int} is unitary as well.
\end{lem}

\begin{proof}[Proof of \Cref{thm:int_ext-mor}]
	Combine \Cref{prop_int_ext_morph} with the complementary results collected in the last subsection, according to the auxiliary constraints involved.
\end{proof}

\section{The case of triangulated monoidal fibered categories}\label{sect:tri-int_ext}

Starting from this section, we specialize the correspondence between internal and external monoidal structures to the setting of triangulated fibered categories.
Our first main result is the triangulated version of \Cref{thm:int_ext}:
\begin{thm}\label[thm]{thm:tri-int_ext}
	Let $\Hbb$ be a triangulated monoidal $\Scal$-fibered category.
	The bijection of \Cref{thm:int_ext} naturally refines to a canonical bijection between triangulated equivalence classes of triangulated internal and external monoidal structures on $\Hbb$.
\end{thm}

In the present section, after fixing our general conventions on triangulated fibered categories, we introduce the notions of triangulated internal and external tensor structures as well as the appropriate notions of equivalence.
Following the route of \Cref{sect:int_ext}, we establish the triangulated analogue of \Cref{prop_bij_in_ext}.
The study of associativity, commutativity, and unit constraints in the triangulated setting is deferred to \Cref{sect:tri-acu}.

\subsection{Recollections on triangulated fibered categories}

In agreement with the conventions of \Cref{subsect:rec_fib}, we regard triangulated fibered categories over $\Scal$ as contravariant $2$-functors with values in the $2$-category of triangulated categories.
In detail:
\begin{defn}\label[defn]{defn:triS-fib}
	A \textit{triangulated $\Scal$-fibered category} $\Hbb$ is the datum of
	\begin{itemize}
		\item for every $S \in \Scal$, a triangulated category $\Hbb(S)$,
		\item for every morphism $f \colon T \rightarrow S$ in $\Scal$, a triangulated functor
		\begin{equation*}
			f^* \colon \Hbb(S) \rightarrow \Hbb(T),
		\end{equation*}  
		\item for every pair of composable morphisms $f \colon T \rightarrow S$, $g \colon S \rightarrow V$ in $\Scal$, a natural isomorphism of functors $\Hbb(V) \rightarrow \Hbb(T)$
		\begin{equation*}
			\conn = \conn_{f,g} \colon (gf)^* A \xrightarrow{\sim} f^* g^* A
		\end{equation*}
	\end{itemize}
	satisfying conditions (fib-$0$) and (fib-$1$) from \Cref{defn:S-fib} plus the following additional condition:
	\begin{enumerate}
		\item[(tri-fib)] For every pair of composable morphisms $f \colon T \rightarrow S$, $g \colon S \rightarrow V$ in $\Scal$, the diagram of functors $\Hbb(V) \rightarrow \Hbb(T)$
		\begin{equation*}
			\begin{tikzcd}
				(gf)^* (A[1]) \arrow[equal]{rr} \arrow{d}{\conn_{f,g}} &&
				((gf)^* A)[1] \arrow{d}{\conn_{f,g}} \\
				f^* g^* (A[1]) \arrow[equal]{r} &
				f^* ((g^* A)[1]) \arrow[equal]{r} &
				(f^* g^* A)[1]
			\end{tikzcd}
		\end{equation*}
		is commutative.
	\end{enumerate}
\end{defn}

\begin{rem}\label[rem]{rem:tri-fib}
	The significance of axiom (tri-fib) is to ensure that any two combinations of inverse image functors and shift functors cannot get identified in more than one way.
\end{rem}

\subsection{Triangulated internal and external tensor structures}

From now until the end of \Cref{sect:tri-acu}, we work over a triangulated $\Scal$-fibered category $\Hbb$ in the sense of \Cref{defn:triS-fib}.

\begin{nota}\label[nota]{nota:tri-when-tens}
	Following \Cref{nota:fib-when-tens}, from now on we write all natural isomorphisms that are part of the triangulated $\Scal$-fibered structure of $\Hbb$ as unmarked equalities:
	this is harmless, as explained in \Cref{rem:tri-fib}. 
\end{nota}

To have a good theory of triangulated monoidal structures on $\Hbb$, one has to specify how the tensor product functors interact with the shift functor on $\Hbb(S)$ for each $S \in \Scal$.
Following \cite[\S~2.1.6]{Ayo07a}, we formalize this idea as follows:
 
\begin{defn}\label[defn]{defn:triITS}
	\begin{enumerate}
		\item A \textit{triangulated internal tensor structure} $(\otimes,m)$ on $\Hbb$ is the datum of
		\begin{itemize}
			\item for every $S \in \Scal$, a bi-triangulated functor
			\begin{equation*}
				- \otimes - = - \otimes_S - \colon \Hbb(S) \times \Hbb(S) \rightarrow \Hbb(S),
			\end{equation*}
			\item for every morphism $f \colon T \rightarrow S$ in $\Scal$, a natural isomorphism of functors $\Hbb(S) \times \Hbb(S) \rightarrow \Hbb(T)$
			\begin{equation*}
				m = m_f \colon f^* A \otimes f^* B \xrightarrow{\sim} f^* (A \otimes B)
			\end{equation*}
		\end{itemize}
		satisfying condition ($m$ITS) from \Cref{defn:ITS} plus the following additional conditions:
		\begin{enumerate}
			\item[(triITS-1)] For every $S \in \Scal$, the diagram of functors $\Hbb(S) \times \Hbb(S) \to \Hbb(S)$
			\begin{equation*}
				\begin{tikzcd}
					A[1] \otimes B[1] \arrow{r}{\sim} \isoarrow{d} &
					(A \otimes B[1])[1] \isoarrow{d} \\
					(A[1] \otimes B)[1] \arrow{r}{\sim} &
					(A \otimes B)[2]
				\end{tikzcd}
			\end{equation*}
			is anti-commutative. 
			\item[(triITS-2)] For every morphism $f \colon T \to S$ in $\Scal$, the two diagrams of functors $\Hbb(S) \times \Hbb(S) \to \Hbb(T)$
			\begin{equation*}
				\begin{tikzcd}
					f^*(A[1]) \otimes f^* B \arrow{r}{m} \arrow[equal]{d} &
					f^* (A[1] \otimes B) \isoarrow{d} \\
					(f^* A)[1] \otimes f^* B \isoarrow{d} &
					f^* ((A \otimes B)[1]) \arrow[equal]{d} \\
					(f^* A \otimes f^* B)[1] \arrow{r}{m} &
					(f^* (A \otimes B))[1]
				\end{tikzcd}
				\qquad
				\begin{tikzcd}
					f^* A \otimes f^*(B[1]) \arrow{r}{m} \arrow[equal]{d} &
					f^* (A \otimes B[1]) \isoarrow{d} \\
					f^* A \otimes (f^* B)[1] \isoarrow{d} &
					f^* ((A \otimes B)[1]) \arrow[equal]{d} \\
					(f^* A \otimes f^* B)[1] \arrow{r}{m} &
					(f^* (A \otimes B))[1]
				\end{tikzcd}
			\end{equation*}
			are commutative.
		\end{enumerate}
		\item Let $(\otimes,m)$ and $(\otimes',m')$ be two triangulated internal tensor structures on $\Hbb$. 
		An equivalence of internal tensor structures $e \colon (\otimes,m) \xrightarrow{\sim} (\otimes',m')$ is \textit{triangulated} if it satisfies the following additional condition:
		\begin{enumerate}
			\item[(eq-triITS)] For every $S \in \Scal$, the two diagrams of functors $\Hbb(S) \times \Hbb(S) \to \Hbb(S)$
			\begin{equation*}
				\begin{tikzcd}
					A[1] \otimes B \arrow{r}{e} \isoarrow{d} &
					A[1] \otimes' B \isoarrow{d} \\
					(A \otimes B)[1] \arrow{r}{e} &
					(A \otimes' B)[1]
				\end{tikzcd}
				\qquad
				\begin{tikzcd}
					A \otimes B[1] \arrow{r}{e} \isoarrow{d} &
					A \otimes' B[1] \isoarrow{d} \\
					(A \otimes B)[1] \arrow{r}{e} &
					(A \otimes' B)[1] 
				\end{tikzcd}
			\end{equation*}
			are commutative.
		\end{enumerate}
	\end{enumerate}
\end{defn}

\begin{ex}
	The anti-commutativity requirement in axiom (triITS-1) is motivated by the case of categories of complexes:
	if $\Acal$ is an additive category endowed with a bi-functor
	\begin{equation*}
		- \otimes - \colon \Acal \times \Acal \to \Acal,
	\end{equation*}
	the induced bi-functor on bounded complexes
	\begin{equation*}
		- \otimes - \colon C^b(\Acal) \times C^b(\Acal) \to C^b(\Acal)
	\end{equation*}
	can be endowed with compatibility isomorphisms with respect to shifts in the two variables;
	these isomorphisms depend on some choice of signs, but in any case the two possible natural isomorphisms
	\begin{equation*}
		A^{\bullet}[1] \otimes B^{\bullet}[1] \xrightarrow{\sim} (A^{\bullet} \otimes B^{\bullet})[2]
	\end{equation*}
	differ by a sign $-1$.
\end{ex}

\begin{defn}\label[defn]{defn:triETS}
	\begin{enumerate}
		\item A \textit{triangulated external tensor structure} $(\boxtimes,m)$ on $\Hbb$ is the datum of
		\begin{itemize}
			\item for every $S_1, S_2 \in \Scal$, a bi-triangulated functor 
			\begin{equation*}
				- \boxtimes -= - \boxtimes_{S_1, S_2} - \colon \Hbb(S_1) \times \Hbb(S_2) \rightarrow \Hbb(S_1 \times S_2),
			\end{equation*}
			\item for every choice of morphisms $f_i \colon T_i \rightarrow S_i$ in $\Scal$, $i = 1,2$, a natural isomorphism of functors $\Hbb(S_1) \times \Hbb(S_2) \rightarrow \Hbb(T_1 \times T_2)$
			\begin{equation*}
				m = m_{f_1,f_2} \colon f_1^* A_1 \boxtimes f_2^* A_2 \xrightarrow{\sim} (f_1 \times f_2)^* (A_1 \boxtimes A_2)
			\end{equation*}
		\end{itemize}
		satisfying condition ($m$ETS) from \Cref{defn:ETS} plus the following additional conditions:
		\begin{enumerate}
			\item[(triETS-1)] For every $S_1, S_2 \in \Scal$, the diagram of functors $\Hbb(S_2) \times \Hbb(S_2) \to \Hbb(S_1 \times S_2)$
			\begin{equation*}
				\begin{tikzcd}
					A_1[1] \boxtimes A_2[1] \arrow{r}{\sim} \isoarrow{d} &
					(A_1 \boxtimes A_2[1])[1] \isoarrow{d} \\
					(A_1[1] \boxtimes A_2)[1] \arrow{r}{\sim} &
					(A_1 \boxtimes A_2)[2]
				\end{tikzcd}
			\end{equation*}
			is anti-commutative.
			\item[(triETS-2)] For every choice of morphisms $f_i \colon T_i \rightarrow S_i$ in $\Scal$, $i = 1,2$, the two diagrams of functors $\Hbb(S_1) \times \Hbb(S_2) \to \Hbb(T_1 \times T_2)$
			\begin{equation*}
				\begin{tikzcd}[font=\small]
					f_1^*(A_1[1]) \boxtimes f_2^* A_2 \arrow{r}{m} \arrow[equal]{d} &
					(f_1 \times f_2)^* (A_1[1] \boxtimes A_2) \isoarrow{d} \\
					(f_1^* A_1)[1] \boxtimes f_2^* A_2 \isoarrow{d} &
					(f_1 \times f_2)^* ((A_1 \boxtimes A_2)[1]) \arrow[equal]{d} \\
					(f_1^* A_1 \boxtimes f_2^* A_2)[1] \arrow{r}{m} &
					((f_1 \times f_2)^* (A_1 \boxtimes A_2))[1]
				\end{tikzcd}
				\qquad 
				\begin{tikzcd}[font=\small]
					f_1^* A_1 \boxtimes f_2^*(A_2[1]) \arrow{r}{m} \arrow[equal]{d} &
					(f_1 \times f_2)^* (A_1 \boxtimes A_2[1]) \isoarrow{d} \\
					f_1^* A_1 \boxtimes (f_2^* A_2)[1] \isoarrow{d} &
					(f_1 \times f_2)^* ((A_1 \boxtimes A_2)[1]) \arrow[equal]{d} \\
					(f_1^* A_1 \boxtimes f_2^* A_2)[1] \arrow{r}{m} &
					((f_1 \times f_2)^* (A_1 \boxtimes A_2))[1]
				\end{tikzcd}
			\end{equation*}
			are commutative.
		\end{enumerate}
		\item Let $(\boxtimes',m')$ and $(\boxtimes,m)$ be two triangulated external tensor structures on $\Hbb$.
		An equivalence of external tensor structures $e \colon (\boxtimes',m') \xrightarrow{\sim} (\boxtimes,m)$ is \textit{triangulated} if it satisfies the following additional condition:
		\begin{enumerate}
			\item[(eq-triETS)] For every $S_1, S_2 \in \Scal$, the two diagrams of functors $\Hbb(S_1) \times \Hbb(S_2) \to \Hbb(S_1 \times S_2)$
			\begin{equation*}
				\begin{tikzcd}
					A_1[1] \boxtimes' A_2 \arrow{r}{e} \isoarrow{d} &
					A_1[1] \boxtimes A_2 \isoarrow{d} \\
					(A_1 \boxtimes' A_2)[1] \arrow{r}{e} &
					(A_1 \boxtimes A_2)[1]
				\end{tikzcd}
				\qquad
				\begin{tikzcd}
					A_1 \boxtimes' A_2[1] \arrow{r}{e} \isoarrow{d} &
					A_1 \boxtimes A_2[1] \isoarrow{d} \\
					(A_1 \boxtimes' A_2)[1] \arrow{r}{e} &
					(A_1 \boxtimes A_2)[1] 
				\end{tikzcd}
			\end{equation*}
			are commutative.
		\end{enumerate}
	\end{enumerate}
\end{defn}

The translation between internal and external tensor structures specialize naturally to the triangulated setting, as follows: 
\begin{lem}\label[lem]{lem:tri-int_to_ext}
	Let $(\otimes,m)$ be a triangulated internal tensor structure on $\Hbb$, and let $(\boxtimes,m^e)$ denote the external tensor structure obtained from it via \Cref{lem:int_to_ext}.
	Then, associating
	\begin{itemize}
		\item to every $S_1, S_2 \in \Scal$, the natural isomorphisms of functors $\Hbb(S_1) \times \Hbb(S_2) \to \Hbb(S_1 \times S_2)$
		\begin{equation*}
			A_1[1] \boxtimes A_2 \xrightarrow{\sim} (A_1 \boxtimes A_2)[1] \xleftarrow{\sim} A_1 \boxtimes A_2[1]
		\end{equation*}
		defined as the composites
		\begin{equation*}
			\begin{tikzcd}
				A_1[1] \boxtimes A_2 \arrow[equal]{d} && 
				(A_1 \boxtimes A_2)[1] \arrow[equal]{d} \\
				\pi^{12,*}_1 (A_1[1]) \otimes \pi^{12,*}_2 A_2 \arrow[equal]{r} &
				(\pi^{12,*}_1 A_1)[1] \otimes \pi^{12,*}_2 A_2 \arrow{r}{\sim}&
				(\pi^{12,*}_1 A_1 \otimes \pi^{12,*}_2 A_2)[1] \\
			\end{tikzcd}
		\end{equation*}
		and
		\begin{equation*}
			\begin{tikzcd}
				(A_1 \boxtimes A_2)[1] \arrow[equal]{d} && 
				A_1 \boxtimes A_2[1], \arrow[equal]{d} \\
				(\pi^{12,*}_1 A_1 \otimes \pi^{12,*}_2 A_2)[1] &
				\pi^{12,*}_1 A_1 \otimes (\pi^{12,*}_2 A_2)[1] \arrow{l}{\sim} &
				\pi^{12,*}_1 A_1 \otimes \pi^{12,*}_2 (A_2[1]) \arrow[equal]{l}
			\end{tikzcd}
		\end{equation*}
		respectively
	\end{itemize}
	turns $(\boxtimes,m^e)$ into a triangulated external tensor structure on $\Hbb$.
\end{lem}
\begin{proof}
	Let us check that the definition in the statement satisfies the conditions of \Cref{defn:triETS}.
	We start from condition (triETS-1):
	we have to show that, for every $S_1, S_2 \in \Scal$, the diagram of functors $\Hbb(S_1) \times \Hbb(S_2) \to \Hbb(S_1 \times S_2)$
	\begin{equation*}
		\begin{tikzcd}
			A_1[1] \boxtimes A_2[1] \arrow{rr}{\sim} \isoarrow{d} &&
			(A_1 \boxtimes A_2[1])[1] \isoarrow{d} \\
			(A_1[1] \boxtimes A_2)[1] \arrow{rr}{\sim} &&
			(A_1 \boxtimes A_2)[2]
		\end{tikzcd}
	\end{equation*}
	is anti-commutative.
	Unwinding the definitions, we obtain the outer rectangle of the diagram
	\begin{equation*}
		\begin{tikzcd}
			\pi^{12,*}_1 (A_1[1]) \otimes \pi^{12,*}_2 (A_2[1]) \arrow[equal]{r} \arrow[equal]{d} &
			(\pi^{12,*}_1 A_1) [1] \otimes \pi^{12,*}_2 (A_2[1]) \arrow{r}{\sim} \arrow[equal]{d} &
			(\pi^{12,*}_1 A_1 \otimes \pi^{12,*}_2 (A_2[1]))[1] \arrow[equal]{d} \\
			\pi^{12,*}_1 (A_1[1]) \otimes (\pi^{12,*}_2 A_2)[1] \arrow[equal]{r} \isoarrow{d} &
			(\pi^{12,*}_1 A_1) [1] \otimes (\pi^{12,*}_2 A_2)[1] \arrow{r}{\sim} \isoarrow{d} &
			(\pi^{12,*}_1 A_1 \otimes (\pi^{12,*}_2 A_2)[1])[1] \isoarrow{d} \\
			(\pi^{12,*}_1 (A_1[1]) \otimes \pi^{12,*}_2 A_2)[1] \arrow[equal]{r} &
			((\pi^{12,*}_1 A_1)[1] \otimes \pi^{12,*}_2 A_2)[1] \arrow{r}{\sim} &
			(\pi^{12,*}_1 A_1 \otimes \pi^{12,*}_2 A_2)[2].
		\end{tikzcd}
	\end{equation*}
	Here, the lower-right square is anti-commutative by axiom (triITS-1), while the other pieces commute by naturality.
	
	Next, we establish condition (triETS-2).
	Let us show that, given two morphisms $f_i \colon T_i \to S_i$, $i = 1,2$, the diagram of functors $\Hbb(S_1) \times \Hbb(S_2) \to \Hbb(T_1 \times T_2)$
	\begin{equation*}
		\begin{tikzcd}
			f_1^* (A_1[1]) \boxtimes f_2^* A_2 \arrow{rr}{m^e} \arrow[equal]{d} &&
			(f_1 \times f_2)^* (A_1[1] \boxtimes A_2) \isoarrow{d} \\
			(f_1^* A_1)[1] \boxtimes f_2^* A_2 \isoarrow{d} &&
			(f_1 \times f_2)^* ((A_1 \boxtimes A_2)[1]) \arrow[equal]{d} \\
			(f_1^* A_1 \boxtimes f_2^* A_2)[1] \arrow{rr}{m^e} &&
			((f_1 \times f_2)^* (A_1 \boxtimes A_2))[1]
		\end{tikzcd}
	\end{equation*}
	commutes;
	the same argument  works for the specular diagram.
	Unwinding the definitions, we obtain the outer rectangle of the diagram
	\begin{equation*}
		\begin{tikzcd}[font=\small]
			\pi^{12,*}_1 f_1^*(A_1[1]) \otimes \pi^{12,*}_2 f_2^* A_2 \arrow[equal]{r} \arrow[equal]{d} &
			f_{12}^* \pi^{12,*}_1 (A_1[1]) \otimes f_{12}^* \pi^{12,*}_2 A_2 \arrow{r}{m} \arrow[equal]{d} &
			f_{12}^* (\pi^{12,*}_1 (A_1[1]) \otimes \pi^{12,*}_2 A_2) \arrow[equal]{d} \\
			\pi^{12,*}_1 ((f_1^* A_1)[1]) \otimes \pi^{12,*}_2 f_2^* A_2 \arrow[equal]{d} &
			f_{12}^* ((\pi^{12,*}_1 A_1)[1]) \otimes f_{12}^* \pi^{12,*}_2 A_2 \arrow{r}{m} \arrow[equal]{d} &
			f_{12}^* ((\pi^{12,*}_1 A_1)[1] \otimes \pi^{12,*}_2 A_2) \isoarrow{d} \\
			(\pi^{12,*}_1 f_1^* A_1)[1] \otimes \pi^{12,*}_2 f_2^* A_2 \arrow[equal]{r} \isoarrow{d} &
			(f_{12}^* \pi^{12,*}_1 A_1)[1] \otimes f_{12}^* \pi^{12,*}_2 A_2 \isoarrow{d} &
			f_{12}^* ((\pi^{12,*}_1 A_1 \otimes \pi^{12,*}_2 A_2)[1]) \arrow[equal]{d} \\
			(\pi^{12,*}_1 f_1^* A_1 \otimes \pi^{12,*}_2 f_2^* A_2)[1] \arrow[equal]{r} &
			(f_{12}^* \pi^{12,*}_1 A_1 \otimes f_{12}^* \pi^{12,*}_2 A_2)[1] \arrow{r}{m} &
			(f_{12}^* (\pi^{12,*}_1 A_1 \otimes \pi^{12,*}_2 A_2))[1],
		\end{tikzcd}
	\end{equation*}
	where $f_1 \times f_2$ is abbreviated as $f_{12}$.
	Here, the lower-right rectangle commutes by axiom (triITS-2), the upper-left rectangle commutes by axiom (tri-fib), while the other two pieces commute by naturality.
\end{proof}

\begin{lem}\label[lem]{lem:tri-ext_to_int}
	Let $(\boxtimes,m)$ be a triangulated external tensor structure on $\Hbb$, and let $(\otimes,m^i)$ denote the internal tensor structure obtained from it via \Cref{lem:ext_to_int}.
	Then, associating
	\begin{itemize}
		\item to every $S \in \Scal$, the natural isomorphisms of functors $\Hbb(S) \times \Hbb(S) \to \Hbb(S)$
		\begin{equation*}
			A[1] \otimes B \xrightarrow{\sim} (A \otimes B)[1] \xleftarrow{\sim} A \otimes B[1]
		\end{equation*}
		defined as the composites
		\begin{equation*}
			\begin{tikzcd}
				A[1] \otimes B \arrow[equal]{d} &&
				(A \otimes B)[1] \arrow[equal]{d} \\
				\Delta_S^* (A[1] \boxtimes B) \arrow{r}{\sim} &
				\Delta_S^* ((A \boxtimes B)[1]) \arrow[equal]{r} &
				(\Delta_S^* (A \boxtimes B))[1]
			\end{tikzcd}
		\end{equation*}
		and
		\begin{equation*}
			\begin{tikzcd}
				(A \otimes B)[1] \arrow[equal]{d} && A \otimes B[1], \arrow[equal]{d} \\
				(\Delta_S^* (A \boxtimes B))[1] &
				(\Delta_S^* (A \boxtimes B)[1]) \arrow[equal]{l} &
				\Delta_S^*(A \boxtimes B[1]) \arrow{l}{\sim}
			\end{tikzcd}
		\end{equation*}
		respectively
	\end{itemize}
	turns $(\otimes,m^i)$ into a triangulated internal tensor structure on $\Hbb$.
\end{lem}
\begin{proof}
	Let us check that the definition in the statement satisfies the conditions of \Cref{defn:triITS}.
	We start from condition (triITS-1):
	we have to show that, for every $S \in \Scal$, the diagram of functors $\Hbb(S) \times \Hbb(S) \to \Hbb(S)$
	\begin{equation*}
		\begin{tikzcd}
			A[1] \otimes B[1] \arrow{rr}{\sim} \isoarrow{d} &&
			(A \otimes B[1])[1] \isoarrow{d} \\
			(A[1] \otimes B)[1] \arrow{rr}{\sim} &&
			(A \otimes B)[2]
		\end{tikzcd}
	\end{equation*}
	is anti-commutative.
	Unwinding the definitions, we obtain the outer rectangle of the diagram
	\begin{equation*}
		\begin{tikzcd}
			\Delta_S^* (A[1] \boxtimes B[1]) \arrow{r}{\sim} \isoarrow{d} &
			\Delta_S^* ((A \boxtimes B[1])[1]) \arrow[equal]{r} \isoarrow{d} &
			(\Delta_S^* (A \boxtimes B[1]))[1] \isoarrow{d} \\
			\Delta_S^* ((A[1] \boxtimes B)[1]) \arrow{r}{\sim} \arrow[equal]{d} &
			\Delta_S^* ((A \boxtimes B)[2]) \arrow[equal]{r} \arrow[equal]{d} &
			(\Delta_S^* ((A \boxtimes B)[1]))[1] \arrow[equal]{d} \\
			(\Delta_S^* (A[1] \boxtimes B))[1] \arrow{r}{\sim} &
			(\Delta_S^* ((A \boxtimes B)[1]))[1] \arrow[equal]{r} &
			(\Delta_S^* (A \boxtimes B))[2].
		\end{tikzcd}
	\end{equation*}
	Here, the upper-left square is anti-commutative by axiom (triETS-1), while the other pieces commute by naturality.
	
	Next, we establish condition (triITS-2).
	Let us show that, given a morphism $f \colon T \to S$, the diagram of functors $\Hbb(S) \times \Hbb(S) \to \Hbb(T)$
	\begin{equation*}
		\begin{tikzcd}
			f^* (A[1]) \otimes f^* B \arrow{rr}{m^i} \arrow[equal]{d} &&
			f^* (A[1] \otimes B) \isoarrow{d} \\
			(f^* A)[1] \otimes f^* B \isoarrow{d} &&
			f^* ((A \otimes B)[1]) \arrow[equal]{d} \\
			(f^* A \otimes f^* B)[1] \arrow{rr}{m^i} &&
			(f^* (A \otimes B))[1]
		\end{tikzcd}
	\end{equation*}
	commutes;
	the same argument works for the specular diagram.
	Unwinding the definitions, we obtain the outer rectangle of the diagram
	\begin{equation*}
		\begin{tikzcd}
			\Delta_T^* (f^* (A[1]) \boxtimes f^* B) \arrow{r}{m} \arrow[equal]{d} &
			\Delta_T^* (f \times f)^* (A[1] \boxtimes B) \arrow[equal]{r} \isoarrow{d} &
			f^* \Delta_S^* (A[1] \boxtimes B) \isoarrow{d} \\
			\Delta_T^* ((f^* A)[1] \boxtimes f^* B) \isoarrow{d} &
			\Delta_T^* (f \times f)^* ((A \boxtimes B)[1]) \arrow[equal]{r} \arrow[equal]{d} &
			f^* \Delta_S^* ((A \boxtimes B)[1]) \arrow[equal]{d} \\
			\Delta_T^* ((f^* A \boxtimes f^* B)[1]) \arrow{r}{m} \arrow[equal]{d} &
			\Delta_T^* (((f \times f)^* (A \boxtimes B))[1]) \arrow[equal]{d} &
			f^* ((\Delta_S^* (A \boxtimes B))[1]) \arrow[equal]{d} \\
			(\Delta_T^* (f^* A \boxtimes f^* B))[1] \arrow{r}{m} &
			(\Delta_T^* (f \times f)^* (A \boxtimes B))[1] \arrow[equal]{r} &
			(f^* \Delta_S^* (A \boxtimes B))[1].
		\end{tikzcd}
	\end{equation*}
	Here, the upper-left rectangle commutes by axiom (triETS-2), the lower-right rectangle commutes by axiom (tri-fib), and the other pieces commute by naturality.
\end{proof}

This leads us to the triangulated analogue of \Cref{prop_bij_in_ext}:
\begin{prop}\label[prop]{prop:bij-tri}
	The constructions of \Cref{lem:tri-int_to_ext} and \Cref{lem:tri-ext_to_int} canonically define mutually inverse bijections between triangulated equivalence classes of triangulated internal and external tensor structures on $\Hbb$.
\end{prop}
\begin{proof}
	Start from a triangulated internal tensor structure $(\otimes,m)$ on $\Hbb$, and let $(\otimes',m')$ denote the triangulated internal tensor structure obtained from it by applying first \Cref{lem:tri-int_to_ext} and then \Cref{lem:tri-ext_to_int}.
	We claim that the equivalence of internal tensor structures $e \colon (\otimes,m) \xrightarrow{\sim} (\otimes',m')$ constructed in the proof of \Cref{prop_bij_in_ext} is triangulated, in the sense of \Cref{defn:triITS}(2).
	Let us show that, for every $S \in \Scal$, the diagram of functors $\Hbb(S) \times \Hbb(S) \to \Hbb(S)$
	\begin{equation*}
		\begin{tikzcd}
			A[1] \otimes B \arrow{rr}{e} \isoarrow{d} &&
			A[1] \otimes' B \isoarrow{d} \\
			(A \otimes B)[1] \arrow{rr}{e} &&
			(A \otimes' B)[1]
		\end{tikzcd}
	\end{equation*}
	commutes;
	the same argument works for the specular diagram.
	Unwinding the definitions, we obtain the outer rectangle of the diagram
	\begin{equation*}
		\begin{tikzcd}
			A[1] \otimes B \isoarrow{ddd} \arrow[equal]{r} \arrow[equal,bend right]{ddr} & 
			\Delta_S^* \pi^{12,*}_1 (A[1]) \otimes \Delta_S^* \pi^{12,*}_2 B \arrow{r}{m} \arrow[equal]{d} & 
			\Delta_S^* (\pi^{12,*}_1 (A[1]) \otimes \pi^{12,*}_2 B) \arrow[equal]{d} \\
			&
			\Delta_S^* ((\pi^{12,*}_1 A)[1]) \otimes \Delta_S^* \pi^{12,*}_2 B \arrow{r}{\sim} \arrow[equal]{d} &
			\Delta_S^* ((\pi^{12,*}_1 A)[1] \otimes \pi^{12,*}_2 B) \isoarrow{d} \\
			&
			(\Delta_S^* \pi^{12,*}_1 A)[1] \otimes \Delta_S^* \pi^{12,*}_2 B \isoarrow{d} & 
			\Delta_S^* ((\pi^{12,*}_1 A \otimes \pi^{12,*}_2 B)[1]) \arrow[equal]{d} \\
			(A \otimes B)[1] \arrow[equal]{r} &
			(\Delta_S^* \pi^{12,*}_1 A \otimes \Delta_S^* \pi^{12,*}_2 B)[1] \arrow{r}{m} &
			(\Delta_S^* (\pi^{12,*}_1 A \otimes \pi^{12,*}_2 B))[1].
		\end{tikzcd}
	\end{equation*}
	Here, the lower-right rectangle commutes by axiom (triITS-2), the upper-left piece commutes by axiom (tri-fib), while the other two pieces commute by naturality.
	This proves the claim.
	
	Conversely, start from a triangulated external tensor structure $(\boxtimes,m)$ on $\Hbb$, and let $(\boxtimes',m')$ denote the triangulated external tensor structure obtained from it by applying first \Cref{lem:tri-ext_to_int} and then \Cref{lem:tri-int_to_ext}.
	We claim that the equivalence of external tensor structures $e \colon (\boxtimes',m') \xrightarrow{\sim} (\boxtimes,m)$ constructed in the proof of \Cref{prop_bij_in_ext} is triangulated, in the sense of \Cref{defn:triETS}(2).
	Let us show that, for every $S_1, S_2 \in \Scal$, the diagram of functors $\Hbb(S_1) \times \Hbb(S_2) \to \Hbb(S_1 \times S_2)$
	\begin{equation*}
		\begin{tikzcd}
			A_1[1] \boxtimes' A_2 \arrow{rr}{e} \isoarrow{d} &&
			A_1[1] \boxtimes A_2 \isoarrow{d} \\
			(A_1 \boxtimes' A_2)[1] \arrow{rr}{e} &&
			(A_1 \boxtimes A_2)[1]
		\end{tikzcd}
	\end{equation*}
	commutes;
	the same argument works for the specular diagram.
	Unwinding the definitions, we obtain the outer rectangle of the diagram
	\begin{equation*}
		\begin{tikzcd}
			\Delta_{S_1 \times S_2}^* (\pi^{12,*}_1 (A_1[1]) \boxtimes' \pi^{12,*}_2 A_2) \arrow{r}{m} \arrow[equal]{d} & 
			\Delta_{S_1 \times S_2}^* (\pi^{12}_1 \times \pi^{12}_2)^* (A_1[1] \boxtimes A_2) \arrow[equal]{r} \isoarrow{d} &
			A_1[1] \boxtimes A_2 \isoarrow{ddd} \\
			\Delta_{S_1 \times S_2}^* ((\pi^{12,*}_1 A_1)[1] \boxtimes \pi^{12,*}_2 A_2) \isoarrow{d} &
			\Delta_{S_1 \times S_2}^* (\pi^{12,*}_1 \times \pi^{12,*}_2)((A_1 \boxtimes A_2)[1]) \arrow[equal]{d} \\
			\Delta_{S_1 \times S_2}^* ((\pi^{12,*}_1 A_1 \boxtimes \pi^{12,*}_2 A_2)[1]) \arrow[equal]{d} \arrow{r}{m} &
			\Delta_{S_1 \times S_2}^* ((\pi^{12}_1 \times \pi^{12}_2)^* (A_1 \boxtimes A_2)[1]) \arrow[equal]{d} \\
			(\Delta_{S_1 \times S_2}^* (\pi^{12,*}_1 A_1 \boxtimes \pi^{12,*}_2 A_2))[1] \arrow{r}{m} & 
			(\Delta_{S_1 \times S_2}^* (\pi^{12}_1 \times \pi^{12}_2)^* (A_1 \boxtimes A_2))[1] \arrow[equal]{r} &
			(A_1 \boxtimes A_2)[1]. \arrow[equal,bend right]{uul}
		\end{tikzcd}
	\end{equation*}
	Here, the upper-left rectangle commutes by axiom (triETS-2), the lower-right piece commutes by axiom (tri-fib), while the other two pieces commute by naturality.
	This proves the claim, thereby concluding the proof.
\end{proof}

\section{The case of triangulated associativity, commutativity and unit constraints}\label{sect:tri-acu}

As in the previous section, we work over a fixed triangulated $\Scal$-fibered category $\Hbb$.
The goal of the present section is to refine the correspondence of \Cref{prop:bij-tri} so as to include associativity, commutativity, and unit constraints:
the idea is to impose their compatibility with shift functors fiber-wise, thus inverse image functors are not involved.
In the end, this will make the proof of \Cref{thm:tri-int_ext} complete.

The results collected below can be proved by the same method used in \Cref{sect:tri-int_ext};
we leave the details to the interested reader.

\subsection{Triangulated associativity constraints}

\begin{defn}
	Let $(\otimes,m)$ be a triangulated internal tensor structure on $\Hbb$.
	We say that an associativity constraint $a$ for $(\otimes,m)$ is \textit{triangulated} if it satisfies the following additional condition:
	\begin{enumerate}
		\item[(tri$a$ITS)] For every $S \in \Scal$, the three diagrams of functors $\Hbb(S) \times \Hbb(S) \times \Hbb(S)$
		\begin{equation*}
			\begin{tikzcd}
				(A[1] \otimes B) \otimes C \arrow{r}{\sim} \arrow{d}{a} &
				(A \otimes B)[1] \otimes C \arrow{r}{\sim} &
				((A \otimes B) \otimes C)[1] \arrow{d}{a} \\
				A[1] \otimes (B \otimes C) \arrow{rr}{\sim} &&
				(A \otimes (B \otimes C))[1]
			\end{tikzcd}
		\end{equation*}
		\begin{equation*}
			\begin{tikzcd}
				(A \otimes B[1]) \otimes C \arrow{r}{\sim} \arrow{d}{a} &
				(A \otimes B)[1] \otimes C \arrow{r}{\sim} &
				((A \otimes B) \otimes C)[1] \arrow{d}{a} \\
				A \otimes (B[1] \otimes C) \arrow{r}{\sim} &
				A \otimes (B \otimes C)[1] \arrow{r}{\sim} &
				(A \otimes (B \otimes C))[1]
			\end{tikzcd}
		\end{equation*}
		\begin{equation*}
			\begin{tikzcd}
				(A \otimes B) \otimes C[1] \arrow{rr}{\sim} \arrow{d}{a} &&
				((A \otimes B) \otimes C)[1] \arrow{d}{a} \\
				A \otimes (B \otimes C[1]) \arrow{r}{\sim} &
				A \otimes (B \otimes C)[1] \arrow{r}{\sim} &
				(A \otimes (B \otimes C))[1]
			\end{tikzcd}
		\end{equation*}
		are commutative.
	\end{enumerate}
\end{defn}

\begin{defn}
	Let $(\boxtimes,m)$ be a triangulated external tensor structure on $\Hbb$.
	We say that an external associativity constraint $a$ for $(\boxtimes,m)$ is \textit{triangulated} if it satisfies the following additional condition:
	\begin{enumerate}
		\item[(tri$a$ETS)] For every $S_1, S_2, S_3 \in \Scal$, the three diagrams of functors $\Hbb(S_1) \times \Hbb(S_2) \times \Hbb(S_3) \rightarrow \Hbb(S_1 \times S_2 \times S_3)$
		\begin{equation*}
			\begin{tikzcd}
				(A_1[1] \boxtimes A_2) \boxtimes A_3 \arrow{r}{\sim} \arrow{d}{a} &
				(A_1 \boxtimes A_2)[1] \boxtimes A_3 \arrow{r}{\sim} &
				((A_1 \boxtimes A_2) \boxtimes A_3)[1] \arrow{d}{a} \\
				A_1[1] \boxtimes (A_2 \boxtimes A_3) \arrow{rr}{\sim} &&
				(A_1 \boxtimes (A_2 \boxtimes A_3))[1]
			\end{tikzcd}
		\end{equation*}
		\begin{equation*}
			\begin{tikzcd}
				(A_1 \boxtimes A_2[1]) \boxtimes A_3 \arrow{r}{\sim} \arrow{d}{a} &
				(A_1 \boxtimes A_2)[1] \boxtimes A_3 \arrow{r}{\sim} &
				((A_1 \boxtimes A_2) \boxtimes A_3)[1] \arrow{d}{a} \\
				A_1 \boxtimes (A_2[1] \boxtimes A_3) \arrow{r}{\sim} &
				A_1 \boxtimes (A_2 \boxtimes A_3)[1] \arrow{r}{\sim} &
				(A_1 \boxtimes (A_2 \boxtimes A_3))[1]
			\end{tikzcd}
		\end{equation*}
		\begin{equation*}
			\begin{tikzcd}
				(A_1 \boxtimes A_2) \boxtimes A_3[1] \arrow{rr}{\sim} \arrow{d}{a} &&
				((A_1 \boxtimes A_2) \boxtimes A_3)[1] \arrow{d}{a} \\
				A_1 \boxtimes (A_2 \boxtimes A_3[1]) \arrow{r}{\sim} &
				A_1 \boxtimes (A_2 \boxtimes A_3)[1] \arrow{r}{\sim} &
				(A_1 \boxtimes (A_2 \boxtimes A_3))[1]
			\end{tikzcd}
		\end{equation*}
		are commutative.
	\end{enumerate}
\end{defn}

\begin{lem}
	Let $(\otimes,m)$ be a triangulated internal tensor structure on $\Hbb$, and let $(\boxtimes,m^e)$ denote the triangulated external tensor structure obtained from it via \Cref{lem:tri-int_to_ext}.
	Suppose that we are given a triangulated associativity constraint $a$ for $(\otimes,m)$.
	Then, the external associativity constraint $a^e$ obtained from it via \Cref{lem:ass-int_to_ext} is triangulated as well.
\end{lem}

\begin{lem}
	Let $(\boxtimes,m)$ be a triangulated external tensor structure on $\Hbb$, and let $(\otimes,m^i)$ denote the triangulated internal tensor structure obtained from it via \Cref{lem:tri-ext_to_int}.
	Suppose that we are given a triangulated external associativity constraint $a$ for $(\boxtimes,m)$.
	Then, the internal associativity constraint $a^i$ obtained from it via \Cref{lem:ass-ext_to_int} is triangulated as well.
\end{lem}

More precisely, the commutativity of each of the three diagrams in axiom (tri$a$ITS) is equivalent to the commutativity of the corresponding diagram in axiom (tri$a$ETS).

\subsection{Triangulated commutativity constraints}

\begin{defn}
	Let $(\otimes,m)$ be a triangulated internal tensor structure on $\Hbb$.
	We say that an internal commutativity constraint $c$ for $(\otimes,m)$ is \textit{triangulated} if it satisfies the following additional condition:
	\begin{enumerate}
	\item[(tri$c$ITS)] For every $S \in \Scal$, the diagram of functors $\Hbb(S) \times \Hbb(S) \rightarrow \Hbb(S)$
	\begin{equation*}
		\begin{tikzcd}
			A[1] \otimes B \arrow{r}{\sim} \arrow{d}{c} &
			(A \otimes B)[1] \arrow{d}{c} \\
			B \otimes A[1] \arrow{r}{\sim} &
			(B \otimes A)[1]
		\end{tikzcd}
	\end{equation*}
	is commutative.
	\end{enumerate}
\end{defn}

\begin{defn}
	Let $(\boxtimes,m)$ be a triangulated external tensor structure on $\Hbb$.
	We say that an external commutativity constraint $c$ for $(\boxtimes,m)$ is \textit{triangulated} if it satisfies the following additional condition:
	\begin{enumerate}
		\item[(tri$c$ETS)] For every $S_1, S_2 \in \Scal$, the diagram of functors $\Hbb(S_1) \times \Hbb(S_2) \rightarrow \Hbb(S_1 \times S_2)$
		\begin{equation*}
			\begin{tikzcd}
				A_1[1] \boxtimes A_2 \arrow{rr}{\sim} \arrow{d}{c} &&
				(A_1 \boxtimes A_2)[1] \arrow{d}{c} \\
				\tau^* (A_2 \boxtimes A_1[1]) \arrow{r}{\sim} &
				\tau^* ((A_2 \boxtimes A_1)[1]) \arrow{r}{\sim} &
				(\tau^* (A_2 \boxtimes A_1))[1]
			\end{tikzcd}
		\end{equation*}
		is commutative.
	\end{enumerate}
\end{defn}

\begin{lem}
	Let $(\otimes,m)$ be a triangulated internal tensor structure on $\Hbb$, and let $(\boxtimes,m)$ denote the triangulated external tensor structure obtained from it via \Cref{lem:tri-int_to_ext}.
	Suppose that we are given a triangulated internal commutativity constraint $c$ for $(\otimes,m)$.
	Then, the external commutativity constraint $c^e$ obtained from it via \Cref{lem:symm_int_to_ext} is triangulated as well.
\end{lem}

\begin{lem}
	Let $(\boxtimes,m)$ be a triangulated external tensor structure on $\Hbb$, and let $(\otimes,m^i)$ denote the triangulated internal tensor structure obtained from it via \Cref{lem:tri-ext_to_int}.
	Suppose that we are given a triangulated external commutativity constraint $c$ for $(\boxtimes,m)$.
	Then, the internal commutativity constraint $c^i$ obtained from it via \Cref{lem:symm_ext_to_int} is triangulated as well.
\end{lem}

\subsection{Triangulated unit constraints}

\begin{defn}
	Let $(\otimes,m)$ be a triangulated internal tensor structure on $\Hbb$.
	We say that an internal unit constraint $u$ (with unit section $\unit$) for $(\otimes,m)$ is \textit{triangulated} if it satisfies the following additional condition:
	\begin{enumerate}
		\item[(tri$u$ITS)] For every $S \in \Scal$, the two diagrams of functors $\Hbb(S) \rightarrow \Hbb(S)$
		\begin{equation*}
			\begin{tikzcd}
				A[1] \otimes \unit_S \arrow{r}{u_r} \isoarrow{d} &
				A[1] \\
				(A \otimes \unit_S)[1] \arrow{ur}{u_r}
			\end{tikzcd}
			\qquad \qquad
			\begin{tikzcd}
				\unit_S \otimes B[1] \arrow{r}{u_l} \isoarrow{d} &
				B[1] \\
				(\unit_S \otimes B)[1] \arrow{ur}{u_l}
			\end{tikzcd}
		\end{equation*}
		are commutative.
	\end{enumerate}
\end{defn}

\begin{defn}
	Let $(\boxtimes,m)$ be a triangulated external tensor structure on $\Hbb$.
	We say that an external unit constraint $u$ (with unit section $\unit$) for $(\boxtimes,m)$ is \textit{triangulated} if it satisfies the following additional condition:
	\begin{enumerate}
		\item[(tri$u$ETS)] For every $S_1, S_2 \in \Scal$, the diagram of functors $\Hbb(S_1) \to \Hbb(S_1 \times S_2)$
		\begin{equation*}
			\begin{tikzcd}
				A_1[1] \boxtimes \unit_{S_2} \arrow{r}{u_r} \isoarrow{d} &
				\pi^{12,*}_1 (A_1[1]) \arrow[equal]{d} \\
				(A_1 \boxtimes \unit_{S_2})[1] \arrow{r}{u_r} &
				(\pi^{12,*}_1 A_1)[1]
			\end{tikzcd}
		\end{equation*}
		and the diagram of functors $\Hbb(S_2) \to \Hbb(S_1 \times S_2)$
		\begin{equation*}
			\begin{tikzcd}
				\unit_{S_1} \boxtimes A_2[1] \arrow{r}{u_l} \isoarrow{d} &
				\pi^{12,*}_2 (A_2[1]) \arrow[equal]{d} \\
				(\unit_{S_1} \boxtimes A_2)[1] \arrow{r}{u_l} &
				(\pi^{12,*}_2 A_2)[1]
			\end{tikzcd}
		\end{equation*}
		are commutative.
	\end{enumerate}
\end{defn}

\begin{lem}
	Let $(\otimes,m)$ be a triangulated internal tensor structure on $\Hbb$, and let $(\boxtimes,m^e)$ denote the triangulated external tensor structure obtained from it via
	\Cref{lem:tri-int_to_ext}.
	Suppose that we are given a triangulated internal unit constraint $u$ (with unit section $\unit$) for $(\otimes,m)$. Then, the external unit constraint $u^e$ obtained from it via \Cref{lem:unit-int_to_ext} is triangulated as well.
\end{lem}

\begin{lem}
	Let $(\boxtimes,m)$ be a triangulated external tensor structure on $\Hbb$, and let $(\otimes,m^i)$ denote the triangulated internal tensor structure obtained from it via \Cref{lem:tri-ext_to_int}.
	Suppose that we are given a triangulated external unit constraint $u$ (with unit section $\unit$) for $(\boxtimes,m)$.
	Then, the internal unit constraint $u^i$ obtained from it via \Cref{lem:unit-ext_to_int} is triangulated as well.
\end{lem}

\begin{proof}[Proof of \Cref{thm:tri-int_ext}]
	Combine \Cref{thm:int_ext} with \Cref{prop:bij-tri} and with the complementary results of the present section, according to the additional constraints involved.
	The notions of compatibility between associativity, commutativity, and unit constraints discussed in \Cref{subsect:comp} need no further comments in the triangulated setting, since these are properties rather than structures.
\end{proof}

\section{The case of monoidal triangulated morphisms}\label{sect:tri-mor}

To complete the picture, it remains to specialize the study of monoidal structures on morphisms to the setting of triangulated fibered categories.
Our second main result is the triangulated version of \Cref{thm:int_ext-mor}:
\begin{thm}\label[thm]{thm:tri-int_ext-mor}
	Let $R \colon \Hbb_1 \to \Hbb_2$ be a morphism of triangulated $\Scal$-fibered categories.
	The bijection of \Cref{thm:int_ext-mor} naturally refines to a canonical bijection between triangulated internal and external monoidal structures on $R$.
\end{thm}

In this short final section, after fixing our conventions on morphisms of triangulated fibered categories, we study the natural triangulated notions of internal and external monoidal structures on a morphism.
This is the only missing piece for the proof of \Cref{thm:tri-int_ext-mor}. 

\subsection{Recollections on morphisms of triangulated fibered categories}

\begin{defn}\label[defn]{defn:tri-mor}
	Let $\Hbb_1$ and $\Hbb_2$ be two triangulated $\Scal$-fibered categories, in the sense of \Cref{defn:triS-fib}. 
	A \textit{morphism of triangulated $\Scal$-fibered categories} $R \colon \Hbb_1 \rightarrow \Hbb_2$ is the datum of
	\begin{itemize}
		\item for every $S \in \Scal$, a triangulated functor 
		\begin{equation*}
			R_S \colon \Hbb_1(S) \rightarrow \Hbb_2(S),
		\end{equation*}
		\item for every morphism $f \colon T \rightarrow S$ in $\Scal$, a natural isomorphism of functors $\Hbb_1(S) \rightarrow \Hbb_2(T)$
		\begin{equation*}
			\theta = \theta_f \colon f^* R_S(A) \xrightarrow{\sim} R_T(f^* A)
		\end{equation*}
	\end{itemize}
	satisfying condition (mor-fib) from \Cref{defn:mor-Scal-fib} plus the following additional condition:
	\begin{enumerate}
		\item[(mor-tri-fib)] For every morphism $f \colon T \rightarrow S$ in $\Scal$, the diagram of functors $\Hbb_1(S) \rightarrow \Hbb_2(T)$
		\begin{equation*}
			\begin{tikzcd}
				f^* R_S(A[1]) \arrow[equal]{r} \arrow{d}{\theta_f} &
				f^* (R_S(A)[1]) \arrow[equal]{r} &
				(f^* R_S(A))[1] \arrow{d}{\theta_f} \\
				R_T(f^* (A[1])) \arrow[equal]{r} &
				R_T ((f^* A)[1]) \arrow[equal]{r} &
				(R_T (f^* A))[1]
			\end{tikzcd}
		\end{equation*}
		is commutative.
	\end{enumerate}
\end{defn}

\begin{rem}\label[rem]{rem:mor-tri}
	The significance of axiom (mor-tri-fib) is to ensure that any two combinations of inverse image functors, shift functors, and realization functors cannot get identified in more than one way.
\end{rem}

\subsection{Triangulated monoidal structures on morphisms}

From now henceforth, let $R \colon \Hbb_1 \to \Hbb_2$ be a morphism of triangulated $\Scal$-fibered categories in the sense of \Cref{defn:tri-mor}.

\begin{nota}
	Following \Cref{nota:mor-when-tens} and \Cref{nota:tri-when-tens}, we write all natural isomorphisms pertaining to the triangulated morphism structure of $R$ as equalities:
	this is harmless, as explained in \Cref{rem:mor-tri}.
\end{nota}

\begin{defn}\label[defn]{defn:mor-triITS}
	Let $(\otimes_1,m_1)$ and $(\otimes_2,m_2)$ be triangulated internal tensor structures on $\Hbb_1$ and $\Hbb_2$, respectively.
	We say that an internal tensor structure $\rho$ on $R$ is \textit{triangulated} if it satisfies the following additional condition:
	\begin{enumerate}
		\item[(mor-triITS)] For every $S \in \Scal$, the two diagrams of functors $\Hbb_1(S) \times \Hbb_1(S) \rightarrow \Hbb_2(S)$
		\begin{equation*}
			\begin{tikzcd}
				R_S(A[1]) \otimes_2 R_S(B) \arrow[equal]{d} \arrow{r}{\rho} &
				R_S(A[1] \otimes_1 B) \isoarrow{d} \\
				(R_S(A))[1] \otimes_2 R_S(B) \isoarrow{d} &
				R_S((A \otimes_1 B)[1]) \isoarrow{d} \\
				(R_S(A) \otimes_2 R_S(B))[1] \arrow{r}{\rho} &
				(R_S(A \otimes_1 B))[1]
			\end{tikzcd}
			\quad
			\begin{tikzcd}
				R_S(A) \otimes_2 R_S(B[1]) \arrow[equal]{d} \arrow{r}{\rho} &
				R_S(A \otimes_1 B[1]) \isoarrow{d} \\
				R_S(A) \otimes_2 (R_S(B))[1] \isoarrow{d} &
				R_S((A \otimes_1 B)[1]) \arrow[equal]{d} \\
				(R_S(A) \otimes_2 R_S(B))[1] \arrow{r}{\rho} &
				(R_S(A \otimes_1 B))[1]
			\end{tikzcd}
		\end{equation*}
		are commutative.
	\end{enumerate}
\end{defn}

\begin{defn}\label[defn]{defn:mor-triETS}
	Let $(\boxtimes_1,m_1)$ and $(\boxtimes_2,m_2)$ be triangulated external tensor structures on $\Hbb_1$ and $\Hbb_2$, respectively.
	We say that an external tensor structure $\rho$ on $R$ is \textit{triangulated} if it satisfies the following additional condition:
	\begin{enumerate}
		\item[(mor-triETS)] For every $S_1, S_2 \in \Scal$, the two diagrams of functors $\Hbb_1(S_1) \times \Hbb_1(S_2) \rightarrow \Hbb_2(S_1 \times S_2)$
		\begin{equation*}
			\begin{tikzcd}[font=\small]
				R_{S_1}(A_1[1]) \boxtimes_2 R_{S_2}(A_2) \arrow[equal]{d} \arrow{r}{\rho} &
				R_{S_1 \times S_2}(A_1[1] \boxtimes_1 A_2) \isoarrow{d} \\
				(R_{S_1}(A_1))[1] \boxtimes_2 R_{S_2}(A_2) \isoarrow{d} &
				R_{S_1 \times S_2}((A_1 \boxtimes_1 A_2)[1]) \arrow[equal]{d} \\
				(R_{S_1}(A_1) \boxtimes_2 R_{S_2}(A_2))[1] \arrow{r}{\rho} &
				(R_{S_1 \times S_2}(A_1 \boxtimes_1 A_2))[1]
			\end{tikzcd}
			\quad
			\begin{tikzcd}[font=\small]
				R_{S_1}(A_1) \boxtimes_2 R_{S_2}(A_2[1]) \arrow[equal]{d} \arrow{r}{\rho} &
				R_{S_1 \times S_2}(A_1 \boxtimes_1 A_2[1]) \isoarrow{d} \\
				R_{S_1}(A_1) \boxtimes_2 (R_{S_2}(A_2))[1] \isoarrow{d} &
				R_{S_1 \times S_2}((A_1 \boxtimes_1 A_2)[1]) \arrow[equal]{d} \\
				(R_{S_1}(A_1) \boxtimes_2 R_{S_2}(A_2))[1] \arrow{r}{\rho} &
				(R_{S_1 \times S_2}(A_1 \boxtimes_1 A_2))[1]
			\end{tikzcd}
		\end{equation*}
		are commutative.
	\end{enumerate}
\end{defn}

\begin{lem}\label[lem]{lem:tri-int_to_ext-mor}
	For $j = 1,2$, let $(\otimes_j,m_j)$ be a triangulated internal tensor structure on $\Hbb_j$, and let $(\boxtimes_j,m_j)$ denote the triangulated external tensor structure obtained from it via \Cref{lem:tri-int_to_ext}.
	Suppose that we are given a triangulated internal tensor structure $\rho$ on $R$.
	Then, the external tensor structure $\rho^e$ obtained from it via \Cref{lem:mor-int_to_ext} is triangulated as well.
\end{lem}
\begin{proof}
	We have to check that condition (mor-triETS) from \Cref{defn:mor-triETS} is satisfied.
	Let us show that, for every $S_1, S_2 \in \Scal$, the diagram of functors $\Hbb_1(S_1) \times \Hbb_1(S_2) \to \Hbb_2(S_1 \times S_2)$
	\begin{equation*}
		\begin{tikzcd}
			R_{S_1}(A_1[1]) \boxtimes_2 R_{S_2}(A_2) \arrow[equal]{d} \arrow{r}{\rho} &
			R_{S_1 \times S_2}(A_1[1] \boxtimes_1 A_2) \isoarrow{d} \\
			(R_{S_1}(A_1))[1] \boxtimes_2 R_{S_2}(A_2) \isoarrow{d} &
			R_{S_1 \times S_2}((A_1 \boxtimes_1 A_2)[1]) \arrow[equal]{d} \\
			(R_{S_1}(A_1) \boxtimes_2 R_{S_2}(A_2))[1] \arrow{r}{\rho} &
			(R_{S_1 \times S_2}(A_1 \boxtimes_1 A_2))[1]
		\end{tikzcd}
	\end{equation*}
	commutes;
	the same argument works for the specular diagram.
	Unwinding the definitions, we obtain the outer rectangle of the diagram
	\begin{equation*}
		\begin{tikzcd}[font=\tiny]
			\pi^{12,*}_1 R_{S_1}(A_1[1]) \otimes_2 \pi^{12,*}_2 R_{S_2}(A_2) \arrow[equal]{r} \arrow[equal]{d} &
			R_{S_1 \times S_2}(\pi^{12,*}_1 (A_1[1])) \otimes_2 R_{S_1 \times S_2}(\pi^{12,*}_2 A_2) \arrow{r}{\rho} \arrow[equal]{d} &
			R_{S_1 \times S_2}(\pi^{12,*}_1 (A_1[1]) \otimes_1 \pi^{12,*}_2 A_2) \arrow[equal]{d} \\
			\pi^{12,*}_1 ((R_{S_1}(A_1))[1]) \otimes_2 \pi^{12,*}_2 R_{S_2}(A_2) \arrow[equal]{d} &
			R_{S_1 \times S_2}((\pi^{12,*}_1 A_1)[1]) \otimes_2 R_{S_1 \times S_2}(\pi^{12,*}_2 A_2) \arrow{r}{\rho} \arrow[equal]{d} &
			R_{S_1 \times S_2}((\pi^{12,*}_1 A_1)[1] \otimes_1 \pi^{12,*}_2 A_2) \isoarrow{d} \\
			(\pi^{12,*}_1 R_{S_1}(A_1))[1] \otimes_2 \pi^{12,*}_2 R_{S_2}(A_2) \arrow[equal]{r} \isoarrow{d} &
			(R_{S_1 \times S_2}(\pi^{12,*}_1 A_1)) [1] \otimes_2 R_{S_1 \times S_2}(\pi^{12,*}_2 A_2) \isoarrow{d} &
			R_{S_1 \times S_2}((\pi^{12,*}_1 A_1 \otimes_1 \pi^{12,*}_2 A_2)[1]) \arrow[equal]{d} \\
			(\pi^{12,*}_1 R_{S_1}(A_1) \otimes_2 \pi^{12,*}_2 R_{S_2}(A_2))[1] \arrow[equal]{r} &
			(R_{S_1 \times S_2}(\pi^{12,*}_1 A_1) \otimes_2 R_{S_1 \times S_2}(\pi^{12,*}_2 A_2))[1] \arrow{r}{\rho} &
			(R_{S_1 \times S_2}(\pi^{12,*}_1 A_1 \otimes \pi^{12,*}_2 A_2))[1].
		\end{tikzcd}
	\end{equation*}
	Here, the lower-right rectangle commutes by axiom (mor-triITS), the upper-left rectangle commutes by axiom (mor-tri-fib), while the other two pieces commute by naturality.
\end{proof}

\begin{lem}\label[lem]{lem:tri-ext_to_int-mor}
	For $j = 1,2$, let $(\boxtimes_j,m_j)$ be a triangulated external tensor structure on $\Hbb_j$, and let $(\otimes_j,m_j)$ denote the triangulated internal tensor structure obtained from it via \Cref{lem:tri-ext_to_int}.
	Suppose that we are given a triangulated external tensor structure $\rho$ on $R$.
	Then, the internal tensor structure $\rho^i$ obtained from it via \Cref{lem:mor-ext_to_int} is triangulated as well.
\end{lem}
\begin{proof}
	We have to check that condition (mor-triITS) from \Cref{defn:mor-triITS} is satisfied.
	Let us show that, for every $S \in \Scal$, the diagram of functors $\Hbb_1(S) \times \Hbb_1(S) \to \Hbb_2(S)$
	\begin{equation*}
		\begin{tikzcd}
			R_S(A[1]) \otimes_2 R_S(B) \arrow[equal]{d} \arrow{r}{\rho} &
			R_S(A[1] \otimes_1 B) \isoarrow{d} \\
			(R_S(A))[1] \otimes_2 R_S(B) \isoarrow{d} &
			R_S((A \otimes_1 B)[1]) \isoarrow{d} \\
			(R_S(A) \otimes_2 R_S(B))[1] \arrow{r}{\rho} &
			(R_S(A \otimes_1 B))[1]
		\end{tikzcd}
	\end{equation*}
	commutes;
	the same argument works for the specular diagram.
	Unwinding the definitions, we obtain the outer rectangle of the diagram
	\begin{equation*}
		\begin{tikzcd}
			\Delta_S^* (R_S(A[1]) \boxtimes_2 R_S(B)) \arrow{r}{\rho} \arrow[equal]{d} &
			\Delta_S^* R_{S \times S}(A[1] \boxtimes_1 B)\arrow[equal]{r} \isoarrow{d} &
			R_S(\Delta_S^* (A[1] \boxtimes_1 B)) \isoarrow{d} \\
			\Delta_S^* ((R_S(A))[1] \boxtimes_2 R_S(B)) \isoarrow{d} &
			\Delta_S^* R_{S \times S}((A \boxtimes_1 B)[1]) \arrow[equal]{r} \arrow[equal]{d} &
			R_S(\Delta_S^* ((A \boxtimes_1 B)[1])) \arrow[equal]{d} \\
			\Delta_S^* ((R_S(A) \boxtimes_2 R_S(B))[1])\arrow{r}{\rho} \arrow[equal]{d} &
			\Delta_S^* (R_{S \times S}(A \boxtimes_1 B) [1]) \arrow[equal]{r} \arrow[equal]{d} &
			R_S((\Delta_S^* (A \boxtimes_1 B)) [1]) \arrow[equal]{d} \\
			(\Delta_S^* (R_S(A) \boxtimes_2 R_S(B)))[1] \arrow{r}{\rho} &
			(\Delta_S^* R_{S \times S}(A \boxtimes_1 B))[1] \arrow[equal]{r} &
			(R_S(\Delta_S^* (A \boxtimes_1 B)))[1]
		\end{tikzcd}
	\end{equation*}
	Here, the upper-left rectangle commutes by axiom (mor-triETS), the lower-right rectangle commutes by axiom (mor-tri-fib), while the other two pieces commute by naturality.
\end{proof}

\begin{proof}[Proof of \Cref{thm:tri-int_ext-mor}]
	Combine \Cref{thm:int_ext-mor,thm:tri-int_ext} with \Cref{lem:tri-int_to_ext-mor,lem:tri-ext_to_int-mor}.
	The notions of associativity, symmetry, and unitarity of a morphism discussed in \Cref{subsect:acu-mor} do not change in the triangulated setting, since these are properties rather than structures.
\end{proof}


\begin{thebibliography}{9}
	\bibitem[SGA1]{SGA1} 
	\textit{Revêtements étales et groupe fondamental}. 
	Séminaire de Géometrie Algébrique du Bois-Marie, dirigé par A. Grothendieck, augmenté de deux exposés de Mme M. Raynaud, Lecture Notes in Math. \textbf{224}, Springer-Verlag, Berlin (1964).
	
	\bibitem[Ayo07a]{Ayo07a} J. Ayoub,
	\textit{Les six opérations de Grothendieck et le formalisme des cycles évanescents dans le monde motivique, I}.
	Astérisque \textbf{314} (2007), x + 466 pp. (2008).
	 
	\bibitem[Ayo07b]{Ayo07b} J. Ayoub,
	\textit{Les six opérations de Grothendieck et le formalisme des cycles évanescents dans le monde motivique, II}.
	Astérisque \textbf{315} (2007), vi + 364 pp. (2008).
	
	
	\bibitem[Bei87]{Bei87} A. Beilinson,
	\textit{On the derived category of perverse sheaves}.
	In: K-theory, Arithmetic and Geometry (Moscow, 1984–1986), Lecture Notes in Math. \textbf{1289}, Springer, Berlin (1987), pp. 27--41.
	
	\bibitem[BBD82]{BBD82} A. Beilinson, J. Bernstein, P. Deligne,
	\textit{Faisceaux pervers}.
	 
	In: Analysis and topology on singular spaces, I (Luminy, 1981), Astérisque \textbf{100}, Soc. Math. France, Paris (1982), pp. 5--171.
	
	\bibitem[CD19]{CisDeg} D.-C. Cisinski, F. Déglise,
	\textit{Triangulated categories of mixed motives}. 
	Springer Monographs in Mathematics (2019), xlii + 406 pp.
	
	\bibitem[Del01]{DelVoe} P. Deligne,
	\textit{Voevodsky's lectures on cross functors}. 
	Motivic Homotopy Theory program, IAS Princeton (Fall 2001).
	Available at \url{https://www.math.ias.edu/vladimir/sites/math.ias.edu.vladimir/files/2015_transfer_from_ps_delnotes01.pdf}.
	 
	
	\bibitem[Hör17]{Hoer17} F. Hörmann,
	\textit{Six-functor-formalisms and fibered multiderivators}.
	Sel. Math. New Ser. \textbf{24} (2018), pp. 2841--2925.
	
	\bibitem[IM24]{IM24} F. Ivorra, S. Morel,
	\textit{The four operations on perverse motives}. 
	J. Eur. Math. Soc. \textbf{26}, no. 11 (2024), pp. 4191--4272.
	
	\bibitem[Mac63]{Mac63} S. Mac Lane,
	\textit{Natural associativity and commutativity}. 
	Rice Univ. Studies \textbf{49}, no. 4 (1963), pp. 28--46.
	
	\bibitem[Mac71]{Mac71} S. Mac Lane,
	\textit{Categories for the working mathematician}. 
	Graduate Texts in Math. \textbf{5}, Springer-Verlag (1971).
	
	\bibitem[MV20]{MV20} J. Moeller, C. Vasilakopoulou,
	\textit{Monoidal Grothendieck construction}. 
	Theory and Applications of Categories \textbf{35}, no. 31 (2020), pp. 1159--1207.
	
	\bibitem[Ram25]{Ram22} M. Ramzi,
	\textit{A monoidal Grothendieck construction for $\infty$-categories}. 
	Nagoya Math. Journal, published online (2025), pp. 1--32.
	
	\bibitem[Sai90]{Sai90} M. Saito,
	\textit{Mixed Hodge modules}. 
	Publ. Res. Inst. Math. Sci. \textbf{26}, no. 2 (1990), pp. 221--333.
	
	\bibitem[Shu08]{Shu08} M. Shulman,
	\textit{Framed bicategories and monoidal fibrations}.
	Theory and Applications of Categories \textbf{20}, no. 18 (2008), pp 650--738.
	
	\bibitem[Ter24a]{Ter24Fact} L. Terenzi,
	\textit{Constructing monoidal structures on fibered categories via factorizations}.
	Available at \url{https://arxiv.org/abs/2401.13489}.
	
	\bibitem[Ter24b]{Ter24Emb} L. Terenzi,
	\textit{Extending monoidal structures on fibered categories via embeddings}.
	Available at \url{https://arxiv.org/abs/2401.13517}.
	
	\bibitem[Ter24c]{Ter24Nori} L. Terenzi,
	\textit{Tensor structure on perverse Nori motives}.
	Ann. $K$-Theory \textbf{11}, no. 1 (2026), pp. 47--170.
\end{thebibliography}
\end{document}